\documentclass[aos,preprint]{imsart}
\usepackage{amsmath}
\usepackage{amsthm}
\usepackage{amssymb}
\usepackage{dsfont}
\usepackage{footmisc}
\usepackage[authoryear,round]{natbib}
\usepackage[colorlinks=true,linkcolor=blue,citecolor=blue]{hyperref}
\usepackage{hypernat}
\usepackage{verbatim}
\usepackage{textcomp}
\usepackage{enumerate}
\usepackage{graphicx}
\usepackage{bbm}
\usepackage{array}
\usepackage{multirow}

\newcolumntype{H}{>{\setbox0=\hbox\bgroup}c<{\egroup}@{}}

\newcommand{\R}{{\mathbb R}}
\newcommand{\E}{{\mathbb E}}
\renewcommand{\P}{{\mathcal P}}
\renewcommand{\Pr}{{\mathbb P}}
\newcommand{\N}{{\mathbb N}}

\newcommand{\eps}{\varepsilon}
\newcommand{\B}{\mathcal B}
\newcommand{\X}{\mathcal X}

\newcommand{\conv}{\text{conv}}
\newcommand{\cl}{\text{cl}}

\newcommand{\dtv}{d_{\text{TV}}}
\newcommand{\dH}{d_{\text{H}}}

\usepackage{xr}

\DeclareMathOperator{\Var}{Var}

\DeclareMathOperator{\diag}{diag}

\newtheorem{theorem}{Theorem}[section]
\newtheorem{proposition}[theorem]{Proposition}
\newtheorem{lemma}[theorem]{Lemma}
\newtheorem{corollary}[theorem]{Corollary}
\newtheorem{remark}[theorem]{Remark}
\newtheorem*{remark*}{Remark}

\newtheorem*{definition*}{Definition}

\numberwithin{equation}{section}

\newcounter{rcnt}[section]

\newcommand{\rem}[1]{}

\newcounter{desccount}

\newcommand{\descref}[1]{\hyperref[#1]{#1}}

\begin{document}

\begin{frontmatter}

\title{Geometrizing  rates of convergence under local differential privacy constraints\thanksref{T1}}
\runauthor{Rohde, A. and Steinberger, L.}

\runtitle{Geometrizing rates under local differential privacy}

\thankstext{T1}{Supported by the DFG Research Grant RO 3766/4-1.}
\begin{aug}
\author{\fnms{Angelika} \snm{Rohde}\ead[label=e1]{angelika.rohde@stochastik.uni-freiburg.de}}
\and
\author{\fnms{Lukas} \snm{Steinberger}\ead[label=e2]{lukas.steinberger@stochastik.uni-freiburg.de}}
\affiliation{University of Freiburg}

\address{
	Institute of Mathematics\\
	Albert-Ludwigs-Universit\"at Freiburg  \\
	Ernst-Zermelo-Stra{\ss}e 1\\
	79104 Freiburg im Breisgau\\
	\printead{e1}\\
	\printead{e2}	}
	
\end{aug}

\begin{abstract}
We study the problem of estimating a functional $\theta(\Pr)$ of an unknown probability distribution $\Pr \in\P$ in which the original iid sample $X_1,\dots, X_n$ is kept private even from the statistician via an $\alpha$-local differential privacy constraint. Let $\omega_{TV}$ denote the modulus of continuity of the functional $\theta$ over $\P$, with respect to total variation distance. For a large class of loss functions $l$ and a fixed privacy level $\alpha$, we prove that the privatized minimax risk is equivalent to $l(\omega_{TV}(n^{-1/2}))$ to within constants, under regularity conditions that are satisfied, in particular, if $\theta$ is linear and $\P$ is convex. Our results complement the theory developed by \citet{Donoho91} with the nowadays highly relevant case of privatized data. Somewhat surprisingly, the difficulty of the estimation problem in the private case is characterized by $\omega_{TV}$, whereas, it is characterized by the Hellinger modulus of continuity if the original data $X_1,\dots, X_n$ are available. We also find that for locally private estimation of linear functionals over a convex model a simple sample mean estimator, based on independently and binary privatized observations, always achieves the minimax rate. 
We further provide a general recipe for choosing the functional parameter in the optimal binary privatization mechanisms and illustrate the general theory in numerous examples. Our theory allows to quantify the price to be paid for local differential privacy in a large class of estimation problems. This price appears to be highly problem specific.
\end{abstract}



\end{frontmatter}

\section{Introduction}

One of the many new challenges for statistical inference in the information age is the increasing concern of data privacy protection. Nowadays, massive amounts of data, such as medical records, smart phone user behavior or social media activity, are routinely being collected and stored. On the other side of this trend is an increasing reluctance and discomfort of individuals to share this sometimes sensitive information with companies or state officials. Over the last few decades, the problem of constructing privacy preserving data release mechanisms has produced a vast literature, predominantly in computer science. One particularly fruitful approach to data protection that is insusceptible to privacy breaches is the concept of differential privacy \citep[see][]{Dwork06, Dinur03, Dwork04, Dwork08a, Evfim03}. In a nutshell, differential privacy is a form of randomization, where, instead of the original data, a perturbed version of the data is released, offering plausible deniability to the data providers, who can always argue that their true answer was different from the one that was actually provided. 
Aside from the academic discussion, (local) differential privacy has also found its way into real world applications. For instance, the Apple Inc. privacy statement explains the notion quite succinctly as follows.

\begin{quotation}
``It is a technique that enables Apple to learn about the user community without learning about individuals in the community. Differential privacy transforms the information shared with Apple before it ever leaves the user's device such that Apple can never reproduce the true data.''\footnote{\url{https://images.apple.com/privacy/docs/Differential_Privacy_Overview.pdf}}
\end{quotation}
The qualification of `local' differential privacy refers to a procedure which randomizes the original data already on the user's `local' machine and the original data is never released, whereas (central) differential privacy may also be employed to privatize and release an entire database that was previously compiled by a trusted curator. Here, we focus only on the local version of differential privacy.

More recently, differential privacy has also received some attention from a statistical inference perspective \citep[see, e.g.,][]{Duchi13a, Duchi13b, Duchi14, Duchi17, Wasserman10, Smith08, Smith11, Dwork10, Ye17, Awan18}. In this line of research, the focus is more on the inherent trade-off between privacy protection and efficient statistical inference and on the question what optimal privacy preserving mechanisms may look like. \citet{Duchi13a, Duchi13b, Duchi14, Duchi17} introduced new variants of the Le Cam, Fano and Assouad techniques to derive lower bounds on the privatized minimax risk. On this way, they were the first to provide minimax rates of convergence for specific estimation problems under privacy constraints in a very insightful case by case study. Here, we develop a general theory, in the spirit of \citet{Donoho91}, to characterize the differentially private minimax rate of convergence. Characterizing the minimax rate of convergence under differential privacy, and comparing it to the minimax risk in the non-private case, is one way to quantify the price, in terms of statistical accuracy, that has to be paid for privacy protection. The theory also allows us to develop (asymptotically) minimax optimal privatization schemes for a large class of estimation problems. 

To be more precise, consider $n$ individuals who possess data $X_1,\dots, X_n$, assumed to be iid from some probability distribution $\Pr\in\P$. However, the statistician does not get to see the original data $X_1,\dots, X_n$, but only a \emph{privatized} version of observations $Z$. The conditional distribution of $Z$ given $X=(X_1,\dots, X_n)$ is denoted by $Q$ and referred to as a channel distribution or a privatization scheme, i.e. $Pr(Z\in A| X=x) = Q(A|x)$. For $\alpha\in(0,\infty)$, the channel $Q$ is said to provide $\alpha$-differential privacy if
\begin{equation}\label{eq:DiffPriv}
\sup_{A} \sup_{x,x' : d_0(x,x')=1} \frac{Pr(Z\in A|X=x)}{Pr(Z\in A|X=x')} \quad\le\quad e^\alpha,
\end{equation}
where the first supremum runs over all measurable sets and $d_0(x,x'):=|\{i:x_i\ne x_i'\}|$ denotes the number of distinct entries of $x$ and $x'$. This definition is due to \citet{Dwork06} \citep[see also][]{Evfim03}. It captures the idea that the distribution of the observation $Z$ does not change too much if the data of any single individual in the database is changed, thereby protecting its privacy. The smaller $\alpha\in(0,\infty)$, the stronger is the privacy constraint \eqref{eq:DiffPriv}. More formally, (if we consider the original data $X$ as fixed) \citet[][Theorem~2.4]{Wasserman10} show that under $\alpha$-differential privacy, any level-$\gamma$ test using $Z$ to test $H_0:X=x$ versus $H_1:X=x'$ has power bounded by $\gamma e^\alpha$. As mentioned above, in this paper we focus on a special case of differential privacy, namely, local differential privacy. Somewhat informally, a channel satisfying \eqref{eq:DiffPriv} is said to provide local differential privacy, if $Z=(Z_1,\dots, Z_n)$ is a random $n$-vector, and if the $i$-th individual can generate its privatized data $Z_i$ using only its original data $X_i$ and possibly other information, but without sharing $X_i$ with anyone else. The point of this definition is that for such a protocol to be realized we do not need a trusted third party to collect or process data. It is reminiscent of the idea of randomized response \citep{Warner65}. We will see more concrete instances of such protocols below.

Suppose now that we want to estimate a real parameter $\theta(\Pr)$ based on the privatized observation vector $Z$, whose unconditional distribution is equal to $Q\Pr^{\otimes n}(dz) := \int Q(dz|x)\,\Pr^{\otimes n}(dx)$, where $\Pr^{\otimes n}$ is the $n$-fold product measure of $\Pr$. The $Q$-privatized minimax risk of estimation under a loss function $l:\R\to\R$ is therefore given by
\begin{equation}\label{eq:QMinMax}
\mathcal M_n(Q,\P, \theta) \quad:=\quad \inf_{\hat{\theta}_n} \sup_{\Pr\in\P} \E_{Q\Pr^{\otimes n}}\left[ l(|\hat{\theta}_n-\theta(\Pr)|)\right],
\end{equation}
where the infimum runs over all estimators $\hat{\theta}_n$ taking $Z$ as input data. Note that if the channel $Q$ is given by $Q(A|x) = Pr(Z\in A|X=x) = \mathds 1_A(x)$, then there is no privatization at all and the $Q$-privatized minimax risk reduces to the conventional minimax risk of estimating $\theta(\Pr)$. If we want to guarantee $\alpha$-differential privacy, then we may choose any channel $Q$ that satisfies \eqref{eq:DiffPriv} and we will try to make \eqref{eq:QMinMax} as small as possible. This leads us to the $\alpha$-private minimax risk
$$
\mathcal M_{n,\alpha}(\P, \theta) \quad:=\quad \inf_{Q\in\mathcal Q_\alpha}\mathcal M_n(Q,\P, \theta),
$$ 
where $\mathcal Q_\alpha$ is some set of $\alpha$-differentially private channels. It is this additional infimum over $\mathcal Q_\alpha$ that makes the theory of private minimax estimation deviate fundamentally from the conventional minimax estimation approach. In particular, this situation is different from the statistical inverse problem setting, because the Markov kernel $Q$ can be chosen in an optimal way and is not given a priori. A sequence of channels $Q^{(n)}\in\mathcal Q_\alpha$, for which $\mathcal M_n(Q^{(n)},\P, \theta)$ is of the order of $\mathcal M_{n,\alpha}(\P, \theta)$, is referred to as a minimax rate optimal channel and may depend on the specific estimation problem under consideration, i.e., on $\theta$ and $\P$. We write $\mathcal M_{n,\infty}(\P, \theta)$ for the classical (non-private) minimax risk.

The novel contribution of this article is to characterize the rate at which $\mathcal M_{n,\alpha}(\P, \theta)$ converges to zero as $n\to\infty$, in high generality, and to provide concrete minimax rate optimal $\alpha$-locally differentially private estimation procedures. 
To this end, we utilize the modulus of continuity of the functional $\theta:\P\to\R$ with respect to the total variation distance $\dtv(\Pr_0,\Pr_1)$, that is, 
$$
\omega_{TV}(\eps) := \sup\{|\theta(\Pr_0)-\theta(\Pr_1)| : \dtv(\Pr_0,\Pr_1)\le \eps, \Pr_0,\Pr_1\in\P \},
$$
and we show that for any fixed $\alpha\in(0,\infty)$,
\begin{equation}\label{eq:MainRes}
\mathcal M_{n,\alpha}(\P, \theta) \quad\asymp\quad l\left(\omega_{TV}\left(n^{-1/2}\right)\right).
\end{equation}
Here, $a_n\asymp b_n$ means that there exist constants $0<c_0<c_1<\infty$ and $n_0\in\N$, not depending on $n$, so that $c_0 b_n \le a_n \le c_1 b_n$, for all $n\ge n_0$.
The lower bound on $\mathcal M_{n,\alpha}(\P, \theta)$ that we establish holds in full generality, whereas, in order to obtain a matching upper bound, it is necessary to impose some regularity conditions on $\P$ and $\theta$. These will be satisfied, in particular, if $\P$ is convex and dominated and $\theta$ is linear and bounded, but also hold in some cases of non-convex and potentially non-dominated $\P$. It is important to compare \eqref{eq:MainRes} to the analogous result for the non-private minimax risk. This was established in the seminal paper by \citet{Donoho91}, who, under regularity conditions similar to those imposed here, showed that
\begin{equation}\label{eq:Donoho}
\mathcal M_{n,\infty}(\P, \theta) \quad\asymp\quad l\left(\omega_H\left(n^{-1/2}\right)\right),
\end{equation}
where $\omega_H(\eps) = \sup\{|\theta(\Pr_0)-\theta(\Pr_1)| : \dH(\Pr_0,\Pr_1)\le \eps, \Pr_0,\Pr_1\in\P \}$ and $\dH$ is the Hellinger distance.
Comparing \eqref{eq:Donoho} to \eqref{eq:MainRes}, we notice that the Hellinger modulus $\omega_H$ of $\theta$ is replaced by the total variation modulus $\omega_{TV}$. This may, and typically will, lead to different rates of convergence in private and non-private problems. Note that even in cases where we do or can not compute the moduli $\omega_{TV}$ and $\omega_H$ explicitly, we always have the a priori information that
$$
\omega_H(\eps) \le \omega_{TV}(\eps) \le \omega_H(\sqrt{2\eps}),
$$
because $\dtv\le \dH \le \sqrt{2\dtv}$ \citep[see, for instance][]{Tsybakov09}. This means that the private rate of estimation is never faster than the non-private rate and is never slower than the square root of the non-private rate. We shall see that both extremal cases can occur (see Section~\ref{SEC:EX}). We shall also see that the fastest possible private rate of convergence over a convex model $\P$ is $l(n^{-1/2})$, whereas it is $l(n^{-1})$ in the non-private case (see Lemma~\ref{lemma:bestL1Rate} in Section~\ref{sec:AppAux} of the supplement). Also note that in \eqref{eq:MainRes} we have suppressed constants that depend on $\alpha$. Our results reveal that if $\alpha$ is small, the effective sample size reduces from $n$ to $n(e^\alpha-1)^2$ when $\alpha$-differential privacy is required. That differential privacy leads to slower minimax rates of convergence was already observed by \citet{Duchi13a, Duchi13b, Duchi17}, for specific estimation problems. Here, we develop a unifying general theory to quantify the privatized minimax rates of convergence in a large class of different estimation problems, including (even irregular) parametric and non-parametric cases. This is also the first step towards a fundamental theory of adaptive estimation under differential privacy that will be pursued elsewhere. 

We also provide a general construction of $\alpha$-locally differentially private estimation procedures that is minimax rate optimal if $\P$ is convex and dominated and $\theta$ is linear and bounded. The construction relies on a functional parameter $\ell\in L_\infty$. Each individual generates $Z_i$ independently and binary distributed on $\{-z_0, z_0\}$, with 
$$
Pr(Z_i = z_0 | X_i=x_i) = \frac{1}{2}\left(1+\frac{\ell(x_i)}{z_0}\right)
$$
and $z_0 = \|\ell\|_\infty \frac{e^\alpha+1}{e^\alpha-1}$. The final estimator is then simply given by the sample mean $\bar{Z}_n = \frac{1}{n}\sum_{i=1}^n Z_i$. For appropriate $\ell = \ell_n$, this yields an $\alpha$-locally differentially private procedure that attains the minimax rate in \eqref{eq:MainRes}. The choice of functional parameter $\ell$ is problem specific but can often be guided by considering optimality of the estimator $\E[\bar{Z}_n|X_1,\dots, X_n] = \frac{1}{n}\sum_{i=1}^n \ell(X_i)$ in the problem with direct observations. We exemplify this choice in many classical moment or density estimation problems (cf. Section~\ref{SEC:EX}). We point out, however, that there are cases where a certain $\ell$ leads to rate optimal locally private estimation even though the estimator $\frac{1}{n}\sum_{i=1}^n \ell(X_i)$ is not rate optimal in the direct problem.

The paper is organized as follows. In the next section (Section~\ref{sec:Prelim}), we formally introduce the private estimation problem, several classes of locally private channel distributions and a few tools required for the analysis of the $\alpha$-private minimax risk $\mathcal M_{n,\alpha}(\P, \theta)$. Section~\ref{sec:lowerB} presents a general lower bound on $\mathcal M_{n,\alpha}(\P, \theta)$. That this lower bound is attainable, in surprisingly high generality and by the simple linear estimation procedure described above, is then established in Section~\ref{sec:Attainability}. The results of that section, however, do not offer an explicit construction of the functional parameter $\ell$. In Section~\ref{sec:Constructive} we then provide some guidance on choosing $\ell$, as well as a high-level condition for optimality of $\ell$ that we verify in all our examples. We illustrate the general theory by a number of concrete examples that are presented in Section~\ref{SEC:EX}. Most of the technical arguments are deferred to the supplementary material.


\section{Preliminaries and notation}
\label{sec:Prelim}

Let $\P$ be a set of probability measures on the measurable space $(\mathcal X,\mathcal F)$. Let $\theta : \P \to \R$ be a functional of interest. In case $\P$ is convex, we say that the functional $\theta:\P\to\R$ is linear if for $\Pr_0,\Pr_1\in\P$ and $\lambda\in[0,1]$, we have $\theta(\lambda\Pr_0+(1-\lambda)\Pr_1) = \lambda\theta(\Pr_0) + (1-\lambda)\theta(\Pr_1)$.
We are given the privatized data $Z_1,\dots, Z_n$ on the measurable space $(\mathcal Z, \B(\mathcal Z))$, $\mathcal Z= \R^q$, where $\B(\mathcal Z)$ denotes the Borel sets with respect to the usual topology. The conditional distribution of the observations $Z=(Z_1,\dots, Z_n)$ given the original sample $X=(X_1,\dots, X_n)$ is described by the \emph{channel distribution} $Q$. That is, $Q$ is a Markov probability kernel from $(\mathcal X^n, \mathcal F^{\otimes n})$ to $(\mathcal Z^n, \B(\mathcal Z^n))$. For ease of notation we suppress its dependence on $n$. Hence, if the $X_i$ are distributed iid according to $\Pr\in\P$ and $\Pr^{\otimes n}$ denotes the corresponding product measure, then the joint distribution of the observation vector $Z = (Z_1, \dots, Z_n)$ on $\mathcal Z^n$ is given by $Q \Pr^{\otimes n} $, i.e., the measure $A\mapsto \int_{\X^n}Q(A|x)d\Pr^{\otimes n}(x)$.

\subsection{Locally differentially private minimax risk}

Recall that for $\alpha\in(0,\infty)$, a channel distribution $Q$ is called \emph{$\alpha$-differentially private}, if 
\begin{equation}\label{eq:alphaPriv}
\sup_{A\in\B(\mathcal Z^n)}\sup_{\substack{x,x'\in\mathcal X^n\\ d_0(x,x')=1}} \frac{Q(A|x)}{Q(A|x')} \quad\le \quad e^\alpha,
\end{equation}
where $d_0(x,x') := |\{i:x_i\ne x_i'\}|$ is the number of distinct components of $x$ and $x'$.
Note that for this definition to make sense, the probability measures $Q(\cdot|x)$, for different $x\in\X^n$, have to be equivalent and we interpret $\frac{0}{0}$ as equal to $1$. 

Next, we introduce two specific classes of locally differentially private channels. A channel distribution $Q: \mathcal B(\mathcal Z^n)\times \mathcal X^n \to [0,1]$ is said to be $\alpha$-\emph{sequentially interactive} (or provides $\alpha$-sequentially interactive differential privacy) if the following two conditions are satisfied. First, we have for all $A\in\B(\mathcal Z^n)$ and $x_1,\dots, x_n\in\mathcal X$,
\begin{align}\label{eq:Seq}
&Q\left( A\Big|x_1,\dots, x_n\right)\notag \\
&\quad=
\int_{\mathcal Z}\dots \int_{\mathcal Z} Q_n(A_{z_{1:n-1}}|x_n,z_{1:n-1}) 
Q_{n-1}(dz_{n-1} |x_{n-1}, z_{1:n-2})\dots Q_1(dz_1|x_1),
\end{align}
where, for each $i=1,\dots, n$, $Q_i$ is a channel from $\mathcal X\times \mathcal Z^{i-1}$ to $\mathcal Z$. Here, $z_{1:n} = (z_1,\dots, z_n)^T$ and $A_{z_{1:n-1}} = \{z\in\mathcal Z : (z_1,\dots, z_{n-1},z)^T\in A\}$ is the $z_{1:n-1}$-section of $A$. Second, we require that the conditional distributions $Q_i$ satisfy
\begin{equation}\label{eq:alphaSeq}
\sup_{A\in\B(\mathcal Z)}\sup_{x_i,x_i',z_1,\dots, z_{i-1}} \frac{Q_i(A|x_i, z_1,\dots, z_{i-1})}{Q_i(A|x_i', z_1,\dots, z_{i-1})} \quad\le \quad e^\alpha \quad\quad\forall i=1,\dots, n.
\end{equation}
By the usual approximation of integrands by simple functions, it is easy to see that \eqref{eq:Seq} and \eqref{eq:alphaSeq} imply \eqref{eq:alphaPriv}. 
This notion coincides with the definition of sequentially interactive channels in \citet[][Definition~1]{Duchi17}.
We note that \eqref{eq:alphaSeq} only makes sense if for all $x_i, x_i', z_1,\dots, z_{i-1}$, the probability measure $Q_i(\cdot|x_i,z_{1:i-1})$ is absolutely continuous with respect to $Q_i(\cdot|x_i',z_{1:i-1})$.
Here, the idea is that individual $i$ can only use $X_i$ and previous $Z_j$, $j<i$, in its local privacy mechanism, thus leading to the sequential structure in the above definition. In the rest of the paper we only consider $\alpha$-sequentially interactive channels, to which we also refer simply as $\alpha$-private channels.

An important subclass of sequentially interactive channels are the so called \emph{non-interactive} channels $Q$ that are of product form
\begin{equation}\label{eq:non-Inter}
Q\left( A_1\times\dots\times A_n\Big|x_1,\dots, x_n\right) = \prod_{i=1}^n Q_i(A_i|x_i), \quad\quad \forall A_i\in\B(\mathcal Z), x_i\in\X.
\end{equation}
Clearly, a non-interactive channel $Q$ satisfies \eqref{eq:alphaPriv} if, and only if, 
$$
\sup_{A\in\B(\mathcal Z)}\sup_{x,x'\in\mathcal X}\frac{Q_i(A|x)}{Q_i(A|x')}\quad \le\quad e^\alpha \quad\quad\forall i=1,\dots, n.
$$
In that case it is also called $\alpha$-non-interactive. Both, $\alpha$-non-interactive and $\alpha$-sequentially interactive channels satisfy the $\alpha$-local differential privacy constraint as defined in the introduction. Of course, every $\alpha$-non-interactive channel is also $\alpha$-sequentially interactive.

If we measure the error of estimation by the measurable loss function $l:\R_+\to \R_+$, where $\R_+:=[0,\infty)$, the minimax risk of the above estimation problem is given by
\begin{align}\label{eq:Qminimax}
\mathcal M_n(Q,\P, \theta) \quad=\quad \inf_{\hat{\theta}_n} \sup_{\Pr\in\P} \E_{Q \Pr^{\otimes n}}\left[ l\left(|\hat{\theta}_n - \theta(\Pr)| \right)\right],
\end{align} 
where the infimum runs over all estimators $\hat{\theta}_n : \mathcal Z^n\to\R$. Finally, define the set of $\alpha$-private channels
\begin{equation}\label{eq:SIset}
\mathcal Q_\alpha := \bigcup_{q\in\N}\left\{ Q : Q \text{ is $\alpha$-sequentially interactive from $\mathcal X^n$ to $\mathcal Z^n = \R^{n\times q}$} \right\}.
\end{equation}
Therefore, the $\alpha$-private minimax risk is given by 
\begin{align}\label{eq:QminimaxPriv}
\mathcal M_{n,\alpha}(\P, \theta) \quad=\quad \inf_{Q\in\mathcal Q_\alpha} \mathcal M_n(Q,\P, \theta).
\end{align}
Note that the above infimum includes all possible dimensions $q$ of $\mathcal Z= \R^q$.

\subsection{Testing affinities and minimax identities}

Let $\P$, $\P_0$ and $\P_1$ be sets of probability measures on a measurable space $(\Omega, \mathcal A)$ and for $\Pr_0\in\P_0$, $\Pr_1\in\P_1$, define the testing affinity
\begin{align}\label{eq:pi}
\pi(\Pr_0,\Pr_1) \quad=\quad \inf_{\text{tests }\phi} \E_{\Pr_0} [\phi] + \E_{\Pr_1}[1-\phi], 
\end{align}
where the infimum runs over all (randomized) tests $\phi:\Omega \to [0,1]$. Moreover, we write
\begin{align}
\pi(\P_0,\P_1) \quad=\quad \sup_{\substack{\Pr_j\in\P_j, j=0,1}} \pi(\Pr_0,\Pr_1).
\end{align}
Throughout, we follow the usual conventions that $\sup \varnothing = -\infty$ and $\inf \varnothing = +\infty$.
If $\theta:\P\to\R$ is a functional of interest, then for $t\in\R$ and $\Delta>0$, denote $\P_{\le t} := \{\Pr\in\P : \theta(\Pr) \le t\}$ and $\P_{\ge t+\Delta} := \{ \Pr\in\P: \theta(\Pr) \ge t+\Delta\}$ and let $\P_{\le t}^{(n)}$ and $\P_{\ge t+\Delta}^{(n)}$ be the sets of $n$-fold product measures with identical marginals from $\P_{\le t}$ and $\P_{\ge t+\Delta}$, respectively.
If $Q$ is a Markov probability kernel, then we write $Q \P^{(n)}$ for the set of all probability measures of the form $Q \Pr^{\otimes n}$, where $\Pr\in\P$. Recall that a family of measures on a common probability space is dominated if there exists a $\sigma$-finite measure $\mu$ such that every element of that family is absolutely continuous with respect to $\mu$. We define the convex hull $\conv(\P)$ in the usual way to be the set of all finite convex combinations $\sum_{i=1}^m \lambda_i \Pr_i$, for $m\in\N$, $\lambda_i\ge0$, $\sum_{i=1}^m\lambda_i=1$ and $\Pr_i\in\P$. 
For $\Pr_0,\Pr_1\in\P$, we consider the Hellinger distance 
$$
\dH(\Pr_0,\Pr_1) := \sqrt{\int_{\Omega} \left(\sqrt{p_0(x)}-\sqrt{p_1(x)}\right)^2\,d\mu(x)},
$$ 
where $p_0$ and $p_1$ are densities of $\Pr_0$ and $\Pr_1$ with respect to some dominating measure $\mu$ (e.g., $\mu=\Pr_0+\Pr_1$), and the total variation distance is defined as $\dtv(\Pr_0,\Pr_1) := \sup_{A\in\mathcal A}|\Pr_0(A)-\Pr_1(A)|$.
Furthermore, for a monotone function $g:\R\to\R$, we write $g(x^-) = \lim_{y\uparrow x}g(y)$ and $g(x^+) = \lim_{y\downarrow x}g(y)$, for the left and right limits of $g$ at $x\in\R$, respectively, and we write $g(\infty^-) = \lim_{x\to\infty} g(x)$ and $g([-\infty]^+) = \lim_{x\to -\infty} g(x)$.
We also make use of the abbreviations $a\lor b = \max(a,b)$ and $a\land b = \min(a,b)$.

Next, we define the \emph{upper affinity}
\begin{equation}\label{eq:etaA}
\eta_A^{(n)}(Q,\Delta) \quad=\quad \sup_{t\in\R}\; \pi\left(\conv\left(Q \P_{\le t}^{(n)} \right), \conv\left( Q \P_{\ge t+ \Delta}^{(n)} \right) \right)
\end{equation}
and its generalized inverse for $\eta\in[0,1)$,
\begin{equation}\label{eq:DeltaA}
\Delta_A^{(n)}(Q,\eta)\quad=\quad \sup\{ \Delta\ge 0 : \eta_A^{(n)}(Q,\Delta) > \eta\}.
\end{equation}
Note that for $\eta<1$ the set in the previous display is never empty, since $\eta_A^{(n)}(Q,0) = 1$, and thus $\Delta_A^{(n)}(Q,\eta)\ge 0$. Also note that $\Delta\mapsto\eta_A^{(n)}(Q,\Delta)$ is non-increasing.

In order to show that our subsequent lower bounds on $\mathcal M_{n,\alpha}(\P, \theta)$ are attained for convex and dominated models $\P$ and linear and bounded functionals $\theta:\P\to\R$, we will need the following consequence of a fundamental minimax theorem of \citet[][Corollary~3.3]{Sion58}. See Section~\ref{sec:AppSion} of the supplementary material for the proof.

\begin{proposition}\label{prop:Sion}
Fix constants $-\infty < a\le b < \infty$. Let $\mathbb S$ be a convex set of finite signed measures on a measurable space $(\Omega, \mathcal A)$, so that $\mathbb S$ is dominated by a $\sigma$-finite measure $\mu$. 
Furthermore, let $\mathbb T = \{\phi\in L_\infty(\Omega, \mathcal A, \mu) : a \le \int_{\Omega} \phi f\,d\mu \le b, \forall f\in L_1(\Omega, \mathcal A, \mu): \|f\|_{L_1}\le1\}$. Then
$$
\sup_{\phi \in\mathbb T} \inf_{\sigma\in\mathbb S} \;\int_{\Omega} \phi \;d\sigma
\quad=\quad
\inf_{\sigma\in\mathbb S}\sup_{\phi \in\mathbb T}  \;\int_{\Omega} \phi \;d\sigma.
$$
\end{proposition}
Proposition~\ref{prop:Sion} implies that for arbitrary subsets $\P_0$ and $\P_1$ of $\P$, and if the class $Q  \P^{(n)}$ is dominated by some $\sigma$-finite measure (note that this is always the case if $Q$ is $\alpha$-private), we have the identity
\begin{align}\label{eq:KraftLeCam}
\inf_{\text{tests } \phi} \sup_{\substack{\Pr_0\in Q \P_0^{(n)}\\ \Pr_1 \in Q \P_1^{(n)}}} 
\E_{\Pr_0}[\phi]
+
\E_{\Pr_1}[1-\phi]
&=
\sup_{\substack{\Pr_0\in \conv(Q \P_0^{(n)})\\ \Pr_1 \in \conv( Q \P_1^{(n)} ) } } \inf_{\text{tests } \phi}
\E_{\Pr_0}[\phi]
+
\E_{\Pr_1}[1-\phi] \\\notag
&= \pi\left(\conv\left(Q \P_0^{(n)} \right), \conv\left( Q \P_1^{(n)} \right) \right).
\end{align}
To see this, note that the left-hand side of \eqref{eq:KraftLeCam} does not change if we replace $Q\P_r^{(n)}$ by its convex hull, for $r=0,1$, because for $\Pr_{r,i}\in Q\P_r^{(n)}$,
\begin{align*}
\E_{\sum_{i=1}^k \alpha_i \Pr_{0,i}}[\phi] + \E_{\sum_{j=1}^l\beta_j \Pr_{1,j}}[1-\phi]
&= \sum_{i,j} \alpha_i\beta_j \left(\E_{\Pr_{0,i}}[\phi] + \E_{\Pr_{1,j}}[1-\phi]\right)\\
&\le
\sup_{\substack{\Pr_0\in Q \P_0^{(n)}\\ \Pr_1 \in Q \P_1^{(n)}}} \E_{\Pr_{0}}[\phi] + \E_{\Pr_{1}}[1-\phi].
\end{align*}
Now apply Proposition~\ref{prop:Sion} with $\mathbb S = \{\Pr_0-\Pr_1:\Pr_r\in \conv(Q\P_r^{(n)}), r=0,1\}$ and $a=0$, $b=1$. 

The identity \eqref{eq:KraftLeCam} was prominently used by \citet{Donoho91} -- in the non-private case where $Q(A|x) = \mathds 1_A(x)$ -- in order to derive their lower bounds on the minimax risk. It is due to C. Kraft and L. Le Cam (Theorem~5 of \cite{Kraft55}, see also page 40 of \citet{LeCam73}), who derived it more directly. We will also make use of \eqref{eq:KraftLeCam} to derive lower bounds (see the proof of Theorem~\ref{thm:lowerDeltaA} in the supplement).
However, in order to show that there exist channel distributions $Q^{(n)}$ so that $\mathcal M_n(Q^{(n)},\P, \theta)$ attains the rate of the lower bound, we need the generality of Proposition~\ref{prop:Sion} (see Section~\ref{sec:GenUpper} below).

\section{A general lower bound on the $\alpha$-private minimax risk}
\label{sec:lowerB}

In this section we establish a lower bound on $\mathcal M_{n,\alpha}(\P, \theta) = \inf_{Q\in\mathcal Q_\alpha} \mathcal M_n(Q,\P, \theta)$, $\alpha\in(0,\infty)$, in terms of the total variation and Hellinger moduli of continuity $\omega_{TV}$ and $\omega_H$ of the functional $\theta:\P\to\R$. We also bridge the gap to the non-private case $\alpha=\infty$ in which the rate is characterized by $\omega_H$ only, and therefore, we extend results of \citet{Donoho91} to the case of privatized data. These extensions, however, do not constitute our main contribution. Therefore, we defer the technical details to Section~\ref{sec:ApplowerB} of the supplement. 
Our main conceptual innovation is to show that the lower bounds are rate optimal for a large class of possible estimation problems.

\begin{corollary}\label{corr:lowerModulus}
Fix $\eta_0,\eps_0\in(0,1)$, $\alpha\in(0,\infty)$ and let $l:\R_+\to\R_+$ be a non-decreasing loss function. Then there exists a positive finite constant $c = c(\eta_0,\eps_0)$, such that for all $\eta\in(0,\eta_0)$ and for all $n>|\log\eta|/\eps_0$,
\begin{align*}
&\mathcal M_{n,\alpha}(\P, \theta) \;=\; \inf_{Q\in\mathcal Q_\alpha} \mathcal M_n(Q,\P, \theta) \\
&\quad\quad\ge \;  l\left( \frac{1}{2} \left[ 
\omega_{TV}\left(\left[ \frac{1-\eta}{\sqrt{2n(e^\alpha-1)^2}}\right]^-\right) 
\lor
\omega_H\left(\left[ c\sqrt{\frac{|\log\eta|}{n}} \right]^- \right) \right]^-\right)\frac{\eta}{2},
\end{align*}
where $\mathcal Q_\alpha$ is the set of $\alpha$-sequentially interactive channels $Q$ as in \eqref{eq:SIset}.
\end{corollary}

Corollary~\ref{corr:lowerModulus} extends the lower bound of \citet{Donoho91} to privatized data. We point out that a similar lower bound with slightly worse constants can be easily derived from Proposition~1 of \citet{Duchi17}. In general, we have $\omega_H(\eps)\le \omega_{TV}(\eps)$, because $\dtv(\Pr_0,\Pr_1)\le \dH(\Pr_0,\Pr_1)$. Therefore, privatization leads to a larger lower bound compared to the direct case. This is hardly any surprise. Moreover, if $\alpha$ is sufficiently large, i.e., the privatization constraint is weak, then the lower bound of Corollary~\ref{corr:lowerModulus} reduces to the classical lower bound in the case of direct estimation derived by \citet{Donoho91}.

In our theory we consider the class $\mathcal Q_\alpha$ of $\alpha$-sequentially interactive channels, because those admit a reasonably simple \citep[cf.][]{Duchi17} and attainable lower bound and they comprise a relevant class of \emph{local} differential privatization mechanisms. In the next section, we show that for estimation of linear functionals $\theta$ over convex parameter spaces $\P$ (and also for more general, but sufficiently regular $\theta$ and $\P$), the rate of our lower bound is attained even within the much smaller class of non-interactive channels. So within the class of sequentially interactive channels, the non-interactive channels already lead to rate optimal private estimation of linear functionals over convex parameter spaces. 

\begin{remark}
Corollary~\ref{corr:lowerModulus} does not restrict the values of $\alpha\in(0,\infty)$ and is formulated for any sample size $n$. In particular, it continues to hold if $\alpha$ is replaced by an arbitrary sequence $\alpha_n\in(0,\infty)$. The choice of this sequence has a fundamental impact on the private minimax rate of convergence. For example, if we consider the highly privatized case where $\alpha_n\asymp n^{-1/2}$, then $n(e^{\alpha_n}-1)^2$ is bounded and the $\alpha_n$-privatized minimax risk no longer converges to zero as $n\to\infty$.
\end{remark}

\section{Attainability of lower bounds}
\label{sec:Attainability}

To establish upper bounds on the private minimax risk that match the rate of our lower bounds, some regularity conditions are needed. In the case where the channel $Q$ is non-interactive and fixed, the main ingredients for a characterization of $\mathcal M_n(Q,\P, \theta)$ are a certain minimax identity and a type of second degree homogeneity of the privatized Hellinger modulus
$$
\omega_H^{(Q_1)}(\eps) = \sup \left\{ |\theta(\Pr_0)-\theta(\Pr_1)| : \dH(Q_1 \Pr_0,Q_1 \Pr_1) \le \eps, \Pr_0,\Pr_1\in\P\right\},
$$
where here $Q_1:\mathcal B(\mathcal Z)\times \mathcal X\to[0,1]$ is a $1$-dimensional marginal channel (see the discussion in Section~\ref{sec:App:ThmUpperB} of the supplement for details). However, for the sake of readability, in the main article we only operate under the sufficient conditions that $\P$ is convex and dominated and that $\theta:\P\to\R$ is linear. Throughout this section we repeatedly make use of the following additional assumptions.
\begin{enumerate}[A)]
\item \label{cond:thetaBound} The functional $\theta:\P\to\R$ of interest is bounded, i.e., $\sup_{\Pr\in\P}|\theta(\Pr)|<\infty$. 
\item \label{cond:loss} The non-decreasing loss function $l:\R_+\to\R_+$ is such that $l(0)=0$ and $l(\frac{3}{2}t) \le a l(t)$, for some $a\in(1,\infty)$ and for every $t\in\R_+$. 
\end{enumerate}

The boundedness assumption~\ref{cond:thetaBound} is also maintained in \citet{Donoho91}. However, in their context, it is actually not necessary in some special cases such as the location model. On the other hand, the boundedness of $\theta$ appears to be much more fundamental in the case of private estimation. See, for example, Section~G in \citet{Duchi14}, who show that in the privatized location model under squared error loss, Assumption~\ref{cond:thetaBound} is necessary in order to obtain finite $\alpha$-private minimax risk $\mathcal M_{n,\alpha}(\P, \theta)$.
Assumption~\ref{cond:loss} is also taken from \citet{Donoho91}. It is satisfied for many common loss functions, such as $l_\gamma(t) = t^\gamma$, with $\gamma>0$, or the Huber loss $l_\gamma(t) = \mathds 1_{[0,\gamma)}(t) t^2/2 + \mathds 1_{[\gamma,\infty)}(t)\gamma(t-\gamma/2)$, which satisfies \ref{cond:loss} with $a=9/2$.

%

The following theorem (Theorem~\ref{THM:UPPERB}) provides sufficient conditions on the sequence of non-interactive channels $Q^{(n)}:\B(\mathcal Z^n)\times\X^n\to[0,1]$ with identical marginals $Q_1^{(n)}$, the model $\P$ and the functional $\theta$, so that the privatized minimax risk $\mathcal M_n(Q^{(n)},\P, \theta)$ is upper bounded by a constant multiple of 
$$
l\circ \omega_H^{(Q_1^{(n)})}(n^{-1/2}).
$$ 
A more general result is discussed and proved in Section~\ref{sec:App:ThmUpperB} of the supplement. This even extends, and improves, the attainability result of \citet{Donoho91} in the non-private case.
For the purpose of attainability under local differential privacy, the crucial point is the next one (cf. Theorem~\ref{THM:ATTAINABILITY} below). Namely, to establish the existence of sequences of non-interactive $\alpha$-private channels $Q^{(n)}$ that satisfy the imposed assumptions and are such that 
$$
\omega_H^{(Q_1^{(n)})}(n^{-1/2})\;\lesssim\;  \omega_{TV}(n^{-1/2}).
$$ 
At this point our theory deviates conceptually from the one developed by \citet{Donoho91}. We propose a class of $\alpha$-private channels $Q_1^{(\alpha,\ell)}$ indexed by a functional parameter $\ell\in L_\infty$ and minimize the resulting Hellinger modulus of continuity 
$$\omega_H^{(Q_1^{(\alpha,\ell)})}(n^{-1/2})$$ 
with respect to $\ell$. For this minimization to be successful we require another minimax identity to hold, which is given by the conclusion of Proposition~\ref{prop:Sion}. Combining Theorem~\ref{THM:UPPERB} and Theorem~\ref{THM:ATTAINABILITY}, which are stated below, then shows that the rate of the lower bound of the previous section can be attained.

%

\subsection{Upper bounds for given channel sequences}

An extended version of the following theorem (not assuming convexity, dominatedness and linearity), its proof and some further discussions are deferred to Section~\ref{sec:App:ThmUpperB} of the supplement.

\begin{theorem}\label{THM:UPPERB}
Fix $n\in\N$, suppose that Conditions~\ref{cond:thetaBound} and \ref{cond:loss} hold and that $Q$ is a non-interactive channel with identical marginals $Q_1$. Moreover, assume that $\P$ is dominated and convex and that $\theta:\P\to\R$ is linear. 
Fix $C \ge \sqrt{2\log 2a}+1$, $\Delta=C^2\omega_H^{(Q_1)}(n^{-1/2})$ and $C_1 = \left[ 1 + \frac{8a^2}{2a-1} \right] a^{\lceil 2\log(C)/\log(3/2)\rceil}$, where $a>1$ is the constant from Condition~\ref{cond:loss}. 
Then there exists a binary search estimator $\hat{\theta}_n^{(\Delta)}:\mathcal Z^n\to\R$ with tuning parameter $\Delta$ (cf. Proposition~\ref{prop:BinarySearch} for details), such that
\begin{align*}
\sup_{\Pr\in\P} \E_{Q \Pr^{\otimes n}}\left[ l\left(|\hat{\theta}_n^{(\Delta)} - \theta(\Pr)| \right)\right]
\;\le\;C_1\cdot l\left( \omega_H^{\left(Q_1\right)}\left( n^{-1/2}\right)\right).
\end{align*}
\end{theorem}

The general version of this theorem (see Section~\ref{sec:App:ThmUpperB} of the supplement) is a strict generalization of results of \citet{Donoho91} to cover also the case where $Q(A|x)$ is an arbitrary non-interactive channel with identical marginals and not necessarily equal to $\mathds 1_A(x)$. Concerning its proof, we introduce a binary search estimator different from the one used by \citet{Donoho91}, which, in particular, takes the privatized data as input data. Our new construction also has the advantage that it facilitates a detailed but highly non-trivial analysis in the case of the specific binary privatization scheme introduced in Section~\ref{sec:GenUpper} below. The crucial point is that in conjunction with this privatization scheme, our minimax optimal binary search estimator can even be shown to be nearly linear (see Section~\ref{sec:linEst}). The details of our construction and an in-depth analysis of the estimator when based on this specific privatization scheme is presented in Section~\ref{SEC:BINSEARCH}.

%

\subsection{A general attainability result}
\label{sec:GenUpper}

The challenge in deriving rate optimal upper bounds on the $\alpha$-private minimax risk $\mathcal M_{n,\alpha}(\P, \theta)$ is now to find $\alpha$-sequentially interactive channel distributions $Q$, such that the upper bound of the form $l\circ\omega_H^{(Q_1)}(n^{-1/2})$ on $\mathcal M_n(Q,\P, \theta)$, obtained in Theorem~\ref{THM:UPPERB}, matches the rate of the lower bound 
$$
l\left( \frac{1}{2}\, \omega_{TV}\left(\left[ \frac{1-\eta}{\sqrt{2n(e^\alpha-1)^2}}\right]^-\right)\right)
$$ 
of Corollary~\ref{corr:lowerModulus}. It turns out that non-interactive channels with identical binary marginals lead to rate optimal procedures for $\alpha$-private estimation of a large class of functionals. More precisely, we suggest to use a channel with binary marginals
\begin{align}\label{eq:binaryChannel}
Q_1^{(\alpha, \ell)}(\{\pm z_0\}| x) \;&=\; \frac{1}{2}\left(1 \pm \frac{\ell(x)}{z_0} \right),
\end{align}
where $z_0:=\|\ell\|_\infty\frac{e^\alpha+1}{e^\alpha-1}$ and where $\ell:\X\to \R$ is an appropriate measurable and bounded function. Note that
\begin{align*}
\sup_{S\in\mathcal B(\R)} \frac{Q_1^{(\alpha, \ell)}(S|x_1)}{Q_1^{(\alpha, \ell)}(S|x_2)}
=
\max\left( \frac{1+\frac{\ell(x_1)}{\|\ell\|_\infty}\frac{e^\alpha-1}{e^\alpha+1}}{1+\frac{\ell(x_2)}{\|\ell\|_\infty}\frac{e^\alpha-1}{e^\alpha+1}},
\frac{1-\frac{\ell(x_1)}{\|\ell\|_\infty}\frac{e^\alpha-1}{e^\alpha+1}}{1-\frac{\ell(x_2)}{\|\ell\|_\infty}\frac{e^\alpha-1}{e^\alpha+1}}  \right)
\le
\frac{1+\frac{e^\alpha-1}{e^\alpha+1}}{1-\frac{e^\alpha-1}{e^\alpha+1}}
= e^\alpha,
\end{align*}
so that a non-interactive channel distribution with identical marginals \eqref{eq:binaryChannel} is $\alpha$-private. 
Actually, the support $\mathcal Z= \{-z_0,z_0\}$ of $Q_1^{(\alpha,\ell)}$ has no effect on its privacy provisions. However, with this specific choice of its support, the channel 
$
Q_1^{(\alpha,\ell)}
$ 
has the property that the conditional expectation of $Z_i$ given $X_i=x$ under $Q_1^{(\alpha,\ell)}$ equals $\int_{\mathcal Z} z \,Q_1^{(\alpha,\ell)}(dz|x) = -z_0 Q_1^{(\alpha,\ell)}(\{-z_0\}|x) + z_0 Q_1^{(\alpha,\ell)}(\{z_0\}|x) =  \ell(x)$.

To motivate the choice in \eqref{eq:binaryChannel}, we make the following observation. The channel \eqref{eq:binaryChannel} has the nice feature that for $\Pr_0,\Pr_1\in\P$ with densities $p_0$ and $p_1$ with respect to $\mu= \Pr_0+\Pr_1$, we have
\begin{align}
\dtv&\left(Q_1^{(\alpha,\ell)} \Pr_0, Q_1^{(\alpha,\ell)} \Pr_1 \right) \label{eq:dTVQP}\\
&= \sup_{A\in\B(\R)} \left| \int_\X Q_1^{(\alpha,\ell)}(A|x) p_0(x)\,d\mu(x) - \int_\X Q_1^{(\alpha,\ell)}(A|x) p_1(x)\,d\mu(x)\right|\notag\\
&=
\max\left\{ \left|\int_\X \frac{1}{2}\left( 1 + \frac{\ell(x)}{z_0}\right)[p_0(x)-p_1(x)]\,d\mu(x)\right|, \right.\notag\\
&\hspace{2cm}\left. \left|\int_\X \frac{1}{2}\left( 1 - \frac{\ell(x)}{z_0}\right)[p_0(x)-p_1(x)]\,d\mu(x)\right| \right\}\notag\\
&=\left|\int_\X  \frac{\ell(x)}{2z_0}[p_0(x)-p_1(x)]\,d\mu(x)\right|
=
\frac{1}{2z_0}| \E_{\Pr_0}[\ell] - \E_{\Pr_1}[\ell]|.\notag
\end{align}
If the functional of interest is actually of the form $\theta(\Pr) = \E_\Pr[\ell]$, then we can use the fact that $\dtv\le \dH$ to see that
\begin{align*}
\omega_H^{(Q_1^{(\alpha, \ell)})}(\eps) 
&= 
\sup\{ |\theta(\Pr_0)-\theta(\Pr_1)| : \dH\left(Q_1^{(\alpha,\ell)} \Pr_0, Q_1^{(\alpha,\ell)} \Pr_1 \right)\le \eps, \Pr_0,\Pr_1\in\P\}\\
&\le
\sup\{ |\theta(\Pr_0)-\theta(\Pr_1)| : |\theta(\Pr_0)-\theta(\Pr_1)|\le 2z_0\eps, \Pr_0,\Pr_1\in\P\}\\
&\le
2\|\ell\|_\infty\frac{e^\alpha+1}{e^\alpha-1}\eps.
\end{align*}
But at least for convex $\P$ and non-constant and linear $\theta$, Lemma~\ref{lemma:bestL1Rate} in Section~\ref{sec:AppAux} of the supplement shows that $\omega_{TV}(\eps) \ge c_0\eps$, for some positive constant $c_0$ and every small $\eps>0$. Thus, 
$$
\omega_H^{(Q_1^{(\alpha, \ell)})}(n^{-1/2}) \;\le\; 
2\|\ell\|_\infty (e^\alpha+1) c_0^{-1}\cdot\omega_{TV}\left(\sqrt{\frac{1}{n(e^\alpha-1)^2}}\right),
$$ 
for all large $n$. In general, if the functional $\theta:\P\to\R$ is of a more complicated form, then we have to find a sequence $(\ell_n)$ in $L_\infty$ for which 
\begin{align}\label{eq:EllSeq}
\omega_H^{(Q_1^{(\alpha, \ell_n)})}(n^{-1/2}) \;\lesssim \; \omega_{TV}\left( \sqrt{\frac{1}{n(e^\alpha-1)^2}}\right).
\end{align}
The following result realizes the claim of the previous display using Proposition~\ref{prop:Sion}. An extended version of it is stated and proved in Section~\ref{thm:Attainability:ext} of the supplement. 

\begin{theorem}\label{THM:ATTAINABILITY}
For $\alpha\in(0,\infty)$ and $\ell\in L_\infty(\X)$, let $Q^{(\alpha,\ell)}$ be the non-interactive $\alpha$-private channel with identical marginals $Q_1^{(\alpha,\ell)}$ as in \eqref{eq:binaryChannel}. If $\P$ is convex and dominated and $\theta:\P\to\R$ is linear,
then
\begin{equation*}
\inf_{\ell:\|\ell\|_\infty\le1}\omega_H^{(Q_1^{(\alpha, \ell)})}(\eps) \;\le\; \omega_{TV}\left(\left[ \eps \frac{e^\alpha+1}{e^\alpha-1}\right]^+\right), \quad\quad \forall \eps>0.
\end{equation*}
\end{theorem}

\begin{proof}
For $s\ge0$, define 
$$
\Phi_{\ell}(s) := \sup\{\theta(\Pr_0)-\theta(\Pr_1) : \Pr_0,\Pr_1\in\P, | \E_{\Pr_0}[\ell] - \E_{\Pr_1}[\ell]| \le  s \},
$$
and note that $d_{TV}\le \dH$, $\|\ell\|_\infty\le 1$ and \eqref{eq:dTVQP}, imply
$$
\omega_H^{(Q^{(\alpha, \ell)})}(\eps) \le \Phi_{\ell}\left(2\eps\frac{e^\alpha+1}{e^\alpha-1}\right).
$$
Clearly, the function $\Phi_{\ell}$ is non-decreasing.
For $t\ge0$, define 
$\Psi_{\ell}(t) := \inf \{ s\ge0 : \Phi_{\ell}(s)>t\}$. We claim that the functions $\Phi_{\ell}$ and $\Psi_{\ell}$ have the following properties.
\begin{align}
&\Psi_{\ell}(t) >s \;\Rightarrow\; \Phi_{\ell}(s)\le t, \quad 
\sup_{\ell:\|\ell\|_\infty\le 1}\Psi_{\ell}(t) >s \;\Rightarrow\; \inf_{\ell:\|\ell\|_\infty\le 1} \Phi_{\ell}(s)\le t \label{eq:PsiPhi}\\
&\Psi_{\ell}(t) \ge \inf\{ | \E_{\Pr_0}[\ell] - \E_{\Pr_1}[\ell]| : \theta(\Pr_0)-\theta(\Pr_1)\ge t, \Pr_0,\Pr_1\in\P\} \label{eq:PsiLower}
\end{align}
The first one is obvious. To establish \eqref{eq:PsiLower}, set $A_{\ell^*}(t) := \{ s\ge0 : \Phi_{\ell}(s)>t\}$ and $B_{\ell}(t) := \{ | \E_{\Pr_0}[\ell] - \E_{\Pr_1}[\ell]| : \theta(\Pr_0)-\theta(\Pr_1)\ge t, \Pr_0,\Pr_1\in\P\}$ and note that for $A_{\ell}(t)=\varnothing$ the claim is trivial. So take $s\in A_{\ell}(t)$. Then $\Phi_{\ell}(s)> t$, which implies that there are $\Pr_0,\Pr_1\in\P$ with $| \E_{\Pr_0}[\ell] - \E_{\Pr_1}[\ell]| \le s$ and $\theta(\Pr_0)-\theta(\Pr_1)>t$. Thus, $\nu := | \E_{\Pr_0}[\ell] - \E_{\Pr_1}[\ell]| \le s$ and $\nu \in B_{\ell}(t)$. We have just shown that for every $s\in A_{\ell}(t)$ there exists a $\nu\in B_{\ell}(t)$ with $\nu\le s$. But this clearly means that $\Psi_{\ell}(t) = \inf A_{\ell}(t) \ge \inf B_{\ell}(t)$, as required. 

Now, abbreviate $\eta:=\eps\frac{e^\alpha+1}{e^\alpha-1}$, $\delta := \omega_{TV}(\eta+\xi_0)+\xi_1$, for $\xi_0,\xi_1>0$, and note that 
$\mathbb T := \{\phi\in L_\infty(\X, \mathcal F, \mu) : -1 \le \int_{\X} \phi f\,d\mu \le 1, \forall f\in L_1(\X, \mathcal F, \mu): \|f\|_{L_1}\le1\} = \{\phi\in L_\infty(\X, \mathcal F, \mu) : \|\phi\|_\infty\le1\}$. Using convexity and dominatedness of $\P$ together with linearity of $\theta$, we see that $\mathbb S_\delta:= \{\Pr_0-\Pr_1 : \theta(\Pr_0)-\theta(\Pr_1) \ge \delta, \Pr_0,\Pr_1\in\P\}$ is a dominated convex set of finite signed measures. Hence, 
\begin{equation*}
\sup_{\ell:\|\ell\|_\infty\le 1} \inf_{\sigma\in\mathbb S_\delta} \int_\X \ell\,d\sigma \;=\; \inf_{\sigma\in\mathbb S_\delta} \sup_{\ell:\|\ell\|_\infty\le 1}\int_\X \ell\,d\sigma,
\end{equation*} 
follows from Proposition~\ref{prop:Sion} with $a=-1$ and $b=1$. Therefore, \eqref{eq:PsiLower} yields
\begin{align*}
\sup_{\ell:\|\ell\|_\infty\le1} \Psi_{\ell}(\delta) 
&\ge 
\sup_{\ell:\|\ell\|_\infty\le1} \inf_{\sigma\in\mathbb S_\delta} \left|\int_{\mathcal X} \ell\,d\sigma \right|
\ge
\sup_{\ell:\|\ell\|_\infty\le1} \inf_{\sigma\in\mathbb S_\delta} \int_{\mathcal X} \ell\,d\sigma \\
&=
\inf_{\sigma\in\mathbb S_\delta} \sup_{\ell:\|\ell\|_\infty\le1} \int_{\mathcal X} \ell\,d\sigma \\
&=
2 \inf_{\sigma\in\mathbb S_\delta} \|\sigma\|_{TV} \\
&=
2 \inf\left\{d_{TV}(\Pr_0,\Pr_1) : \theta(\Pr_0)-\theta(\Pr_1)\ge\omega_{TV}\left(\eta+\xi_0\right)+\xi_1 \right\} \\
&\ge
2 \inf\left\{d_{TV}(\Pr_0,\Pr_1) : \theta(\Pr_0)-\theta(\Pr_1)>\omega_{TV}\left(\eta+\xi_0\right)\right\} \\
&\ge 2 (\eta+\xi_0) > 2\eta = 2\eps\frac{e^\alpha+1}{e^\alpha-1}.
\end{align*}
An application of \eqref{eq:PsiPhi} and letting $\xi_0\to0$ now finishes the proof.
\end{proof}

The next corollary now puts together Theorem~\ref{THM:UPPERB} and Theorem~\ref{THM:ATTAINABILITY}. Its proof is deferred to Section~\ref{sec:App:COR:ATTAINABILITY} of the supplement. A somewhat more general version of this result that relaxes convexity of the model $\P$ and linearity of the functional $\theta$ is stated and proved in Section~\ref{sec:Attainability:App} of the supplement. We want to emphasize here that the assumptions of the more general result of Section~\ref{sec:Attainability:App} in the supplement can be verified, for instance, in the non-convex case of estimating the endpoint of a uniform distribution (cf. Section~\ref{sec:uniform} in the supplement).

\begin{corollary}\label{COR:ATTAINABILITY}
Fix $\alpha\in(0,\infty)$, $n\in\N$, suppose that Assumptions~\ref{cond:thetaBound} and \ref{cond:loss} hold, that $\P$ is convex and dominated and that $\theta:\P\to\R$ is linear.
Then, 
$$
\mathcal M_{n,\alpha}(\P, \theta) := \inf_{Q\in\mathcal Q_\alpha} \mathcal M_n(Q,\P, \theta)\;\le\; C_1\cdot l\left(\omega_{TV}\left( \frac{4}{\sqrt{n}} \frac{e^\alpha+1}{e^\alpha-1} \right) \right),
$$
where $\mathcal Q_\alpha$ is the collection of $\alpha$-sequentially interactive channels as in \eqref{eq:SIset}. The constant $C_1$ is given by 
$$
C_1 = \left[1+\frac{8a^2}{2a-1} \right] a^{\lceil 2\log(C)/\log(3/2)\rceil+1},
$$ 
where $C = \sqrt{2\log(2a)}+1$ and $a>1$ is the constant from Condition~\ref{cond:loss}.
\end{corollary}

Summarizing, under the conditions of Corollary~\ref{COR:ATTAINABILITY} and invoking our results of Section~\ref{sec:lowerB} (see also Section~\ref{sec:ApplowerB} in the supplement), we obtain the characterization \eqref{eq:MainRes} announced in the introduction, i.e., for any fixed $\alpha\in(0,\infty)$,
$$
\mathcal M_{n,\alpha}(\P, \theta) \;\asymp\; l\circ \omega_{TV}\left(n^{-1/2}\right).
$$
More precisely, we even find that for all $n\in\N$ and $\alpha\in(0,\infty)$, 
\begin{align*}
\frac{1}{4}l\left( \frac{1}{2} \omega_{TV}\left( \sqrt{\frac{1}{8 n(e^\alpha-1)^2}}\right)\right) 
\;&\le\;\mathcal M_{n,\alpha}(\P, \theta) \\
\;&\le\; 
C_1 l\left( \omega_{TV}\left(\sqrt{ \frac{16}{n}\frac{(e^\alpha+1)^2}{(e^\alpha-1)^2}}\right)\right).
\end{align*}
This also shows that in the private case the effective sample size reduces from $n$ to $n\alpha^2$, for $\alpha$ small.

\begin{remark}
Although our analysis was tailored to $\alpha\le 1$ and the bounds in the previous display are tight (up to universal constants) in that regime, they are not tight for large $\alpha$, in the sense that $\frac{1}{e^\alpha-1}$ and $\frac{e^\alpha+1}{e^\alpha-1}$ are vastly different for $\alpha$ large. In view of the additive noise mechanism of \citet{Geng15}, an anonymous referee has pointed out the plausible conjecture that the correct scaling should actually be
$$
\omega_{TV}\left( \frac{1}{\sqrt{n(e^\alpha-1)^2}}\right) 
\lor \omega_{TV}\left(\frac{1}{\sqrt{ne^{2\alpha/3}}}\right)
\lor\omega_{H}\left( \frac{1}{\sqrt{n}}\right),
$$
for all $\alpha>0$. Inspired by this conjecture we improved the $\alpha$-dependence of our lower bounds for non-interactive channels to
$$
\omega_{TV}\left( \frac{1}{\sqrt{n(e^\alpha-1)^2}}\right) 
\lor \omega_{TV}\left(\frac{1}{\sqrt{ne^{\alpha}}}\right)
\lor\omega_{H}\left( \frac{1}{\sqrt{n}}\right)
$$
(see Section~\ref{sec:ApplargeAlpha} of the supplement).
The interesting question of whether one of these scalings is the correct one, or whether there exist even many more different regimes which describe the $\alpha$-dependence in large generality is left for future research.

\end{remark}

\subsection{Optimality of affine estimators}
\label{sec:linEst}

In the previous subsection we have seen that a simple non-interactive channel with binary marginals $Q_1^{(\alpha, \ell)}$ supported on $\mathcal Z=\{-z_0,z_0\}$ with $z_0=\|\ell\|_\infty\frac{e^\alpha+1}{e^\alpha-1}$ and such that $Q_1^{(\alpha, \ell)}(\{z_0\}|x) = \frac{1}{2}\left( 1 + \frac{\ell(x)}{z_0}\right)$, leads to rate optimal locally private estimation, if $\ell\in L_\infty(\X)$ is chosen appropriately. This means that the actual observations $Z_1,\dots, Z_n$ that are available for estimation of $\theta(\Pr)$, are iid according to a binary distribution on $\mathcal Z$ with probability of outcome $z_0$ equal to $[Q_1^{(\alpha,\ell)}\Pr](\{z_0\}) = \frac{1}{2}\left( 1 + \frac{\E_\Pr[\ell]}{z_0}\right)$. But therefore clearly, $\bar{Z}_n=\frac{1}{n}\sum_{i=1}^nZ_i$ is sufficient for $\Pr$. It is important to note that this only works because the chosen channel $Q_1^{(\alpha, \ell)}$ is binary. By Rao-Blackwellization
\begin{align*}
\E_{[Q_1^{(\alpha,\ell)}\Pr]^{(n)}}\left[ l\left( \left|\hat{\theta}_n(Z) - \theta(\Pr)\right|\right)\right] 
\;\ge\;
\E_{[Q_1^{(\alpha,\ell)}\Pr]^{(n)}}\left[ l\left( \left|\E\left[\hat{\theta}_n(Z)\Big|\bar{Z}_n\right] - \theta(\Pr)\right|\right)\right],
\end{align*}
we conclude that, at least for convex loss functions $l$, there must be a minimax optimal estimator that is a function of $\bar{Z}_n$ only.  
In fact, we will show more than that. Under the additional assumptions that $\P$ is convex and dominated and $\theta$ is linear, there is a choice of $\ell$ and a constant $b\in\R$ such that $\bar{Z}_n + b$ is minimax rate optimal (possibly after projection onto the range of $\theta$). 
The proof of the following result is deferred to Section~\ref{sec:App:COR:ATTAINABILITY} in the supplement. It crucially relies on Proposition~\ref{prop:BinarySearch}.\ref{lemma:BinarySearch:B} where we show that our binary search estimator in conjunction with the binary channel $Q_1^{(\alpha, \ell)}$ is approximately affine.

\begin{corollary}\label{COR:LINEST}
Fix $\alpha\in(0,\infty)$, $n\in\N$, suppose that Assumptions~\ref{cond:thetaBound} and \ref{cond:loss} hold, that $\P$ is dominated and convex and that $\theta:\P\to\R$ is linear. 
Then there exists a function $\ell^*\in L_\infty(\X)$ and a constant $b\in\R$, such that 
$$
\sup_{\Pr\in\P}\E_{[Q_1^{(\alpha,\ell^*)}\Pr]^{(n)}}\left[ l\left( \left| \Pi\left[\bar{Z}_n + b\right] - \theta(\Pr)\right| \right)\right]\;\le\; C_2\cdot l\left(\omega_{TV}\left( \frac{4}{\sqrt{n}} \frac{e^\alpha+1}{e^\alpha-1} \right) \right),
$$
where $\Pi:\R\to[M_-,M_+] := \cl\, [\theta(\P)]$ is the projection onto the closure of the range of $\theta$, which must be an interval, and $Q_1^{(\alpha,\ell^*)}$ is the binary channel of \eqref{eq:binaryChannel}. The constant $C_2$ is given by 
$$
C_2 = \left[2+a^2+\frac{8a^2}{2a-1} \right] a^{\lceil 2\log(C)/\log(3/2)\rceil+3},
$$ 
where $C = \sqrt{\frac{3}{2}}[\max\{8(e^\alpha+1), \sqrt{2\log 2a}+1\}+1]$ and $a>1$ is the constant from Condition~\ref{cond:loss}.
\end{corollary}

It is remarkable, and perhaps somewhat surprising, that a private estimation procedure as simple as the one described above, can be rate optimal in such a broad class of different estimation problems. In particular, the sample mean $\bar{Z}_n = \frac{1}{n}\sum_{i=1}^n Z_{i}$ can never achieve a faster rate than $l(n^{-1/2})$. Correspondingly, Lemma~\ref{lemma:bestL1Rate} in Section~\ref{sec:AppAux} of the supplementary material, in conjunction with the lower bound of Corollary~\ref{corr:lowerModulus}, reveals (at least for convex $\P$) that $l(n^{-1/2})$ is the best possible rate of convergence in locally differentially private estimation problems.

A prominent example of a (non-private) estimation problem with faster optimal convergence rate than $l(n^{-1/2})$ is estimation of the endpoint of a uniform distribution. In that case, the non-private optimal rate is $l(n^{-1})$, which can not be attained by a sample mean estimator. Even though the set of uniform distributions is not convex, we show in Section~\ref{sec:uniform} of the supplement, that $\bar{Z}_n = \frac{1}{n}\sum_{i=1}^n Z_{i}$ with $\ell(x) = 2x$ is rate optimal in the corresponding private estimation problem.

Detailed inspection of the proof of Corollary~\ref{COR:LINEST} reveals that the function $\ell^*\in L_\infty(\X)$ is a solution of a certain saddle-point problem.
Solving this explicitly is not always straight-forward, even though the assumptions of the corollary guarantee existence. Therefore, in the next section, we explore another direction for finding an appropriate candidate $\ell^*$ in the construction of the locally private estimation procedure of Corollary~\eqref{COR:LINEST}.

\section{Constructing rate optimal privatization mechanisms and estimators}
\label{sec:Constructive}

We have seen so far (Corollary~\ref{COR:LINEST}), that the non-interactive mechanism generating $\alpha$-private observations as
\begin{align*}
Z_i|X_i = \begin{cases}
z_0, &\text{ with probability } \frac{1}{2}\left(1 + \frac{\ell(X_i)}{z_0} \right),\\
-z_0, &\text{ with probability } \frac{1}{2}\left(1 - \frac{\ell(X_i)}{z_0} \right),
\end{cases}
\end{align*}
with $z_0 = \|\ell\|_\infty \frac{e^\alpha+1}{e^\alpha-1}$ and an appropriate choice of the function $\ell:\X\to\R$ leads to minimax rate optimality of the sample mean $\bar{Z}_n = \frac{1}{n}\sum_{i=1}^n Z_i$ for estimating $\theta(\Pr)$ (possibly after an appropriate shift and projection). Since this result of the previous Subsection~\ref{sec:linEst} was non-constructive, it remains to determine $\ell$ for practical use. 
By construction of the above privacy mechanism, we have
$$
\E[\bar{Z}_n|X_1,\dots, X_n] = \frac{1}{n} \sum_{i=1}^n \E[Z_i|X_i] =  \frac{1}{n} \sum_{i=1}^n \ell(X_i) =: \tilde{\theta}_n^{(\ell)}.
$$
This means, in particular, that the bias of $\bar{Z}_n$ is the same as that of the linear estimator $\tilde{\theta}_n^{(\ell)}$ in the estimation problem with direct observations $X_1,\dots, X_n$, with worst case absolute bias denoted by
$$
B_{\P,\theta}(\ell) := \sup_{\Pr\in\P}|\E_\Pr[\ell] - \theta(\Pr)|.
$$ 
We are thus lead to ask: What is the optimal choice of $\ell$ in the direct estimation problem using $\tilde{\theta}_n^{(\ell)}$? 

Clearly, there is no universal answer, but the solution must depend on the estimand $\theta$. For instance, if $\theta(\Pr) = p(x_0)$, where $p = d \Pr/d \lambda$ is a Lebesgue density of $\Pr$, then one might take $\ell(x) = \frac{1}{h}K\left( \frac{x-x_0}{h}\right)$ for some kernel function $K$ and an appropriate bandwidth $h=h_n>0$. Or, if $\theta(\Pr) = \E_\Pr[f]$, for some measurable $f:\X\to\R$, then $\ell = f$ is natural. Note however, that $\ell$ has to be bounded in order for the privacy mechanism to be well defined. Thus, if $f$ is unbounded, $\ell$ has to be taken as a truncated version of $f$, e.g., $\ell(x) = f(x) \mathds 1_{|f(x)|\le \frac{1}{h}}$.

Now classically we would trade-off the bias $B_{\P,\theta}(\ell)$ of the estimator $\tilde{\theta}_n^{(\ell)}$ against its variance (or standard deviation)
$$
\Var[\tilde{\theta}_n^{(\ell)}] = \frac{1}{n}\left( \E_{\Pr}[\ell^2] - \E_\Pr[\ell]^2\right).
$$
Instead, we must now trade-off the bias $B_{\P,\theta}(\ell)$ of the private estimator $\bar{Z}_n$ against its variance, which is easily seen to equal
$$
\Var[\bar{Z}_n] = \frac{1}{n}\left( z_0^2 - \E_\Pr[\ell]^2\right)
=
\frac{1}{n}\left( \|\ell\|_\infty^2\left(\frac{e^\alpha+1}{e^\alpha-1}\right)^2 - \E_\Pr[\ell]^2\right).
$$
This trade-off is still hard to do over general $\ell\in L_\infty(\X)$. Therefore, guided by our examples above, we here restrict to parametric classes $\{\ell_h : h\in\R^k\}$. As mentioned before, the choice of the class is problem specific, but in what follows we isolate a high-level sufficient condition (Condition~\ref{cond:ellh} below) for the class $\{\ell_h : h\in\R^k\}$ that allows for solving the bias-variance trade-off such that the solution $\ell_{h^*}$ leads to rate optimal private estimation. Furthermore, we then show that Condition~\ref{cond:ellh} can be checked in many classical examples. In particular, Condition~\ref{cond:ellh} also holds in cases where the model $\P$ is not convex (cf. Section~\ref{sec:uniform} in the supplement).

In the case $k=1$, $h\in\R$, the mentioned regularity condition~\ref{cond:ellh} below states that the collection of measurable functions $\ell_h : \mathcal X\to\R$, $h>0$, satisfies $\|\ell_h\|_\infty\lesssim h^{-s}$, for some $s\ge 0$, and is such that the worst case absolute bias 
$$
B_{\P,\theta}(\ell_h) := \sup_{\Pr\in\P}|\E_\Pr[\ell_h] - \theta(\Pr)|
$$ 
of the private estimator $\bar{Z}_n$ is bounded by an expression of the order $h^t$, as $h\to0$, for some $t>0$.
We then show that for the choice of tuning parameter
$$
h=h_n= \left( \frac{1}{\sqrt{n}}\frac{e^\alpha+1}{e^\alpha-1}\right)^{\frac{1}{s+t}},
$$ 
the above privatization and estimation protocol based on $\ell_{h_n}$ is $\alpha$-private minimax rate optimal if $\eps^r \lesssim \omega_{TV}(\eps)$ for $r=t/(s+t)$. This consideration misleadingly suggests that the estimator $\frac{1}{n}\sum_{i=1}^n \ell_h(X_i)$ is minimax optimal in the non-private case for a possibly different  choice of $h = \tilde{h}_n$. Although this appears to be correct in Examples~\ref{SEC:EXMomentFuncts}, \ref{SEC:EXPointFuncts} and \ref{SEC:EXMultDens} (see Section~\ref{SEC:EX} below and the supplementary material for details), it is not true in general (see Example~\ref{sec:uniform} where the minimax rate optimal estimator in the direct problem is not even of linear form).

For some estimation problems, such as estimating a multivariate anisotropic density at a point (cf. Section~\ref{SEC:EXMultDens}), the case $k=1$ is not sufficient and we need the full flexibility of Condition~\ref{cond:ellh}.

\begin{enumerate}[A)]
\stepcounter{enumi}
\stepcounter{enumi}
\item \label{cond:ellh}Suppose that $\P$ and $\theta$ are such that there exists $k\in\N$, $t\in(0,\infty)^k$, $s\in[0,\infty)^k$, $D_0\in(0,\infty)$ and ${h}_0\in(0,1]$ and a class of measurable functions $\ell_h : \mathcal X\to \R$ indexed by $h\in\R^k$, such that for all $h\in(0,{h}_0]^k$,
\begin{align}\label{eq:BiasCondition}
&\|\ell_h\|_\infty \quad\le \quad D_0 \prod_{j=1}^k h_j^{-s_j} \quad\quad\text{and} \quad\quad
B_{\P,\theta}(\ell_h) \quad\le\quad D_0 \frac{1}{k} \sum_{j=1}^k h_j^{t_j}.
\end{align}
\end{enumerate}

\begin{remark}
Note that Condition~\ref{cond:ellh} implies Condition~\ref{cond:thetaBound}, because the converse of \ref{cond:thetaBound} implies that $B_{\P,\theta}(\ell)$ is infinite whenever $\ell$ is bounded.
\end{remark}

The proof of the following theorem is deferred to Section~\ref{sec:App:THM:CONSTRUCTIVE} in the supplement.

\begin{theorem}\label{THM:CONSTRUCTIVE}
Suppose that Conditions~\ref{cond:loss} and \ref{cond:ellh} hold and set $\bar{r} = \sum_{j=1}^k \frac{s_j}{t_j}$.
For $\alpha\in(0,\infty)$, let $Q^{(\alpha,\ell)}$ be the $\alpha$-private channel with identical marginals \eqref{eq:binaryChannel} and set $h_n = (h_{n,1},\dots, h_{n,k})^T$ and
$$h_{n,j} = \left( \frac{1}{\sqrt{n}} \frac{e^\alpha+1}{e^\alpha-1}\right)^{\frac{1}{t_j(1+\bar{r})}}.$$ 
Then the arithmetic mean $\bar{Z}_n(z) := \frac{1}{n}\sum_{i=1}^n z_i$, $z=(z_1,\dots, z_n)'\in\R^n$, satisfies
\begin{align}\label{eq:MinimaxEst}
\sup_{\Pr\in\P} \E_{Q^{(\alpha,\ell_{h_n})}\Pr^{\otimes n}}\left[ l(|\bar{Z}_n-\theta(\Pr)|)\right]
\;\le\;
C_0\cdot l\left( \left(\frac{1}{\sqrt{n}}\frac{e^\alpha+1}{e^\alpha-1}\right)^{\frac{1}{1+\bar{r}}} \right),
\end{align}
for all $n\in\N$ and a positive finite constant $C_0$ that depends only on $a$ and $D_0$.
\end{theorem}

The private estimator $\bar{Z}_n$ of Theorem~\ref{THM:CONSTRUCTIVE} is $\alpha$-private minimax rate optimal if the derived upper bound \eqref{eq:MinimaxEst} on the worst case risk is of the same order as the lower bound of Corollary~\ref{corr:lowerModulus}. The latter is true if $ \eps^{\frac{1}{1+\bar{r}}} \lesssim \omega_{TV}(\eps)$, for all small $\eps>0$. That this is often satisfied simultaneously with the conditions of Theorem~\ref{THM:CONSTRUCTIVE} is demonstrated in the following example section~\ref{SEC:EX}. 
Lemma~\ref{lemma:CondC} in the supplement shows that under Condition~\ref{cond:ellh} the corresponding upper bound $\omega_{TV}(\eps) \lesssim \eps^{\frac{1}{1+\bar{r}}}$ holds.

\section{Examples}
\label{SEC:EX}

In this section, we discuss several concrete estimation problems for which we derive bounds on the total variation modulus 
$$
\omega_{TV}(\eps) \;=\; \sup\{ |\theta(\Pr_0)- \theta(\Pr_1)| : \dtv(\Pr_0,\Pr_1)\le \eps, \Pr_j\in\P\}
$$
which characterizes the rate of local differentially private estimation, and compare them to the Hellinger modulus $\omega_{H}(\eps)$ that determines the estimation rate under direct observations (see Table~\ref{table:Moduli} below). Furthermore, we exhibit families of functions $\ell_h$ which, in conjunction with the binary construction in \eqref{eq:binaryChannel} and an appropriate choice of tuning parameters, lead to minimax rate optimal locally private estimation procedures (see Table~\ref{table:ellFun}). 

Even in cases where the moduli of continuity are hard to evaluate explicitly, the following relationship is always true,
$$
\omega_H(\eps) \le \omega_{TV}(\eps) \le \omega_H(\sqrt{2\eps})\quad\forall \eps>0,
$$
because $\dtv \le \dH \le \sqrt{2\dtv}$ \citep[cf. for instance][Lemma~2.3]{Tsybakov09}. This shows that in the worst case, the private minimax rate of estimation is the square root of the non-private minimax rate, whereas the private rate can never be better than the non-private one. Both extremal cases can occur; see examples below.

The details and proofs of all claims made in Tables~\ref{table:Moduli} and \ref{table:ellFun} are deferred to Section~\ref{sec:AppExamples} of the supplementary material, except for the well known facts about the Hellinger modulus. Our list of examples is far from being exhaustive, but due to space constraints we present only a few classical cases for which the non-private rates are well known. The first column in each of the two following tables describes the statistical model $\P$, that is, the set of (marginal) data generating distributions, and the second column displays the functional $\theta:\P\to\R$ that is to be estimated. In the first row, we consider moment estimation. Here, $|f|(\X) := \{y\in\R : \exists x\in\X: |f(x)|=y\}$ denotes the range of $x\mapsto |f(x)|$. In the second row, we consider estimation of the $m$-th derivative of a density at a fixed point $x_0\in\X=\R$ over the class $\mathcal H_{\beta,L}^{\ll\lambda}(\R)$ of Lebesgue densities on $\R$ that are H\"{o}lder continuous with exponent $\beta>m$. That is, $p\in\mathcal H_{\beta,L}^{\ll\lambda}(\R)$ is $b := \lfloor \beta \rfloor$ times differentiable with $b$-th derivative $p^{(b)}$ satisfying
$$
|p^{(b)}(x) - p^{(b)}(y)| \le L |x-y|^{\beta-b},\quad\quad \forall x,y\in\R.
$$
In the third row of our tables, we consider density estimation at a point $x_0\in\X=\R^d$ over the anisotropic  class $\mathcal H_{\beta,L}^{\ll\lambda}(\R^d)$ of Lebesgue densities on $\R^d$ such that for every $j\in\{1,\dots, d\}$ and every $x,x'\in\R^d$,
$$
|p(x_1,\dots, x_{j-1},x_j', x_{j+1},\dots, x_d) - p(x)| \;\le \; L_j|x_j'-x_j|^{\beta_j}.
$$
Finally, in the last row we consider estimating the endpoint of a uniform distribution. The representations of the moduli of continuity in the last two columns of Table~\ref{table:Moduli} are to be understood as upper and lower bounds up to constants and for small values of $\eps>0$.

We see that already for moment estimation, both extreme cases mentioned above can occur. If $f$ is bounded, then $\theta(\Pr) = \E_\Pr[f]$ can be estimated at $n^{-1/2}$ rate in both the locally private as well as in the direct case. If, however, the range of $|f|$ contains the whole positive real line, and we have no more than a second moment being bounded (i.e., $\kappa\in(1,2]$), then the locally private rate is the square root of the rate under direct observations.\footnote{Note that this private minimax rate of convergence was already discovered by \citet{Duchi17} but with a rate optimal channel sequence and estimator different from ours.}
The density estimation problems are intermediate cases. There is some price to be paid for local differential privacy in terms of convergence rate, but it is not as bad as the square root of the direct rate. Finally, estimating the endpoint of a uniform distribution is another instance of a worst case situation under local differential privacy.

\begin{table}[h]
\begin{center}
\bgroup
\def\arraystretch{2}
\begin{tabular}{c|c|c|c|c}
$\P$ 		& 	$\theta: \P\to\R$	&		& $\omega_{H}(\eps)$		& 	$\omega_{TV}(\eps)$ \\ \hline\hline 
$\{\Pr:\E_\Pr[|f|^\kappa]\le L\}$ &  \multirow{2}{*}{$\Pr\mapsto \E_P[f]$} & $\|f\|_\infty<\infty$	&  $\eps$ 	& $\eps$	\\ \cline{3-5}
$L>0$, $\kappa>1$	& 	& $|f|(\X)\supseteq (0,\infty)$	&  $\eps^{(2\frac{\kappa-1}{\kappa})\land1}$	& $\eps^{\frac{\kappa-1}{\kappa}}$	\\ \hline
$\mathcal H_{\beta,L}^{\ll\lambda}(\R)$ 	&  \multirow{2}{*}{$\Pr\mapsto p^{(m)}(x_0)$} 	& 	&  \multirow{2}{*}{$\eps^{\frac{\beta-m}{\beta+1/2}}$}	&	\multirow{2}{*}{$\eps^{\frac{\beta-m}{\beta+1}}$}	\\
$L>0$, $\beta>0$ 		&    	& 	&  	&		\\ \hline
$\mathcal H_{\beta,L}^{\ll\lambda}(\R^d)$ 	&  \multirow{2}{*}{$\Pr\mapsto p(x_0)$} 	& 	\multirow{2}{*}{$\bar{r}=\sum_{j=1}^d\frac{1}{\beta_j}$} &  \multirow{2}{*}{$\eps^{\frac{1}{1+\bar{r}/2}}$}	&	\multirow{2}{*}{$\eps^{\frac{1}{1+\bar{r}}}$}	\\
$L\in\R_+^d$, $\beta\in(0,1]^d$ 		&    	& 	&  	&		\\ \hline
$\text{Unif}[0,\vartheta]  $ &  \multirow{2}{*}{$\Pr\mapsto \vartheta$} 	& 	&  \multirow{2}{*}{$\eps^2$} & \multirow{2}{*}{$\eps$}	\\
$\vartheta\in(0,M]$ 		&    	& 	&  	&		\\ \hline
%
\end{tabular}%
\egroup
\end{center}
\caption{Comparison of Hellinger (non-private) and total variation (private) moduli of continuity for several estimation problems. The minimax rate of convergence (for fixed $\alpha$) in each problem is given by $l\circ\omega(n^{-1/2})$, where $l$ is the loss function.}
\label{table:Moduli}
\end{table}

\begin{table}[h]
\begin{center}
\bgroup
\def\arraystretch{2}
\begin{tabular}{c|c|c|c}
$\P$ 		& 	$\theta: \P\to\R$	&		& $\ell_h(x) $		\\ \hline\hline 
$\{\Pr:\E_\Pr[|f|^\kappa]\le L\}$ &  \multirow{2}{*}{$\Pr\mapsto \E_\Pr[f]$} & $\|f\|_\infty<\infty$	&  $f(x)$		\\ \cline{3-4}
$L>0$, $\kappa>1$	& 	& $|f|(\X) \supseteq (0,\infty)$	&  $f(x)\mathds 1_{|f(x)|\le \frac{1}{h}}$	\\ \hline
$\mathcal H_{\beta,L}^{\ll\lambda}(\R)$ 	&  \multirow{2}{*}{$\Pr\mapsto p^{(m)}(x_0)$} 	& 	&  \multirow{2}{*}{$\mathds 1_{[-1,1]}\left(\frac{x-x_0}{h}\right)\cdot\frac{d^m}{dx^m} \frac{1}{h}K\left( \frac{x-x_0}{h}\right)$}	\\
$L>0$, $\beta>0$ 		&    	& 	&  		\\ \hline
$\mathcal H_{\beta,L}^{\ll\lambda}(\R^d)$ 	&  \multirow{2}{*}{$\Pr\mapsto p(x_0)$} 	& 	\multirow{2}{*}{$\bar{r}=\sum_{j=1}^d\frac{1}{\beta_j}$} &  \multirow{2}{*}{$\prod_{j=1}^d\frac{1}{h_j}K\left(\frac{x_j-x_{0,j}}{h_j} \right)$}\\
$L\in\R_+^d$, $\beta\in(0,1]^d$ 		&    	& 	& 		\\ \hline
$\text{Unif}[0,\vartheta]  $ &  \multirow{2}{*}{$\Pr\mapsto \vartheta$} 	& 	&  \multirow{2}{*}{$2x$} \\
$\vartheta\in(0,M]$ 		&    	& 	& 		\\ \hline
%
\end{tabular}%
\egroup
\end{center}
\caption{Examples for the choice of class $\{\ell_h:h\in\R^k\}$ leading to rate optimal estimation in all the problems of Table~\ref{table:Moduli}.}
\label{table:ellFun}
\end{table}


\section{Studying the binary search estimator}
\label{SEC:BINSEARCH}
Within this section, let $M_-:=\inf_{\Pr\in\P}\theta(\Pr)$, $M_+ :=\sup_{\Pr\in\P}\theta(\Pr)$, $M := M_+-M_-$ and $\mathbb S_\Delta := \{\Pr_0-\Pr_1: \theta(\Pr_0)-\theta(\Pr_1)\ge \Delta, \Pr_0,\Pr_1\in\P\}$, for $\Delta\ge 0$.

\begin{proposition}\label{prop:BinarySearch}
Fix a finite constant $\Delta>0$ and suppose that $-\infty<M_- < M_+<\infty$. Let $Q: \mathcal B(\mathcal Z^n)\times \mathcal X^n\to [0,1]$ be a non-interactive channel distribution with identical marginals $Q_1$. Moreover, let $N= N(M,\Delta)$ be the smallest integer such that $N\Delta>M>0$. For $l\in\N_0$, set $\eta_l = (l+1)\Delta$. 
\begin{enumerate}
	\setlength\leftmargin{-20pt}
	\renewcommand{\theenumi}{(\roman{enumi})}
	\renewcommand{\labelenumi}{{\theenumi}} 
\item \label{lemma:BinarySearch:A}
If $Q_1 \P$ is dominated (by a $\sigma$-finite measure), then there exists an estimator $\hat{\theta}_n^{(\Delta)} : \mathcal Z^n\to \R$ with tuning parameter $\Delta$, such that for every $l\in\N_0$,
$$
\sup_{\Pr\in\P} Q \Pr^{\otimes n}\left(z\in\mathcal Z^n: \left| \hat{\theta}_n^{(\Delta)}(z) - \theta(\Pr)\right| > \eta_l\right) \;\le\;  4\sum_{k=l+1}^{N-2}\left[\eta_A^{(n)}(Q, k\Delta)\lor 0\right],
$$
and an empty sum is interpreted as equal to zero. Moreover, $\hat{\theta}_n^{(\Delta)}$ takes values in the set $\{M_- +j\Delta:j\in\{1,\dots, N-1\}\}$ if $N\ge 3$, and $\hat{\theta}_n^{(\Delta)}\equiv(M_-+M_+)/2$ else. We set $\hat{\theta}_n^{(0)}(z) = (M_- +M_+)/2$.
\item \label{lemma:BinarySearch:B}
Fix $\alpha\in(0,\infty)$. Suppose that $\P$ is convex and $\theta:\P\to\R$ is linear and there exists $\ell^*\in L_\infty(\X)$ such that $\|\ell^*\|_\infty\le 1$ and 
$$
\inf_{\sigma\in\mathbb S_\Delta} \int_\X \ell^*\,d\sigma >0.
$$
If $Q_1=Q_1^{(\alpha, \ell^*)}$ is the binary channel of \eqref{eq:binaryChannel} with $\mathcal Z = \{-z_0,z_0\}$, then there exists an affine function $g^{\text{(aff)}}:\R\to\R$ such that 
$$\left|\Pi_{M_-,M_+}[g^{\text{(aff)}}(\bar{z}_n)]-\hat{\theta}_n^{(\Delta)}(z)\right| \le 2\Delta,
\quad\quad\forall z\in\mathcal Z^n,
$$ 
for $\hat{\theta}_n^{(\Delta)}$ as in part~\ref{lemma:BinarySearch:A} and where $\Pi_{M_-,M_+}:\R\to[M_-,M_+]$ is the projection onto $[M_-,M_+]$.

\end{enumerate}
\end{proposition}
\begin{proof}
Without loss of generality, we may assume that $0=\inf_{\Pr\in\P}\theta(\Pr)<\sup_{\Pr\in\P}\theta(\Pr) = M$, by estimating $\theta(\Pr) - M_-$ instead of $\theta(\Pr)$ and, in case of part~\ref{lemma:BinarySearch:B}, by noting that $\Pi_{M_-,M_+}[x] - M_- = \Pi_{0,M}[x-M_-]$.

To rigorously introduce the binary search estimator, consider first the case where $\Delta>0$ is such that $N=N(\Delta,M)\le 2$. In that case, we set $\hat{\theta}_n^{(\Delta)} \equiv M/2$, which satisfies the desired inequality trivially, because in this case $\Delta>M/2$ which implies $|\hat{\theta}_n^{(\Delta)}(z) -\theta(\Pr)| = |M/2-\theta(\Pr)| \le M/2 < \Delta = \eta_0\le \eta_l$. If $N\ge 3$, we shall construct the estimator $\hat{\theta}_n^{(\Delta)}$ such that it takes values in the set $\{j\Delta:j=1,\dots, N-1\}$. 

\begin{figure}[htbp]
\includegraphics[width=10cm]{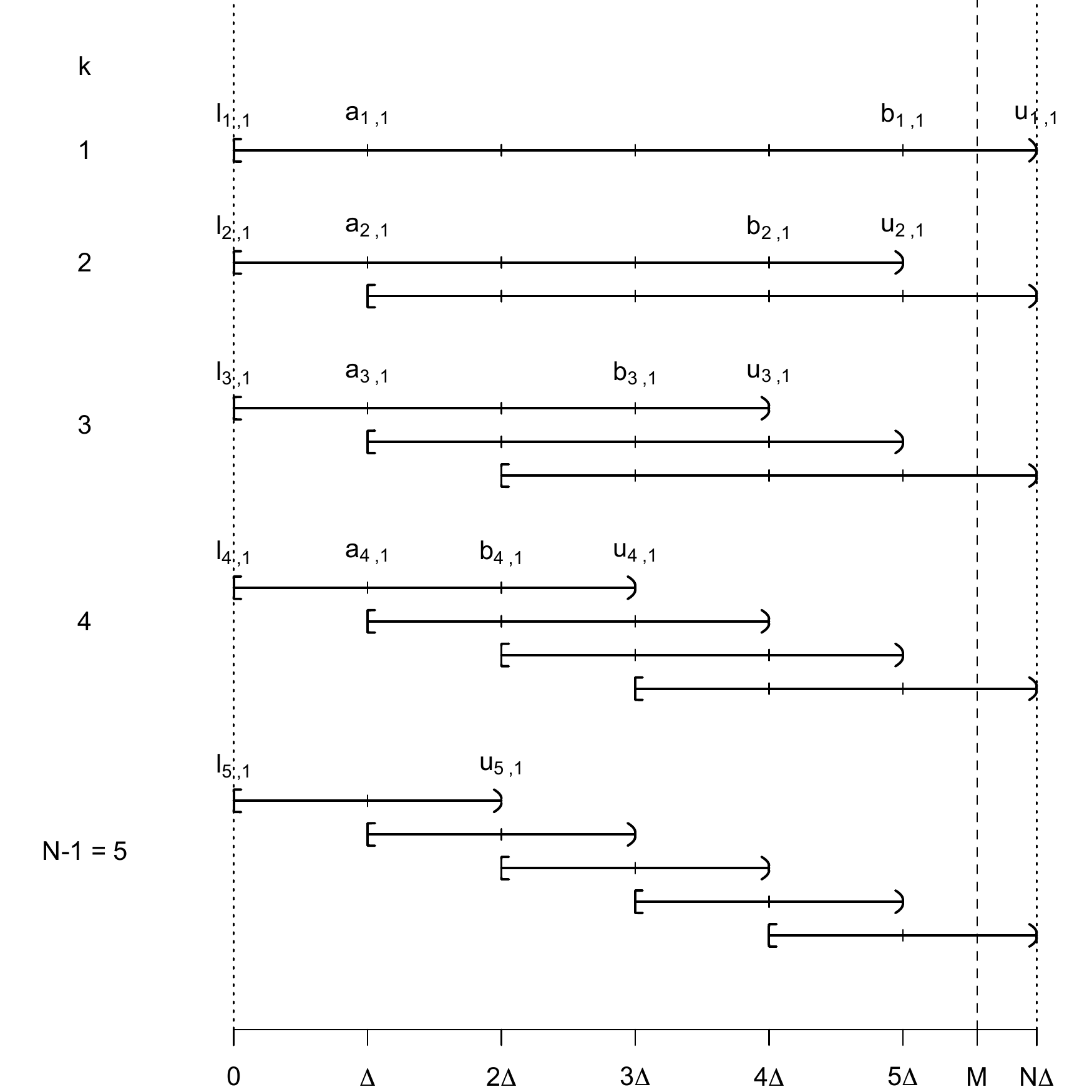} 
\caption{An example of the interval construction for the binary search estimator.}
\label{fig:Intervals}
\end{figure}

\begin{figure}[htbp]
\includegraphics[width=8cm]{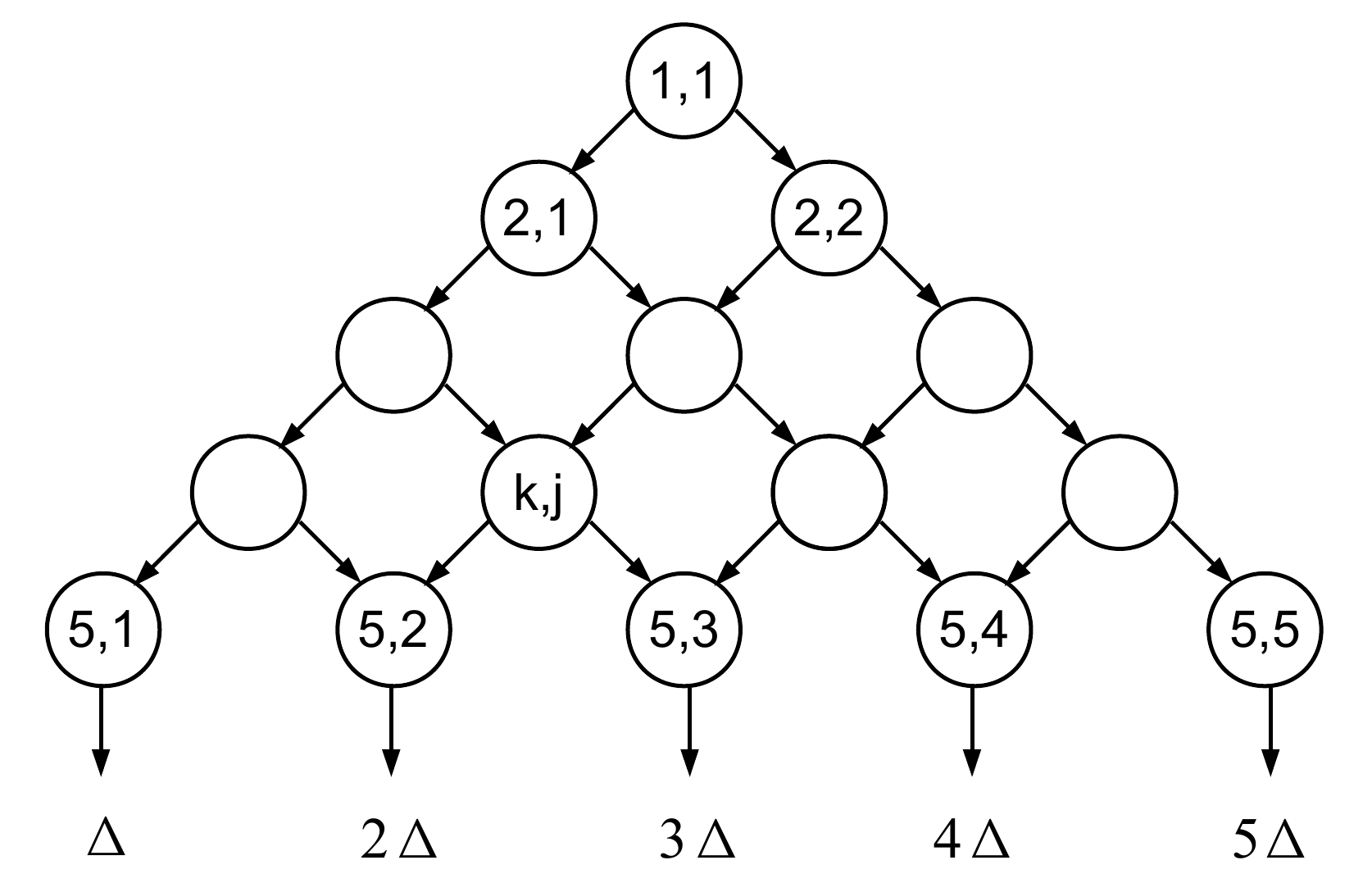} 
\caption{A graphical representation of the binary search estimator}
\label{fig:Graph}
\end{figure}

The construction is as follows: To select one of the values in the set $\{j\Delta:j=1,\dots, N-1\}$ we first introduce a scheme to partition $[0,M)$ (cf. Figure~\ref{fig:Intervals}). Start with the interval $[l_{1,1},u_{1,1})=[0, N\Delta)$, which contains $[0,M)$ by definition of $N$, and remove either the leftmost or the rightmost subinterval of length $\Delta$, i.e., $[0,\Delta)$ or $[(N-1)\Delta,N\Delta)$, to produce two new intervals $[l_{2,1}, u_{2,1}) = [0,(N-1)\Delta)$ and $[l_{2,2}, u_{2,2}) = [\Delta,N\Delta)$, each of length $(N-1)\Delta$. Then proceed in the same way again to produce three (note that removing the leftmost subinterval of length $\Delta$ in the first step and then removing the rightmost in the second step results in the same interval as if we had removed them in the opposite order) new intervals $[l_{3,1},u_{3,1})$, $[l_{3,2},u_{3,2})$, $[l_{3,3},u_{3,3})$, each of length $(N-2)\Delta$. Continue this process for $N-2$ steps to arrive at the intervals $[l_{N-1,j}, u_{N-1,j})$, $j=1,\dots, N-1$, of length $2\Delta$ whose midpoints are exactly the values $j \Delta$. 

Formally, for $k\in\{1,\dots, N-1\}$ and $j\in\{1,\dots, k\}$, we set $l_{k,j} = (j-1)\Delta$, $u_{k,j} = l_{k,j} + (N-k+1)\Delta$, and we also define $a_{k,j} = l_{k,j}+\Delta$ and $b_{k,j} = u_{k,j} - \Delta$, so that $b_{k,j}-a_{k,j}=(N-k-1)\Delta =: d_k$. With each pair $(k,j)$ as before, we associate a (randomized) minimax test $\xi_{k,j}:\mathcal Z^n\to[0,1]$ for $H_0:Q_1 \P_{\le a_{k,j}} = \{Q_1\Pr: \theta(\Pr)\le a_{k,j}, \Pr\in\P\}$ against $H_1:Q_1 \P_{\ge b_{k,j}}= \{Q_1\Pr: \theta(\Pr)\ge b_{k,j}, \Pr\in\P\}$. Recall that such a minimax test has the property that
\begin{align*}
\sup_{\substack{\Pr_0\in [Q_1 \P_{\le a_{k,j}}]^{(n)} \\ \Pr_1\in [Q_1 \P_{\ge b_{k,j}}]^{(n)}}} \E_{\Pr_0}(\xi_{k,j}) + \E_{\Pr_1}(1-\xi_{k,j})
&=
\inf_{\text{tests }\phi}\sup_{\substack{\Pr_0\in [Q_1 \P_{\le a_{k,j}}]^{(n)} \\ \Pr_1\in [Q_1 \P_{\ge b_{k,j}}]^{(n)}}} \E_{\Pr_0}(\phi) + \E_{\Pr_1}(1-\phi).
\end{align*}
Existence is well known (see Lemma~\ref{lemma:minimaxTest} in Section~\ref{sec:AppAux} of the supplement, which is a minor modification of a result by \citet{Krafft67}, see also \citet[][Problem~8.1 and Theorem~A.5.1]{Lehmann05}). To obtain a non-randomized test from $\xi_{k,j}$, we set $\xi_{k,j}^* = \mathds 1_{(1/2,1]}(\xi_{k,j})$. Since $\E_\Pr[\xi_{k,j}^*] = \Pr(\xi_{k,j}>1/2) \le 2\E_\Pr[\xi_{k,j}]$ and $\E_\Pr[1-\xi_{k,j}^*] = \Pr(\xi_{k,j}\le 1/2) = \Pr(1-\xi_{k,j}\ge 1/2) \le 2\E_\Pr[1-\xi_{k,j}]$, we get in view of \eqref{eq:KraftLeCam}, that
\begin{align}\label{eq:proof:lemma:BinarySearch}
\sup_{\substack{\Pr_0\in [Q_1 \P_{\le a_{k,j}}]^{(n)} \\ \Pr_1\in [Q_1 \P_{\ge b_{k,j}}]^{(n)}}} \E_{\Pr_0}(\xi_{k,j}^*) + \E_{\Pr_1}(1-\xi_{k,j}^*)
&\le
2\pi\left( \conv\left(Q\P_{\le a_{k,j}}^{(n)}\right), \conv\left(Q \P_{\ge b_{k,j}}^{(n)}\right) \right)\notag\\
&\le
2\eta_A^{(n)}(Q,b_{k,j}-a_{k,j})\\
&=
2\eta_A^{(n)}(Q,d_k).\notag
\end{align}

By definition of $M_-$ and $M_+$, all the sets $\P_{\le a_{k,j}}$ and $\P_{\ge b_{k,j}}$ are non-empty, with the only exception of $\P_{\ge b_{1,1}}$, which is empty if, and only if, $(N-1)\Delta=M$ and $\theta$ does not attain its supremum $M = \sup_{\Pr\in\P}\theta(P)$. In that case, we take $\xi_{1,1} \equiv 0$.

To define the value of the binary search estimator $\hat{\theta}_n^{(\Delta)}(z)$ for a given observation $z\in\mathcal Z^n$, we perform a stepwise testing procedure (cf. Figure~\ref{fig:Graph}). Starting at the full interval $[0,N\Delta)$ ($k=1$), we remove the outer-most subinterval of length $\Delta$ that was rejected by the test $\xi_{1,1}^*$. Depending on the outcome of $\xi_{1,1}^*$ and the corresponding sub-interval of length $(N-1)\Delta$ we are left with ($k=2$), we either perform the test $\xi_{2,1}^*$ or $\xi_{2,2}^*$ and, again, remove the rejected interval of length $\Delta$. We proceed in this way until $k=N-2$. Finally, we set $\hat{\theta}_n^{(\Delta)}(z)$ equal to the midpoint of the remaining interval of length $2\Delta$ that was selected by the test performed at level $k=N-2$. This procedure leads to the formal definition: Set $j_1(z)=1$ and for $k\in\{2,\dots, N-1\}$, set $j_k(z) = j_{k-1}(z) + \xi_{k-1,j_{k-1}(z)}^*(z) \in \{1,\dots, k\}$, i.e., $j_k(z)$ is the index of the test to be performed on level $k\le N-2$ and $j_{N-1}(z)\Delta$ is the value of the estimator $\hat{\theta}_n^{(\Delta)}(z) = (u_{N-1,j_{N-1}(z)}+l_{N-1,j_{N-1}(z)})/2 = j_{N-1}(z)\Delta$.

We now analyze the estimation error of $\hat{\theta}_n^{(\Delta)}$. Fix $\Pr\in\P$ and $z\in\mathcal Z^n$. We say that the test $\xi_{k,j_k(z)}^*(z)$ decided incorrectly, if its decision lead to the removal of a length-$\Delta$ subinterval that actually contained $\theta(\Pr)$. Formally, $\xi_{k,j_k(z)}^*(z)$ decided incorrectly if $\xi_{k,j_k(z)}^*(z)=0$ and $\theta(\Pr)\in[b_{k,j_k(z)},u_{k,j_k(z)})$, or $\xi_{k,j_k(z)}^*(z)=1$ and $\theta(\Pr)\in[l_{k,j_k(z)},a_{k,j_k(z)})$. Note that the test $\xi_{k,j_k(z)}^*(z)$ can not decide incorrectly if $\theta(\Pr) \notin [l_{k,j_k(z)}, u_{k,j_k(z)})$. If, for some $l\in\{0,\dots, N-3\}$, all the tests $\xi^*_{k,j_k(z)}(z)$, for $k=1,\dots, N-2-l$, decide correctly, then $\theta(\Pr)\in [l_{N-1-l,j_{N-1-l}(z)}, u_{N-1-l,j_{N-1-l}(z)})]$. Since, by construction, we have $\hat{\theta}_n(z)\in [a_{N-1-l,j_{N-1-l}(z)}, b_{N-1-l,j_{N-1-l}(z)})]$, and the latter interval has length $d_{N-1-l} = l\Delta$, this means that $|\hat{\theta}_n^{(\Delta)}(z)-\theta(\Pr)|\le (l+1)\Delta=\eta_l$. Therefore, if $|\hat{\theta}_n^{(\Delta)}(z)-\theta(\Pr)|>\eta_l$, then there exists $k\in\{1,\dots, N-2-l\}$, so that $\xi^*_{k,j_k(z)}(z)$ decided incorrectly. If $\xi^*_{k,j_k(z)}(z)$ incorrectly decided for $H_0$, then $\theta(\Pr)\in [b_{k,j_k(z)},u_{k,j_k(z)})$. By disjointness of $[b_{k,j},u_{k,j})$, $j=1,\dots, k$, there is at most one index $j_k^*= j_k^*(\Pr)\in\{1,\dots, k\}$, so that $\theta(\Pr)\in [b_{k,j_k^*},u_{k,j_k^*})$. Thus, $j_k(z)=j_k^*$, $\xi^*_{k,j_k^*}(z)=0$ and $\theta(\Pr)\in [b_{k,j_k^*},u_{k,j_k^*})$. If, on the other hand,  $\xi^*_{k,j_k(z)}(z)$ incorrectly decided for $H_1$, then $\theta(\Pr)\in [l_{k,j_k(z)},a_{k,j_k(z)})$. But analogously there is at most one index $j_k^{**}=j_k^{**}(\Pr)\in\{1,\dots, k\}$, such that $\theta(\Pr)\in [l_{k,j_k^{**}},a_{k,j_k^{**}})$. Thus, $\xi^*_{k,j_k^{**}}(z)=1$ and $\theta(\Pr)\in [l_{k,j_k^{**}},a_{k,j_k^{**}})$. This fact, that at any level $k$ there are at most two tests that can decide incorrectly, is the crucial point of our construction. Consequently, for $l=0,\dots, N-3$,
\begin{align*}
\quad&Q\Pr^{\otimes n} \left(z\in\mathcal Z^n: |\hat{\theta}_n^{(\Delta)}(z)-\theta(\Pr)|>\eta_l \right) \\
&\le
\sum_{k=1}^{N-2-l} Q\Pr^{\otimes n} \left( \xi^*_{k,j_k(z)}(z) \text{ decides incorrectly}\right) \\
&\le
\sum_{k=1}^{N-2-l} \left[ Q\Pr^{\otimes n} \left(\xi^*_{k,j_k^*}(z) = 0, \theta(\Pr)\in [b_{k,j_k^*},u_{k,j_k^*})\right)\right.\\
&\hspace{2cm}
\left.+Q\Pr^{\otimes n} \left( \xi^*_{k,j_k^{**}}(z) =1, \theta(\Pr)\in [l_{k,j_k^{**}},a_{k,j_k^{**}})\right) \right].
\end{align*}
But both 
$$
Q\Pr^{\otimes n} \left(\xi^*_{k,j_k^*}(z) = 0, \theta(\Pr)\in [b_{k,j_k^*},u_{k,j_k^*})\right)
$$ 
and 
$$
Q\Pr^{\otimes n} \left( \xi^*_{k,j_k^{**}}(z) =1, \theta(\Pr)\in [l_{k,j_k^{**}},a_{k,j_k^{**}})\right)
$$
are bounded by the worst case risk of the respective test and thus, in view of \eqref{eq:proof:lemma:BinarySearch}, they are both bounded by $2\eta_A^{(n)}(Q,d_k)\lor 0$. We conclude that
$$
Q\Pr^{\otimes n} \left( |\hat{\theta}_n^{(\Delta)}-\theta(\Pr)|>\eta_l \right) 
\le
4\sum_{k=1}^{N-2-l}[\eta_A^{(n)}(Q,d_k)\lor0] = 4\sum_{k=l+1}^{N-2}[\eta_A^{(n)}(Q, k\Delta)\lor0].
$$
If $l>N-3$, then the event $|\hat{\theta}_n^{(\Delta)}-\theta(\Pr)|>\eta_l$ is impossible, because the range of $\hat{\theta}_n^{(\Delta)}$ is $\{j\Delta:j=1,\dots, N-1\}$.

For part~\ref{lemma:BinarySearch:B}, we further investigate the binary search estimator in the case $Q_1=Q_1^{(\alpha,\ell^*)}$, as in \eqref{eq:binaryChannel}. In particular, we have $\mathcal Z=\{-z_0,z_0\}$. If $N\le 3$, then any estimator taking values in $[0,M]$ is at most $2\Delta$ away from $\hat{\theta}_n^{(\Delta)}$. We continue with $N\ge 4$. The marginal data generating distribution $Q_1^{(\alpha, \ell^*)}\Pr$ is actually a binary distribution on $\{-z_0,z_0\}$, taking the value $z_0 = \|\ell^*\|_\infty\frac{e^\alpha+1}{e^\alpha-1}$ with probability $p(\Pr) := \frac{1}{2}\left( 1 + \frac{\E_\Pr[\ell^*]}{z_0}\right)\in\left[\frac{1}{2}\left(1-\frac{e^\alpha-1}{e^\alpha+1}\right), \frac{1}{2}\left(1+\frac{e^\alpha-1}{e^\alpha+1}\right)\right]$. The corresponding likelihood is given by
$$
q(z_1,\dots, z_n; p) :=\prod_{i=1}^n p^{\frac{1}{2}\left(1+\frac{z_i}{z_0} \right)}(1-p)^{\frac{1}{2}\left(1-\frac{z_i}{z_0} \right)}
=
p^{nT(z)}(1-p)^{n(1-T(z)},
$$
where $T(z) = \frac{1}{2}\left(1+\frac{\bar{z}_n}{z_0} \right)$, with $\bar{z}_n = \frac{1}{n}\sum_{i=1}^n z_i$.
For $0<a\le b< M$, define $\bar{\P}_{\le a} := \{p(\Pr): \theta(\Pr)\le a\}$, $\bar{\P}_{\ge b}:=\{p(\Pr): \theta(\Pr)\ge b\}$, $\bar{a} := \sup \bar{\P}_{\le a}$ and $\bar{b}:= \inf \bar{\P}_{\ge b}$. Clearly, $p_-:= \inf_{\Pr\in\P} p(\Pr)\le \bar{a}$ and $\bar{b}\le \sup_{\Pr\in\P} p(\Pr) =: p_+$. Note that for any $\Pr_0,\Pr_1\in\P$, with $\theta(\Pr_0)-\theta(\Pr_1)\ge \Delta$, we have 
\begin{equation}
p(\Pr_0)-p(\Pr_1)= (\E_{\Pr_0}[\ell^*] - \E_{\Pr_1}[\ell^*])/(2z_0) \ge \inf_{\sigma\in\mathbb S_\Delta} \int_\X \ell^*\,d\sigma/(2z_0) >0 \label{eq:P0-P1}
\end{equation} 
by assumption. In particular, if $b-a\ge \Delta$, then $\bar{a} < \bar{b}$. Hence, a minimax test $\xi$ for testing $H_0 : [Q_1 \P_{\le a}]^{(n)} \cong \bar{\P}_{\le a}^{(n)}$ against $H_1: Q_1 \P_{\ge b} \cong  \bar{\P}_{\ge b}$ based on an iid sample of size $n$ is given by
\begin{align*}
\xi_{a,b}(z_1,\dots, z_n) = 
\begin{cases}
1, &\text{if } \frac{q(z_1,\dots, z_n; \bar{b})}{q(z_1,\dots, z_n; \bar{a})}\ge 1,\\
0, &\text{else.}
\end{cases}
\end{align*}
The test $\xi_{a,b}(z)$ decides for $H_1$ (i.e., $\xi_{a,b}(z)=1$) iff $T(z) \ge G(\bar{a},\bar{b})$, where 
$$
G(s,t) := \frac{\log\left( \frac{1-s}{1-t} \right)}{\log\left( \frac{t}{s} \frac{1-s}{1-t}\right)}, \quad\text{if } 0<s<t<1,
$$
and $G(s,s):=s$. In the following, we make repeated use of the facts that $G$ is strictly increasing in both arguments and that $s < G(s,t) < t$, for $0<s<t<1$ (see Lemma~\ref{lemma:CriticalVals}).

Abbreviate $a_j := a_{N-2,j}$ and $b_j := b_{N-2,j}$ and set $a_{N-1} = (N-1)\Delta$ and $b_0 = \Delta$. Next, define the critical values for the tests $\xi_{k,j} := \xi_{a_{k,j},b_{k,j}}$ by $c_{k,j} := G(\bar{a}_{k,j},\bar{b}_{k,j})$ and abbreviate $c_j := c_{N-2,j}$. Since $\bar{a}_{k,j}\le \bar{a}_{k,j+1}$ and $\bar{b}_{k,j}\le \bar{b}_{k,j+1}$, this defines a partition of $[0,1]$ (note that $\text{range}(T)\subseteq[0,1]$), i.e., $C_1 := [0,c_1)$, $C_{N-1} := [c_{N-2},1]$ and $C_j := [c_{j-1},c_j)$, for $j=2,\dots, N-2$, where we interpret $[c,c)=\varnothing$. We now show that $T(z)\in C_j$ implies that $\hat{\theta}_n^{(\Delta)}(z) = j\Delta$. If $T(z)\in C_1$, then all the tests along the binary search path must have decided for $H_0$, because $c_1 = G(\bar{a}_1,\bar{b}_1)$ is the smallest critical value among all critical values $c_{k,j}$. Thus, $\hat{\theta}_n^{(\Delta)}(z)=\Delta$. For $j\in\{2,\dots, N-2\}$, suppose that $T(z)\in C_j$. At some level $k_0$ along the binary search path, either $a_{k_0,j_{k_0}(z)} = j\Delta = a_j$ or $b_{k_0,j_{k_0}(z)}=j\Delta = b_{j-1}$ occurs first. In the former case, all tests at levels $k\ge k_0$ must decide for $H_0$, because $T(z)<c_j$ and $c_j$ is the smallest critical value of all tests with $k\ge k_0$ and $a_{k,j} = j\Delta$. In the latter case, all tests at levels $k\ge k_0$ must decide for $H_1$, because $T(z)\ge c_{j-1}$ and $c_{j-1}$ is the largest critical value among all tests for which $k\ge k_0$ and $b_{k,l}=j\Delta$. Thus, in either case, $\hat{\theta}_n^{(\Delta)}(z)=j\Delta$. Finally, if $T(z)\in C_{N-1}$, all tests must decide for $H_1$ and $\hat{\theta}_n^{(\Delta)}(z)=(N-1)\Delta$.

Next, by convexity of $\P$ and linearity of $p:\P\to[0,1]$, we see that the range $p(\P)\subseteq[p_-,p_+]\subseteq [0,1]$ of $p$ is an interval. 
Moreover, we have $p_-<c_1\le c_{N-2}<p_+$, because $\bar{a}_1<\bar{b}_1$, $\bar{a}_{N-2}<\bar{b}_{N-2}$ and the properties of $G$. 
In this paragraph we investigate the correspondence between the two functionals $\theta:\P\to[0,M]$ and $p:\P\to[p_-,p_+]$.
For $t\in (p_-,p_+)$, define $\varphi(t) := \sup\{\theta(\Pr) : p(\Pr)=t, \Pr\in\P\}$. If $t\le p_-$, set $\varphi(t) := 0$, and if $t\ge p_+$, set $\varphi(t) :=M$. For $p_-<t<c_1$, we see that $0\le \varphi(t) \le 2\Delta$, because
$$
\varphi(t) 
\le \sup\{\theta(\Pr):p(\Pr)\le c_1, \Pr\in\P \} 
\le \sup\{\theta(\Pr):p(\Pr)< \bar{b}_1, \Pr\in\P \} \le b_1 = 2\Delta.
$$
If $t\in C_j$, for some $j\in\{2, \dots, N-2\}$, then
\begin{align*}
\varphi(t) 
&\le \sup\{\theta(\Pr):p(\Pr)\le c_j, \Pr\in\P \} 
\le \sup\{\theta(\Pr):p(\Pr)< \bar{b}_j, \Pr\in\P \} \\
&\le b_j = (j+1)\Delta, 
\end{align*}
and 
\begin{align*}
\varphi(t) 
&\ge \inf\{\theta(\Pr):p(\Pr)\ge c_{j-1}, \Pr\in\P \} 
\ge \inf\{\theta(\Pr):p(\Pr)> \bar{a}_{j-1}, \Pr\in\P \} \\
&\ge a_{j-1} = (j-1)\Delta.
\end{align*}
Finally, for $c_{N-1}\le t<p_+$, one obtains $(N-2)\Delta\le \varphi(t) \le M$. Thus, we have defined $\varphi$ on all of $[0,1]$ in such a way that $|\varphi(T(z))-\hat{\theta}_n^{(\Delta)}(z)| \le \Delta$.

Next, we show that $\varphi$ can be approximated on $(p_-,p_+)$ by an affine function. First, note that for $t\in(p_-,p_+)$,
\begin{align*}
E(t)&:=\sup\{\theta(\Pr) : p(\Pr)=t, \Pr\in\P\} - \inf\{\theta(\Pr) : p(\Pr)=t, \Pr\in\P\} \\
&\le
\sup\{\theta(\Pr_0) - \theta(\Pr_1) : p(\Pr_0)-p(\Pr_1)=0, \Pr_0,\Pr_1\in\P\}.
\end{align*}
But if $\theta(\Pr_0) - \theta(\Pr_1)\ge \Delta$, then $p(\Pr_0)-p(\Pr_1)>0$ (cf. \eqref{eq:P0-P1}). Thus, $E(t)\le \Delta$. Now, for $\lambda\in[0,1]$ and $s,t\in(p_-,p_+)$, choose $\Pr_0, \Pr_1\in\P$ such that $p(\Pr_0)=s$ and $p(\Pr_1)=t$, set $\bar{\Pr}= \lambda\Pr_0 + (1-\lambda)\Pr_1\in\P$, by convexity, and note that
\begin{align*}
\varphi(\lambda s + (1-\lambda) t) 
&\le \inf\{ \theta(\Pr) : p(\Pr) = \lambda s + (1-\lambda)t, \Pr\in\P\} + \Delta\\
&\le \theta(\bar{\Pr})+\Delta = \lambda \theta(\Pr_0) + (1-\lambda)\theta(\Pr_1) + \Delta\\
&\le \lambda\varphi(s) + (1-\lambda)\varphi(t) + \Delta,
\end{align*}
where we have used linearity of $\theta$ and the previously derived bound on $E$.
Similarly, we obtain
\begin{align*}
\varphi(\lambda s + (1-\lambda) t) 
&\ge \theta(\bar{\Pr}) = \lambda \theta(\Pr_0) + (1-\lambda)\theta(\Pr_1)\\
&\ge \lambda\varphi(s) + (1-\lambda)\varphi(t) - \Delta.
\end{align*}
We conclude that $|\varphi(\lambda s + (1-\lambda) t) - [\lambda\varphi(s) + (1-\lambda)\varphi(t)]| \le \Delta$.
Now, fix $s_0\in(p_-,c_1)$ and $t_0\in(c_{N-1},p_+)$ and, for $t\in\R$, define 
$$
\psi(t) := \frac{t_0 - t}{t_0-s_0} \varphi(s_0) + \frac{t - s_0}{t_0-s_0}\varphi(t_0).
$$
Thus, for $t\in[s_0,t_0]$ and $\lambda_t := (t_0-t)/(t_0-s_0)\in[0,1]$, we have
\begin{align*}
|\psi(t)-\varphi(t)| &= |\psi(\lambda_t s_0 + (1-\lambda_t)t_0) - \varphi(\lambda_t s_0 + (1-\lambda_t)t_0)| \\
&\le
|\lambda_t \psi(s_0) + (1-\lambda_t)\psi(t_0) - [\lambda_t \varphi(s_0) + (1-\lambda_t)\varphi(t_0)]| + \Delta = \Delta.
\end{align*}
We have therefore found an affine function $\psi$, such that for $T(z)\in[c_1,c_{N-1}]$, we have $|\psi(T(z))-\hat{\theta}_n^{(\Delta)}(z)| \le 2\Delta$. 
Recall that by construction of $\varphi$, we have $\varphi(s_0) \le 2\Delta \le (N-2)\Delta\le \varphi(t_0)$, because $N\ge 4$. Thus, $\psi$ is non-decreasing.
If $T(z)<c_1$, then $\hat{\theta}_n^{(\Delta)}(z)=\Delta$ and $0\le \Pi[\psi(T(z))] \le \Pi[ \varphi(T(z)) + \Delta] \le \Pi[3\Delta] =3\Delta$, where $\Pi:\R\to[0,M]$ is the projection onto $[0,M]$. Thus, $|\Pi[\psi(T(z))] - \hat{\theta}_n^{(\Delta)}(z)|\le 2\Delta$. An analogous argument shows that the same bound holds if $T(z)>c_{N-1}$. Since $T(z)$ is affine in $\bar{z}_n$, the claim follows.
\end{proof}

\section*{Acknowledgements}
We want to thank an anonymous referee for raising the interesting conjecture that in the private case all linear functionals can be estimated by simple sample averages of appropriately privatized observations. We were able to confirm this conjecture after a substantial revision of an earlier version of the paper.


\appendix


\section{Proofs of lower bounds}
\label{sec:ApplowerB}

As a first step to lower bound $\mathcal M_{n,\alpha}(\P, \theta)$, we extend Theorem~2.1 of \citet{Donoho91} \citep[see also][Theorem 2.14]{Tsybakov09} to deduce a lower bound on $\mathcal M_n(Q,\P,\theta)$ for a fixed Markov probability kernel $Q$. In the case where $Q\in\mathcal Q_\alpha$, $\alpha\in(0,\infty)$, it holds without any assumptions on $\theta$ and $\P$, because the dominatedness condition on $Q\P^{(n)}$ is always satisfied if $Q$ is an $\alpha$-private channel. 

\begin{theorem}\label{thm:lowerDeltaA}
Let $\eta\in(0,1)$ be fixed and let $l:\R_+\to\R_+$ be a non-decreasing loss function. If $Q \P^{(n)}$ is dominated, then
$$
\mathcal M_n(Q,\P, \theta) \;=\; \inf_{\hat{\theta}_n}\sup_{\Pr\in\P} \E_{Q \Pr^{\otimes n}}\left[ l\left(|\hat{\theta}_n - \theta(\Pr)| \right)\right] \;\ge\; l\left( \left[\frac{1}{2}\Delta_A^{(n)}(Q,\eta)\right]^-\right)\frac{\eta}{2},
$$
where $l(0^-):= l(0)$.
\end{theorem}

\begin{proof}
If $\Delta_A^{(n)}(Q,\eta) = 0$, then the bound holds because $l\left( |\hat{\theta}_n - \theta(\Pr)|\right) \ge l(0)\frac{\eta}{2} =l\left( \frac{1}{2}\Delta_A^{(n)}(Q,\eta)\right)\frac{\eta}{2}$, in view of the monotonicity of $l$. Now, for arbitrary $\Delta\in[0,\infty)$, define the sets $S := \{ z\in\mathcal Z^n : | \hat{\theta}_n(z) - \theta(\Pr)| \ge \Delta\}$, $S_1 := \{ z\in\mathcal Z^n : \hat{\theta}_n(z) \ge t+ \Delta, \theta(\Pr)\le t\}$ and $S_2  := \{ z\in\mathcal Z^n : \hat{\theta}_n(z) < t+ \Delta, \theta(\Pr)\ge t+2\Delta\}$, which obey the inclusions $S_j\subseteq S$, for $j=1,2$. Therefore, we obtain the lower bound
\begin{align*}
&\sup_{\Pr\in\P} Q \Pr^{\otimes n}( S ) \ge \sup_{\Pr\in\P} \max\left\{ Q \Pr^{\otimes n}(S_1), Q  \Pr^{\otimes n}(S_2)\right\}\\
&\quad=
\max\left\{ \sup_{\Pr\in\P_{\le t}} Q \Pr^{\otimes n}\left( \hat{\theta}_n \ge t+\Delta\right), 
\sup_{\Pr\in\P_{\ge t + 2\Delta}} Q \Pr^{\otimes n}\left( \hat{\theta}_n < t+\Delta\right)
\right\}\\
&\quad\ge
\frac{1}{2}\sup_{\substack{\Pr_0\in\P_{\le t}\\ \Pr_1 \in \P_{\ge t+2\Delta}}} 
Q \Pr_0^{\otimes n}\left( \hat{\theta}_n \ge t+\Delta\right) 
+
Q \Pr_1^{\otimes n}\left( \hat{\theta}_n < t+\Delta\right)\\
&\quad\ge
\frac{1}{2} 
\inf_{\text{tests }\phi} \sup_{\substack{\Pr_0\in\P_{\le t}\\ \Pr_1 \in \P_{\ge t+2\Delta}}} 
\E_{Q \Pr_0^{\otimes n}}[\phi]
+
\E_{Q \Pr_1^{\otimes n}}[1-\phi]\\
&\quad=\frac{1}{2} 
\inf_{\text{tests }\phi} \sup_{\substack{\Pr_0\in Q \P_{\le t}^{(n)}\\ \Pr_1 \in Q \P_{\ge t+2\Delta}^{(n)}}} 
\E_{\Pr_0}[\phi]
+
\E_{\Pr_1}[1-\phi],
\end{align*}
which holds for any $t\in\R$. Since $Q \P^{(n)}$ is dominated by assumption, we can use \eqref{eq:KraftLeCam} to obtain
\begin{equation}\label{eq:supQP}
\sup_{\Pr\in\P} Q \Pr^{\otimes n}\left(
|\hat{\theta}_n - \theta(\Pr)| \ge \Delta
\right) \ge \frac{1}{2}\eta_A^{(n)}(Q,2\Delta).
\end{equation}
If $\Delta_A^{(n)}(Q,\eta) \in(0,\infty)$, then take $\eps>0$ such that $\Delta_0 := \frac{1}{2}[\Delta_A^{(n)}(Q,\eta) - \eps] > 0$. Since $\Delta\mapsto \eta_A^{(n)}(Q,\Delta)$ is non-increasing, the set $D:=\{ \Delta\ge 0 : \eta_A^{(n)}(Q,\Delta) > \eta\}$ is of interval form $D \supseteq [0,\Delta_A^{(n)}(Q,\eta))$ and thus $2\Delta_0\in D$, so that $\eta_A^{(n)}(Q,2\Delta_0) > \eta$. Thus, from \eqref{eq:supQP}, we obtain
$$
\sup_{\Pr\in\P} Q \Pr^{\otimes n}\left(
|\hat{\theta}_n - \theta(\Pr)| \ge \frac{1}{2}[\Delta_A^{(n)}(Q,\eta) - \eps]
\right) \ge \frac{\eta}{2}.
$$
If $l\left(\left[\frac{1}{2}\Delta_A^{(n)}(Q,\eta)\right]^-\right) = 0$, then the claimed lower bound is trivial. Otherwise, the result follows from  Markov's inequality, since $\eps>0$ was arbitrarily small.

If $\Delta_A^{(n)}(Q,\eta) = \infty$, then $\eta_A^{(n)}(Q,\Delta) > \eta > 0$ for all $\Delta> 0$, and the inequality \eqref{eq:supQP} together with Markov's inequality, yields
$$
\mathcal M_n(Q,\P, \theta) \ge \inf_{\hat{\theta}_n}\sup_{\Pr\in\P} Q \Pr^{\otimes n}\left(|\hat{\theta}_n - \theta(\Pr)| \ge \Delta\right) l(\Delta) \ge \frac{\eta}{2} l(\Delta).
$$
Since $\Delta> 0$ was arbitrary, the claim follows.
\end{proof}

As pointed out by \citet{Donoho91} in the non-private case, the quantity $\Delta_A^{(n)}(Q,\eta)$ is not easy to calculate in general. However, in this case, where there is no privatization and $Q(\cdot|x)$ is the Dirac measure at $x$, these authors derive a general lower bound on $\Delta_A^{(n)}(Q,\eta)$ in terms of the Hellinger modulus of continuity of $\theta$, i.e.,
$$
\omega_H(\eps) \quad=\quad \sup\left\{ |\theta(\Pr_0)- \theta(\Pr_1)| : \dH(\Pr_0,\Pr_1)\le \eps, \Pr_j\in\P, j=0,1 \right\}.
$$
The Hellinger distance $\dH$ turns out to be exactly the right metric to characterize the minimax rate in the non-private case, because of its relation to the testing affinity \eqref{eq:pi} and its convenient behavior under product measures. In particular, we have the well-known identities 
\begin{align}\label{eq:HellingerIds}
\dH^2 = 2(1-\rho) \quad\text{and} \quad\rho(\Pr_0^{\otimes n},\Pr_1^{\otimes n})=\rho(\Pr_0,\Pr_1)^n,
\end{align}
where $\rho(\Pr_0,\Pr_1) = \int\sqrt{p_0p_1}$ is the Hellinger affinity, and
\begin{align}\label{eq:rhoPi}
\pi \le \rho \le \sqrt{\pi(2-\pi)} = \sqrt{(1-\dtv)(1+\dtv)} = \sqrt{1-\dtv^2};
\end{align}
cf. Equation~(3.7) of \citet{Donoho91}.
In that reference, the authors show that
$$
\Delta_A^{(n)}(Q,\eta) \quad\ge\quad  \omega_H\left(c\sqrt{\frac{|\log \eta|}{n}}\right),
$$
for all small $\eta>0$, all large $n\in\N$, and in the special case where $Q(A|x) = \mathds 1_A(x)$ is the channel that returns the original observations without privatization.
In the privatized case, and if $Q(A_1\times\dots\times A_n|x_1,\dots, x_n) = \prod_{i=1}^n Q_1(A_i|x_i)$ is a non-interactive channel with identical marginals $Q_1$, one can follow the same strategy to obtain a bound of the form
\begin{equation}\label{eq:DonohChar}
\Delta_A^{(n)}(Q,\eta) \quad\ge\quad \omega_H^{(Q_1)}\left(c\sqrt{\frac{|\log \eta|}{n}}\right),
\end{equation}
where $\omega_H^{(Q_1)}(\eps):= \sup\left\{ |\theta(\Pr_0)- \theta(\Pr_1)| : \dH(Q_1\Pr_0,Q_1\Pr_1)\le \eps, \Pr_j\in\P, j=0,1 \right\}$.
Moreover, if $Q_1$ is $\alpha$-private, one can use Theorem~1 of \citet{Duchi17} to show that
\begin{equation}\label{eq:lowerH1}
\omega_H^{(Q_1)}(n^{-1/2}) \;\ge \; \omega_{TV}\left( \frac{1}{2\sqrt{n}(e^\alpha-1)}\right),
\end{equation}
where 
$
\omega_{TV}(\eps) = \sup\left\{ |\theta(\Pr_0)- \theta(\Pr_1)| : d_{\text{TV}}(\Pr_0,\Pr_1)\le \eps, \Pr_j\in\P, j=0,1 \right\}
$
is the total variation (or $L_1$) modulus of continuity of $\theta$. This strategy, however, applies only to non-interactive channels with identical marginals. Of course, for interactive channels, an inequality as in \eqref{eq:DonohChar} can not be obtained. A more general approach can be based on the remarkable inequality 
\begin{align}\label{eq:dTVPrivateBound}
\dtv\left(Q \Pr_0^{\otimes n},Q \Pr_1^{\otimes n}\right) \;\le \; \sqrt{2n(e^\alpha-1)^2} \dtv(\Pr_0,\Pr_1),
\end{align}
which was first observed in \citet[][their Corollary~1 combined with Pinsker's inequality]{Duchi14}, and which holds for all $\alpha$-sequentially interactive channels $Q$. It is instructive to compare inequality \eqref{eq:dTVPrivateBound} to 
$$
\dtv\left(\Pr_0^{\otimes n},\Pr_1^{\otimes n}\right) \;\le \; n\cdot \dtv(\Pr_0,\Pr_1).
$$ 
The $n$-dependence of the latter is of best possible order, as can be easily verified in the example of $n$-fold products of uniform distributions.
The different factors $n$ and $\sqrt{n}$ in the bounds already hint at the fact that private minimax rates of convergence differ from their non-private counterparts.

To establish our lower bound on $\Delta_A^{(n)}(Q,\eta)$, we also need the quantities
\begin{align*}
\eta_2^{(n)}(Q,\Delta) \quad=\quad \sup_{t\in\R} \pi\left( Q \P_{\le t}^{(n)}, Q \P_{\ge t+\Delta}^{(n)}\right)
\end{align*}
and
\begin{align}\label{eq:Delta2}
\Delta_2^{(n)}(Q,\eta) \quad=\quad \sup\{ \Delta\ge 0 : \eta_2^{(n)}(Q,\Delta) > \eta\}.
\end{align}
Note that, compared to $\eta_A^{(n)}$, there is no convex hull in the definition of $\eta_2^{(n)}$. Hence, it is clear that we have $\eta_2^{(n)}(Q,\Delta) \le \eta_A^{(n)}(Q,\Delta)$, for all $n\in\N$, for all $\Delta\ge 0$ and for all channel distributions $Q$. This entails that  also $\Delta_2^{(n)}(Q,\eta)\le \Delta_A^{(n)}(Q,\eta)$, for all $n$, $\eta$ and $Q$. Thus, a lower bound on the minimax risk can be based on $\Delta_2^{(n)}(Q,\eta)$. The next step is, still for fixed $Q$, to pass over from $\Delta_2^{(n)}(Q,\eta)$ to the moduli of continuity $\omega_{TV}$ and $\omega_H$ of $\theta$.

\begin{lemma}\label{lemma:g1H}
Fix $\eta\in(0,1)$ and a channel distribution $Q$. Then
\begin{align*}
\omega_{TV}\left( g_{TV}(Q,\eta)^-\right) &\le \Delta_2^{(n)}(Q,\eta),\quad\text{and}\\
\omega_H\left( g_H(Q,\eta)^-\right) &\le \Delta_2^{(n)}(Q,\eta),
\end{align*}
where
\begin{align*}
g_{TV}(Q,\eta) &:= \inf\{ \dtv(\Pr_0,\Pr_1) : \pi(Q \Pr_0^{\otimes n}, Q \Pr_1^{\otimes n}) \le \eta, \Pr_j\in\P, j=0,1\}, \quad\text{and}\\
g_H(Q,\eta) &:= \inf\{ \dH(\Pr_0,\Pr_1) : \pi(Q \Pr_0^{\otimes n}, Q \Pr_1^{\otimes n}) \le \eta, \Pr_j\in\P, j=0,1\}.
\end{align*}
\end{lemma}
\begin{proof}
For $\delta>0$, set $C := C(Q,\eta,\delta) := \{|\theta(\Pr_0)-\theta(\Pr_1)| : \dtv(\Pr_0,\Pr_1) \le g_{TV}(Q,\eta)-\delta, \Pr_j\in\P \}$, so that $\sup C = \omega_{TV}\left( g_{TV}(Q,\eta) - \delta\right)$. If $C=\varnothing$, then the desired inequality is trivial. So let $\Delta\in C$. Then there exist $\Pr_0, \Pr_1\in\P$, such that $\Delta=|\theta(\Pr_0)-\theta(\Pr_1)|$ and $\dtv(\Pr_0,\Pr_1)\le g_{TV}(Q,\eta) - \delta$. But this entails that $\pi\left( Q \Pr_0^{\otimes n}, Q \Pr_1^{\otimes n}\right) >\eta$, or otherwise $g_{TV}$ could not be the infimum. Now, without loss of generality, let $t_0:= \theta(\Pr_0) \le \theta(\Pr_1)$, so that $\Pr_0\in\P_{\le t_0}$ and $\Pr_1\in\P_{\ge t_0+\Delta}$. Thus, $\eta_2^{(n)}(Q,\Delta) = \sup_{t\in\R} \pi\left( Q \P_{\le t}^{(n)}, Q \P_{\ge t+\Delta}^{(n)}\right) >\eta$. Hence, we have established that $C = C(Q,\eta,\delta) \subseteq \{\Delta\ge0: \eta_2^{(n)}(Q,\Delta)>\eta\}$ and therefore, $\sup C \le \Delta_2^{(n)}(Q,\eta)$. Now let $\delta\to0$. The result for $\omega_H$ is established in an analogous way.
\end{proof}

It remains to derive lower bounds on $g_{TV}$ and $g_H$. 

\begin{theorem}\label{thm:g1}
Fix $n\in\N$, $\eta\in(0,1)$, $\alpha\in(0,\infty)$ and let $Q$ be an $\alpha$-sequentially interactive channel as in \eqref{eq:Seq} and \eqref{eq:alphaSeq}. Then
$$
g_{TV}(Q,\eta) \;\ge \; \frac{1-\eta}{\sqrt{2n(e^\alpha-1)^2}},
$$
where $g_{TV}$ is as in Lemma~\ref{lemma:g1H}. Consequently, we have
$$
\Delta_A^{(n)}(Q,\eta) \;\ge\; 
\omega_{TV}\left(\left[ \frac{1-\eta}{\sqrt{2n(e^\alpha-1)^2}}\right]^- \right).
$$
Moreover, for all $\eta_0\in(0,1)$ and every $\eps_0>0$, there exists a finite positive constant $c>0$, such that 
$$
g_H(R,\eta) \;\ge \; c\sqrt{\frac{|\log\eta|}{n}},
$$
for all $\eta\in(0,\eta_0)$, for all $n>|\log\eta|/\eps_0$ and for \textbf{all} channels $R$, where $g_H$ is as in Lemma~\ref{lemma:g1H}. Consequently, for such $\eta$, $n$ and $\alpha$-private channel $Q$, we have
$$
\Delta_A^{(n)}(Q,\eta) \;\ge\; 
\omega_{TV}\left(\left[ \frac{1-\eta}{\sqrt{2n(e^\alpha-1)^2}}\right]^- \right)
\lor \omega_H\left(\left[ c\sqrt{\frac{|\log\eta|}{n}} \right]^- \right).
$$
\end{theorem}

\begin{proof}
Recall \eqref{eq:dTVPrivateBound}, i.e.,
\begin{align*}
\dtv(Q \Pr_0^{\otimes n}, Q \Pr_1^{\otimes n}) \le 
\sqrt{2n}(e^\alpha -1) \dtv(\Pr_0,\Pr_1),
\end{align*}
which holds for all $\alpha$-sequentially interactive channels $Q$ and every $\Pr_0,\Pr_1\in\P$. On the other hand, we have the well known identity
\begin{align*}
\pi(Q \Pr_0^{\otimes n},Q \Pr_1^{\otimes n}) &= 1 -  \sup_{\text{tests }\phi} \left(\E_{Q \Pr_0^{\otimes n}}[\phi] - \E_{Q \Pr_1^{\otimes n}}[\phi]\right) \\
&= 1 - \dtv(Q \Pr_0^{\otimes n},Q \Pr_1^{\otimes n}),
\end{align*}
so that
$$
\dtv(\Pr_0,\Pr_1) \ge \frac{1-\pi(Q \Pr_0^{\otimes n},Q \Pr_1^{\otimes n})}{\sqrt{2n}(e^\alpha-1)}.
$$
The first result now follows from the definition of $g_{TV}$.
For the lower bound on $g_H$, we use the Hellinger identities in \eqref{eq:HellingerIds} and \eqref{eq:rhoPi}, as well as Lemma~\ref{lemma:Hdistance} in the supplement, to obtain
\begin{align*}
\dH^2(\Pr_0,\Pr_1) &= 2\left(1-\rho\left(\Pr_0^{\otimes n},\Pr_1^{\otimes n}\right)^{1/n}\right) 
\ge 
2\left(1-\rho\left(R \Pr_0^{\otimes n},R \Pr_1^{\otimes n}\right)^{1/n}\right) \\
&\ge
2\left(1-\left(\pi\left(R \Pr_0^{\otimes n}, R \Pr_1^{\otimes n} \right)\left[2- \pi\left(R \Pr_0^{\otimes n}, R \Pr_1^{\otimes n} \right)\right]\right)^{\frac{1}{2n}}\right).
\end{align*}
Thus, by definition of $g_H$, we arrive at the lower bound
$$
g_H(R,\eta) \ge \sqrt{2\left(1-\left(\eta\left[2- \eta\right]\right)^{\frac{1}{2n}}\right)}.
$$
The result now follows from Lemma~3.3 of \citet{Donoho91}.
\end{proof}

We conclude with the following corollary also presented in the main article.

\begin{corollary}\label{corr:lowerModulusApp}
Fix $\eta_0,\eps_0\in(0,1)$, $\alpha\in(0,\infty)$ and let $l:\R_+\to\R_+$ be a non-decreasing loss function. Then there exists a positive finite constant $c = c(\eta_0,\eps_0)$, such that for all $\eta\in(0,\eta_0)$ and for all $n>|\log\eta|/\eps_0$,
\begin{align*}
&\mathcal M_{n,\alpha}(\P, \theta) \;=\; \inf_{Q\in\mathcal Q_\alpha} \mathcal M_n(Q,\P, \theta) \\
&\quad\ge \;  l\left( \frac{1}{2} \left[ 
\omega_{TV}\left(\left[ \frac{1-\eta}{\sqrt{2n(e^\alpha-1)^2}}\right]^-\right) 
\lor
\omega_H\left(\left[ c\sqrt{\frac{|\log\eta|}{n}} \right]^- \right) \right]^-\right)\frac{\eta}{2},
\end{align*}
where $\mathcal Q_\alpha$ is the set of $\alpha$-sequentially interactive channels $Q$ as in \eqref{eq:SIset}.
\end{corollary}


\section{Proof and discussion of an extension of Theorem~\ref{THM:UPPERB}}
\label{sec:App:ThmUpperB}

The following is an extended version of Theorem~\ref{THM:UPPERB}.

\begin{theorem}\label{thm:upperB:App}
Fix $n\in\N$, suppose that \ref{cond:thetaBound} and \ref{cond:loss} hold and that $Q$ is a non-interactive channel with identical marginals $Q_1$, such that $Q_1 \P$ is dominated (by a $\sigma$-finite measure), and such that for every $r\in\N$, we have
\begin{align}\label{eq:condOmegaH:App}
\omega_H^{(Q_1)}(rn^{-1/2}) \;\le\; r^2 \omega_H^{(Q_1)}(n^{-1/2}).
\end{align}
Fix $C \ge \sqrt{2\log 2a}+1$, $\Delta=C^2\omega_H^{(Q_1)}(n^{-1/2})$, $C_1 = \left[ 1 + \frac{8a^2}{2a-1} \right] a^{\lceil 2\log(C)/\log(3/2)\rceil}$ and $C_2 = \left[ 2+a^2 + \frac{8a^2}{2a-1} \right] a^{\lceil 2\log(C)/\log(3/2)\rceil}$, where $a>1$ is the constant from Condition~\ref{cond:loss}. If one of the following conditions
\begin{align}
\sup_{\substack{\Pr_0\in\conv(Q_1\P_{\le s})\\ \Pr_1\in\conv(Q_1\P_{\ge t})}} \rho(\Pr_0,\Pr_1) 
\quad&=\quad
\sup_{\substack{\Pr_0\in Q_1\P_{\le s}\\ \Pr_1\in Q_1\P_{\ge t}}} \rho(\Pr_0,\Pr_1), 
\quad\forall s,t: t-s\ge \Delta, 
\quad \text{\textbf{or}} \label{eq:Sufficient}\\
\sup_{
	\substack{\Pr_0\in \conv(Q \P_{\le s}^{(n)})\\
			\Pr_1\in \conv(Q \P_{\ge t}^{(n)}) 
			} 
	}  \pi(\Pr_0,\Pr_1) 
&=
\sup_{
	\substack{\Pr_0\in Q \P_{\le s}^{(n)}\\
			\Pr_1\in Q \P_{\ge t}^{(n)} 
			} 
	}  \pi(\Pr_0,\Pr_1), \quad\forall s,t:t-s\ge \Delta, \label{eq:SufficientDonoho}
\end{align}
holds, then there exists a binary search estimator $\hat{\theta}_n^{(\Delta)}:\mathcal Z^n\to\R$ with tuning parameter $\Delta$ (cf. Proposition~\ref{prop:BinarySearch} for details), such that
\begin{align*}
\sup_{\Pr\in\P} \E_{Q \Pr^{\otimes n}}\left[ l\left(|\hat{\theta}_n^{(\Delta)} - \theta(\Pr)| \right)\right]
\;\le\;C_1\cdot l\left( \omega_H^{\left(Q_1\right)}\left( n^{-1/2}\right)\right),
\end{align*}
and
\begin{align*}
\sup_{\Pr\in\P} \E_{Q \Pr^{\otimes n}}\left[ l\left(|\hat{\theta}_n^{(\Delta)} - \theta(\Pr)| + 2\Delta \right)\right]
\;\le\;a^2C_2\cdot l\left( \omega_H^{\left(Q_1\right)}\left( n^{-1/2}\right)\right).
\end{align*}
\end{theorem}

\begin{remark}\quad
\label{rem:L1HQmoduli}
\begin{itemize}
\item Condition~\eqref{eq:condOmegaH:App} replaces and relaxes the assumption of a H\"olderian behavior of the Hellinger modulus (i.e., $\omega_H(\eps) \asymp \eps^\beta$) that was maintained by \citet{Donoho91} in the case of direct observations ($Q(A|x)=\mathds 1_A(x)$). In that case the H\"older condition on $\omega_H$ with exponent $\beta\le 2$ implies \eqref{eq:condOmegaH:App} (up to constants). The requirement that $\beta\le2$ is natural (cf. Lemma~\ref{lemma:bestL1Rate}). 

\item In Lemma~\ref{lemma:Hom2} we show that Condition~\eqref{eq:condOmegaH:App} is satisfied in particular if $Q_1\P$ is convex, which holds for any channel $Q_1$ if $\P$ itself is convex. However, if $Q_1(\cdot|x)$ is supported on a binary set $\mathcal Z= \{-z_0,z_0\}$ that is independent of $x\in\X$ (as, for instance, in \eqref{eq:binaryChannel}), then it suffices that $(\P, \dtv)$ is connected in order for $Q_1\P$ to be convex. To see this, simply note that then $Q_1\P$ is also connected, by $\dtv$-continuity of $\Pr\mapsto Q_1\Pr$, and any set of binary distributions with a common support is homeomorphic to a subset of $[0,1]$, which is connected if, and only if, it is convex.
\end{itemize}
\end{remark}

\begin{remark}
Condition~\eqref{eq:SufficientDonoho} is the privatized version of Condition~(4.2) in \citet{Donoho91}. It is instructive to study Section~4 of that reference to gain some intuition on the mechanism that leads to the attainability result. This discussion, extended to the privatized case, applies also in the present paper, but we do not repeat it here. We only point out that under \eqref{eq:SufficientDonoho}, the quantities $\Delta_A^{(n)}(Q,\eta)$ and $\Delta_2^{(n)}(Q,\eta)$ defined in \eqref{eq:DeltaA} and \eqref{eq:Delta2} coincide, so that the bound $\Delta_2^{(n)}(Q,\eta) \le \Delta_A^{(n)}(Q,\eta)$, which was instrumental in the previous section, is actually tight.
\end{remark}

\begin{remark}\label{rem:Sufficient}
Condition~\eqref{eq:Sufficient} is an alternative to Condition~\eqref{eq:SufficientDonoho} and a version of it in the non-private case also appears in \citet[][Lemma~4.3]{Donoho91}. Note that, unlike \eqref{eq:SufficientDonoho}, the suprema in \eqref{eq:Sufficient} are only over sets of one-dimensional marginal distributions, which may be convex even if the sets of corresponding product distributions are not. In particular, convexity of $Q_1\P_{\le s}$ and $Q_1\P_{\ge t}$ follows, if $\theta:\P\to\R$ is linear and $\P$ is convex, because then $\P_{\le s}$ and $\P_{\ge t}$ are both convex. However, because of the privatization effect, $Q_1\P_{\le s}$ and $Q_1\P_{\ge t}$ may be convex even though $\P_{\le s}$ and $\P_{\ge t}$ are not (cf. the discussion of Remark~\ref{rem:L1HQmoduli}). Of course, Condition~\eqref{eq:Sufficient} may also be satisfied for non-convex $Q_1\P_{\le s}$ and $Q_1\P_{\ge t}$.
\end{remark}

\begin{proof}

The proof is inspired by the proofs of Lemma~2.2, Theorem~2.3 and Theorem~2.4 of \citet{Donoho91}, but we consider the private case, we directly focus on the modulus $\omega_H^{(Q_1)}$ rather than on $\Delta_A^{(n)}(Q,\eta)$ and we relax an assumption. First, we propose an alternative version of the binary search estimator of \citet{Donoho91}, which is particularly designed for the privatized setting (see Proposition~\ref{prop:BinarySearch} in the main article).

For the proof of Theorem~\ref{thm:upperB:App}, we apply Proposition~\ref{prop:BinarySearch} with $\Delta := C^2\omega_H^{(Q_1)}(n^{-1/2})$, as in the theorem. If $\omega_{H}^{(Q_1)}(n^{-1/2})=0$, then, by Lemma~\ref{lemma:Hdistance}, monotonicity and \eqref{eq:condOmegaH:App}, we have for $s>0$ and $r\ge s\sqrt{n}$, that $\omega_H(s) \le \omega_{H}^{(Q_1)}(s)\le \omega_{H}^{(Q_1)}(rn^{-1/2}) \le r^2\omega_{H}^{(Q_1)}(n^{-1/2}) =0$. Thus $\theta:\P\to\R$ is constant and $\hat{\theta}_n^{(0)}$ as in Proposition~\ref{prop:BinarySearch}\ref{lemma:BinarySearch:A} chooses that value. We proceed in the case $\omega_{H}^{(Q_1)}(n^{-1/2})>0$ and $\theta$ not constant. As in \citet{Donoho91} we establish bounds on the sum of upper affinities $\sum_{k=l+1}^{N-2}[\eta_A^{(n)}(Q,k\Delta)\lor0]$. Without loss of generality, we assume $N\ge 3$, since otherwise any non-negative bound is trivial.
For $k\in\{1,\dots, N-2\}$ we first show that
\begin{align}\label{eq:etaBoundH2}
\eta_A^{(n)}(Q,k\Delta) 
\quad\le\quad 
\sup_t 
\sup_{
	\substack{\Pr_0\in \P_{\le t}\\
			\Pr_1\in \P_{\ge t+k\Delta}
			} 
	}  \left(1-\frac{1}{2}\dH^2(Q_1  \Pr_0,Q_1  \Pr_1)\right)^n.
\end{align}
If \eqref{eq:SufficientDonoho} holds, then, using \eqref{eq:rhoPi} and \eqref{eq:HellingerIds}, we obtain
\begin{align*}
\eta_A^{(n)}(Q,k\Delta) \quad&= \quad
\sup_t 
\sup_{
	\substack{\Pr_0\in \conv(Q \P_{\le t}^{(n)})\\
			\Pr_1\in \conv(Q \P_{\ge t+k\Delta}^{(n)}) 
			} 
	}  \pi(\Pr_0,\Pr_1)\\
\quad&=\quad
\sup_t 
\sup_{
	\substack{\Pr_0\in Q \P_{\le t}^{(n)}\\
			\Pr_1\in Q \P_{\ge t+k\Delta}^{(n)} 
			} 
	}  \pi(\Pr_0,\Pr_1)\\
\quad&\le \quad
\sup_t 
\sup_{
	\substack{\Pr_0\in [Q_1 \P_{\le t}]^{(n)}\\
			\Pr_1\in [Q_1 \P_{\ge t+k\Delta}]^{(n)} 
			} 
	}  \rho(\Pr_0,\Pr_1)\\
\quad&=\quad
\sup_t 
\sup_{
	\substack{\Pr_0\in Q_1\P_{\le t}\\
			\Pr_1\in Q_1\P_{\ge t+k\Delta} 
			} 
	}  \rho(\Pr_0,\Pr_1)^n.
\end{align*}
If, on the other hand, \eqref{eq:Sufficient} holds, then, using \eqref{eq:rhoPi} and Lemma~2 of \citet[][page~477]{LeCam86}\footnote{Notice the typo in the formulation of that Lemma.}, we have
\begin{align*}
\eta_A^{(n)}(Q,k\Delta) \quad&= \quad
\sup_t 
\sup_{
	\substack{\Pr_0\in \conv(Q \P_{\le t}^{(n)})\\
			\Pr_1\in \conv(Q \P_{\ge t+k\Delta}^{(n)}) 
			} 
	}  \pi(\Pr_0,\Pr_1)\\
\quad&\le\quad
\sup_t 
\sup_{
	\substack{\Pr_0\in \conv([Q_1\P_{\le t}]^{(n)})\\
			\Pr_1\in \conv([Q_1\P_{\ge t+k\Delta}]^{(n)}) 
			} 
	}  \rho(\Pr_0,\Pr_1)\\
\quad&\le \quad
\sup_t 
\sup_{
	\substack{\Pr_0\in \conv(Q_1\P_{\le t})\\
			\Pr_1\in \conv(Q_1\P_{\ge t+k\Delta}) 
			} 
	}  \rho(\Pr_0,\Pr_1)^n\\
\quad&=\quad
\sup_t 
\sup_{
	\substack{\Pr_0\in Q_1\P_{\le t}\\
			\Pr_1\in Q_1\P_{\ge t+k\Delta} 
			} 
	}  \rho(\Pr_0,\Pr_1)^n.
\end{align*}
In view of \eqref{eq:HellingerIds}, the expression on the last line of the two previous displays is equal to
$$
\sup_t 
\sup_{
	\substack{\Pr_0\in \P_{\le t}\\
			\Pr_1\in \P_{\ge t+k\Delta}
			} 
	}  \left(1-\frac{1}{2}\dH^2(Q_1  \Pr_0,Q_1  \Pr_1)\right)^n.
$$
This establishes \eqref{eq:etaBoundH2}.
Now, for $k\in\{1,\dots, N-2\}$ and for any $\delta\in(0,1)$ such that $C\sqrt{\delta}\ge 1$, set $r:= \lfloor \sqrt{kC^2\delta}\rfloor\ge1$ and note that because of \eqref{eq:condOmegaH:App}
\begin{align*}
k\Delta  &= kC^2\omega_{H}^{(Q_1)}(n^{-1/2}) > kC^2\delta\omega_{H}^{(Q_1)}(n^{-1/2}) \ge r^2 \omega_{H}^{(Q_1)}(n^{-1/2}) \ge \omega_{H}^{(Q_1)}(rn^{-1/2}).
\end{align*}
Since $\dH^2 = 2(1-\rho) \le 2$, we see that the set $\bar{\P} := \{(P_0,P_1)\in\P^2 : \theta(P_1)-\theta(P_0) \ge k\Delta\}$ over which the supremum in \eqref{eq:etaBoundH2} is taken, is empty if $rn^{-1/2}\ge \sqrt{2}$. If, however, $rn^{-1/2}<\sqrt{2}$, then for every pair $(P_0,P_1)\in\bar{\P}$, we have $|\theta(P_0) - \theta(P_1)|>\omega_{H}^{(Q_1)}(rn^{-1/2})$, which implies $\dH(Q_1P_0,Q_1P_1) > rn^{-1/2}$. Hence, in view of $\log(1+x)\le x$ for all $x>-1$, we obtain
\begin{align*}
&\left(1-\frac{1}{2}\dH^2(Q_1P_0,Q_1P_1) \right)^n \le \left(1 - \frac{1}{2}\frac{r^2}{n} \right)^n\\
&\quad\le
\exp\left(-n \frac{1}{2}\frac{r^2}{n}\right)
= e^{-r^2/2}\\
&\quad\le
\exp\left( -\frac{1}{2}\left[ \sqrt{kC^2\delta} - 1\right]^2\right)
=
\exp\left( -k\frac{1}{2}\left[ \sqrt{C^2\delta} - 1/\sqrt{k}\right]^2\right)\\
&\quad\le
\exp\left( -k\frac{1}{2}\left[ C\sqrt{\delta} - 1\right]^2\right) \\
&\quad\le
\exp\left( -k\frac{1}{2}\left[ (\sqrt{2\log(2a)}+1)\sqrt{\delta} - 1\right]^2\right) 
\xrightarrow[]{\delta\to1}
\exp\left( -k\log(2a)\right)
=
\left( \frac{1}{2a}\right)^k.
\end{align*}

Using these considerations, we can now derive an upper bound on the sum of $\eta_A^{(n)}(Q,k\Delta)\lor0$ in Proposition~\ref{prop:BinarySearch}, namely, for $j\in\{1,\dots, N-2\}$,
$$
\sum_{k=j}^{N-2} \left[\eta_A^{(n)}(Q,k\Delta) \lor0\right]
\le
\sum_{k=j}^{N-2} \exp\left( -k\frac{1}{2}\left[ (\sqrt{2\log(2a)}+1)\sqrt{\delta} - 1\right]^2\right). 
$$
Taking the limes superior $\delta\to1$ on both sides we obtain
$$
\sum_{k=j}^{N-2} \left[\eta_A^{(n)}(Q,k\Delta) \lor0\right] \le
\sum_{k=j}^{N-2} \left( \frac{1}{2a}\right)^{k}
\le
\sum_{k=j}^\infty \left( \frac{1}{2a}\right)^{k}
=
\left( \frac{1}{2a}\right)^{j}\frac{2a}{2a-1}.
$$
Consequently, from Proposition~\ref{prop:BinarySearch}, using the monotone convergence theorem, Condition~\ref{cond:loss} and Lemma~\ref{lemma:32x} and setting $\eta_{-1}=0$, we get    
\begin{align*}
&\sup_{\Pr\in\P} \E_{Q  \Pr^{\otimes n}} \left[ l\left( \left| \hat{\theta}_n^{(\Delta)} - \theta(\Pr)\right|\right)\right]
\le
\sup_{\Pr\in\P} \E_{Q  \Pr^{\otimes n}} \left[ \sum_{j=0}^\infty l(\eta_j) \mathds 1_{\left\{ \eta_{j-1} < |\hat{\theta}_n^{(\Delta)} - \theta(\Pr)| \le \eta_j\right\}}\right]\\
&\quad\le
\sum_{j=0}^\infty l(\eta_j) \sup_{\Pr\in\P}Q  \Pr^{\otimes n}\left( |\hat{\theta}_n^{(\Delta)} - \theta(\Pr)| > \eta_{j-1}\right) \\
&\quad\le
l(\Delta) + 4\sum_{j=1}^{\infty} l((j+1)\Delta) \sum_{k=j}^{N-2} \left[\eta_A^{(n)}(Q,k\Delta)\lor 0\right] \\
&\quad\le
l(\Delta) + 4\sum_{j=1}^{\infty} l\left((3/2)^{j+1}\Delta\right)\left( \frac{1}{2a}\right)^{j}\frac{2a}{2a-1}\\
&\quad\le
l(\Delta) + 4\sum_{j=1}^{\infty} a^{j+1} l(\Delta) \left( \frac{1}{2a}\right)^{j}\frac{2a}{2a-1}\\
&\quad=
l(\Delta)\left[ 1 + \frac{8a^2}{2a-1} \right].
\end{align*}
Now
\begin{align*}
l(\Delta)\left[ 1 + \frac{8a^2}{2a-1} \right]&=
l\left(C^2\omega_H^{(Q_1)}(n^{-1/2})\right)
\left[ 1 + \frac{8a^2}{2a-1} \right]\\
&\le
\left[ 1 + \frac{8a^2}{2a-1} \right] a^{\lceil 2\log(C)/\log(3/2)\rceil} l\left(\omega_H^{(Q_1)}(n^{-1/2})\right),
\end{align*}
where we have used the regularity condition \ref{cond:loss} of the loss function $l$ again. 

The second claim is established in almost the same way, using that $l(t) \le l((3/2)^2 t/2) \le a^2 l(t/2)$ and
\begin{align*}
&\sup_{\Pr\in\P} \E_{Q  \Pr^{\otimes n}} \left[ l\left( \left| \hat{\theta}_n^{(\Delta)} - \theta(\Pr)\right| + 2\Delta\right)\right]\\
&\quad\le
a^2\sup_{\Pr\in\P} \E_{Q  \Pr^{\otimes n}} \left[ \sum_{j=0}^\infty l(\eta_j) \mathds 1_{\left\{ \eta_{j-1} < \frac{1}{2}\left(|\hat{\theta}_n^{(\Delta)} - \theta(\Pr)|+2\Delta\right) \le \eta_j\right\}}\right]\\
&\quad\le
a^2 \sum_{j=0}^\infty l(\eta_j) \sup_{\Pr\in\P}Q  \Pr^{\otimes n}\left( |\hat{\theta}_n^{(\Delta)} - \theta(\Pr)| +2\Delta > 2\eta_{j-1}\right)\\
&\quad\le
a^2 \sum_{j=0}^\infty l(\eta_j) \sup_{\Pr\in\P}Q  \Pr^{\otimes n}\left( |\hat{\theta}_n^{(\Delta)} - \theta(\Pr)| > \eta_{j-1}\right)\\
&\quad\quad
+a^2 \sum_{j=0}^\infty l(\eta_j) \sup_{\Pr\in\P}Q  \Pr^{\otimes n}\left( 2\Delta > \eta_{j-1}\right)\\
&\quad\le
a^2 \sum_{j=0}^\infty l(\eta_j) \sup_{\Pr\in\P}Q  \Pr^{\otimes n}\left( |\hat{\theta}_n^{(\Delta)} - \theta(\Pr)| > \eta_{j-1}\right) + a^2[l(\Delta) + l(2\Delta)]\\
&\quad\le
a^2l(\Delta)\left[ 1 + \frac{8a^2}{2a-1} \right] + a^2[l(\Delta) + a^2l(\Delta)]\\
&\quad=
l(\Delta) a^2 \left[ 2 + a^2 + \frac{8a^2}{2a-1} \right].
\end{align*}

\end{proof}


\section{An extended version of Theorem~\ref{THM:ATTAINABILITY}}
\label{thm:Attainability:ext}

\begin{theorem}\label{thm:Attainability}
For $\alpha\in(0,\infty)$ and $\ell\in L_\infty(\X)$ let $Q^{(\alpha,\ell)}$ be the non-interactive $\alpha$-private channel with identical marginals $Q_1^{(\alpha,\ell)}$ as in \eqref{eq:binaryChannel}. Let $\eps,\xi_0,\xi_1>0$, $c\in(0,1)$, $\delta = \delta(\eps, c,\xi_0,\xi_1) = \omega_{TV}\left(\frac{\eps}{c}\frac{e^\alpha+1}{e^\alpha-1} + \xi_0\right)+\xi_1$ and $\mathbb S_\delta = \{\Pr_0-\Pr_1 : \theta(\Pr_0)-\theta(\Pr_1) \ge \delta, \Pr_0,\Pr_1\in\P\}$. 
\begin{enumerate}
	\setlength\leftmargin{-20pt}
	\renewcommand{\theenumi}{(\roman{enumi})}
	\renewcommand{\labelenumi}{{\theenumi}} 
\item \label{thm:Attainability:A0}
We have
$$
\inf_{\sigma\in\mathbb S_\delta} \sup_{\ell:\|\ell\|_\infty\le 1}\int_\X \ell\,d\sigma \;\ge\;  2\frac{\eps}{c} \frac{e^\alpha+1}{e^\alpha-1}>0.
$$

\item \label{thm:Attainability:A}
If there exists $\ell^*\in L_\infty(\X)$, such that $\|\ell^*\|_\infty\le1$ and
\begin{equation}\label{eq:saddlepoint:App}
\inf_{\sigma\in\mathbb S_\delta} \int_\X \ell^*\,d\sigma \;\ge\; c \inf_{\sigma\in\mathbb S_\delta} \sup_{\ell:\|\ell\|_\infty\le 1}\int_\X \ell\,d\sigma,
\end{equation} 
then
$$
\omega_H^{(Q_1^{(\alpha, \ell^*)})}(\eps) \;\le\; \delta = \omega_{TV}\left( \frac{\eps}{c} \frac{e^\alpha+1}{e^\alpha-1} + \xi_0\right)+\xi_1.
$$

\item \label{thm:Attainability:B}
Suppose that
\begin{equation}\label{eq:saddlepoint2:App}
\sup_{\ell:\|\ell\|_\infty\le 1} \inf_{\sigma\in\mathbb S_\delta} \int_\X \ell\,d\sigma \;=\; \inf_{\sigma\in\mathbb S_\delta} \sup_{\ell:\|\ell\|_\infty\le 1}\int_\X \ell\,d\sigma,
\end{equation} 
where $\delta=\delta(\eps, c,\xi_0,\xi_1)$ is defined as above. Then there exists $\ell^*\in L_\infty(\X)$, such that $\|\ell^*\|_\infty\le1$ and \eqref{eq:saddlepoint:App} holds.

\item \label{thm:Attainability:C}
If $\P$ is convex and dominated by a $\sigma$-finite measure $\mu$ and $\theta:\P\to\R$ is linear, then \eqref{eq:saddlepoint2:App} holds for every $\delta\ge0$. In particular, we have
\begin{equation*}
\inf_{\ell:\|\ell\|_\infty\le1}\omega_H^{(Q_1^{(\alpha, \ell)})}(\eps) \;\le\; \omega_{TV}\left(\left[ \eps \frac{e^\alpha+1}{e^\alpha-1}\right]^+\right), \quad\quad \forall \eps>0.
\end{equation*}

\end{enumerate}
\end{theorem}

\begin{proof}
Part~\ref{thm:Attainability:A0} will be a by-product of the proof of part~\ref{thm:Attainability:A}. Thus, for $s\ge0$, define 
$$
\Phi_{\ell^*}(s) := \sup\{\theta(\Pr_0)-\theta(\Pr_1) : \Pr_0,\Pr_1\in\P, | \E_{\Pr_0}[\ell^*] - \E_{\Pr_1}[\ell^*]| \le  s \},
$$
and note that $d_{TV}\le \dH$, $\|\ell^*\|_\infty\le 1$ and \eqref{eq:dTVQP}, imply
$$
\omega_H^{(Q^{(\alpha, \ell^*)})}(\eps) \le \Phi_{\ell^*}\left(2\eps\frac{e^\alpha+1}{e^\alpha-1}\right).
$$
Clearly, the function $\Phi_{\ell^*}$ is non-decreasing.
For $t\ge0$, define 
$\Psi_{\ell^*}(t) := \inf \{ s\ge0 : \Phi_{\ell^*}(s)>t\}$. We claim that the functions $\Phi_{\ell^*}$ and $\Psi_{\ell^*}$ have the following properties.
\begin{align}
&\Psi_{\ell^*}(t) >s \;\Rightarrow\; \Phi_{\ell^*}(s)\le t, \label{eq:PsiPhi}\\
&\Psi_{\ell^*}(t) \ge \inf\{ | \E_{\Pr_0}[\ell^*] - \E_{\Pr_1}[\ell^*]| : \theta(\Pr_0)-\theta(\Pr_1)\ge t, \Pr_0,\Pr_1\in\P\} \label{eq:PsiLower}
\end{align}
The first one is obvious. To establish \eqref{eq:PsiLower}, set $A_{\ell^*}(t) := \{ s\ge0 : \Phi_{\ell^*}(s)>t\}$ and $B_{\ell^*}(t) := \{ | \E_{\Pr_0}[\ell^*] - \E_{\Pr_1}[\ell^*]| : \theta(\Pr_0)-\theta(\Pr_1)\ge t, \Pr_0,\Pr_1\in\P\}$ and note that for $A_{\ell^*}(t)=\varnothing$ the claim is trivial. So take $s\in A_{\ell^*}(t)$. Then $\Phi_{\ell^*}(s)> t$, which implies that there are $\Pr_0,\Pr_1\in\P$ with $| \E_{\Pr_0}[\ell^*] - \E_{\Pr_1}[\ell^*]| \le s$ and $\theta(\Pr_0)-\theta(\Pr_1)>t$. Thus, $\nu := | \E_{\Pr_0}[\ell^*] - \E_{\Pr_1}[\ell^*]| \le s$ and $\nu \in B_{\ell^*}(t)$. We have just shown that for every $s\in A_{\ell^*}(t)$ there exists a $\nu\in B_{\ell^*}(t)$ with $\nu\le s$. But this clearly means that $\Psi_{\ell^*}(t) = \inf A_{\ell^*}(t) \ge \inf B_{\ell^*}(t)$, as required. 

Now, abbreviate $\eta:=\frac{\eps}{c}\frac{e^\alpha+1}{e^\alpha-1}$ and note that \eqref{eq:PsiLower} together with \eqref{eq:saddlepoint} yields
\begin{align*}
\Psi_{\ell^*}(\delta) 
&\ge 
\inf_{\sigma\in\mathbb S_\delta} \left|\int_{\mathcal X} \ell^*\,d\sigma \right|
\ge
\inf_{\sigma\in\mathbb S_\delta} \int_{\mathcal X} \ell^*\,d\sigma 
\ge
c \inf_{\sigma\in\mathbb S_\delta} \sup_{\ell:\|\ell\|_\infty\le1} \int_{\mathcal X} \ell\,d\sigma \\
&=
2c \inf_{\sigma\in\mathbb S_\delta} \|\sigma\|_{TV} \\
&=
2c \inf\left\{d_{TV}(\Pr_0,\Pr_1) : \theta(\Pr_0)-\theta(\Pr_1)\ge\omega_{TV}\left(\eta+\xi_0\right)+\xi_1 \right\} \\
&\ge
2c \inf\left\{d_{TV}(\Pr_0,\Pr_1) : \theta(\Pr_0)-\theta(\Pr_1)>\omega_{TV}\left(\eta+\xi_0\right)\right\} \\
&\ge 2c (\eta+\xi_0) > 2c\eta = 2\eps\frac{e^\alpha+1}{e^\alpha-1}.
\end{align*}
An application of \eqref{eq:PsiPhi} now finishes the proof of part~\ref{thm:Attainability:A}.

Part~\ref{thm:Attainability:B} follows, because the supremum on the left-hand-side of \eqref{eq:saddlepoint2:App} can be arbitrarily well approximated and the right-hand-side is strictly positive, by \ref{thm:Attainability:A0}.
For part~\ref{thm:Attainability:C}, first note that $\mathbb T := \{\phi\in L_\infty(\X, \mathcal F, \mu) : -1 \le \int_{\X} \phi f\,d\mu \le 1, \forall f\in L_1(\X, \mathcal F, \mu): \|f\|_{L_1}\le1\} = \{\phi\in L_\infty(\X, \mathcal F, \mu) : \|\phi\|_\infty\le1\}$. Now, using convexity and dominatedness of $\P$ together with linearity of $\theta$, we see that $\mathbb S_\delta$, defined in the theorem, is a dominated convex set of finite signed measures. Hence, \eqref{eq:saddlepoint2:App} follows from Proposition~\ref{prop:Sion} with $a=-1$ and $b=1$, irrespective of the value of $\delta$. By parts~\ref{thm:Attainability:A} and \ref{thm:Attainability:B}, 
$$
\inf_{\|\ell^*\|_\infty\le1}\omega_H^{(Q_1^{(\alpha, \ell^*)})}(\eps) \;\le\; \omega_{TV}\left( \frac{\eps}{c} \frac{e^\alpha+1}{e^\alpha-1} + \xi_0\right)+\xi_1,
$$
holds for every $c\in(0,1)$ and $\eps,\xi_0,\xi_1>0$. Hence, the second claim follows upon letting $c\to1$ and $\xi_0,\xi_1\to0$.
\end{proof}


\section{Proofs of Corollaries~\ref{COR:ATTAINABILITY} and \ref{COR:LINEST}}
\label{sec:App:COR:ATTAINABILITY}

We begin with the proof of Corollary~\ref{COR:ATTAINABILITY}. For arbitrary $\eps,\xi_0,\xi_1>0$ and $c\in(0,1)$ to be chosen later, define $\delta:= \omega_{TV}\left( \frac{\eps}{c}\frac{e^\alpha+1}{e^\alpha-1} + \xi_0\right)+\xi_1$ as in Theorem~\ref{thm:Attainability} in the supplement. By  \ref{thm:Attainability:C}, \ref{thm:Attainability:B}, \ref{thm:Attainability:A} and \ref{thm:Attainability:A0} of that same theorem there exists $\ell^*\in L_\infty(\X)$, such that $\|\ell^*||_\infty\le 1$,
\begin{equation}\label{eq:SaddleCorr}
\inf_{\sigma\in\mathbb S_\delta} \int_\X \ell^*\,d\sigma \;\ge\; c \inf_{\sigma\in\mathbb S_\delta} \sup_{\ell:\|\ell\|_\infty\le 1}\int_\X \ell\,d\sigma >0,
\end{equation}
and
$$
\omega_H^{(Q_1^{(\alpha, \ell^*)})}(\eps) \;\le\; \delta = \omega_{TV}\left( \frac{\eps}{c} \frac{e^\alpha+1}{e^\alpha-1} + \xi_0\right)+\xi_1,
$$
where $\mathbb S_\delta = \{\Pr_0-\Pr_1: \theta(\Pr_0)-\theta(\Pr_1) \ge \delta, \Pr_0,\Pr_1\in\P\}$.
Let $Q_1^{(\alpha,\ell^*)}$ be the binary channel of \eqref{eq:binaryChannel}. By convexity of $\P$, also $Q_1^{(\alpha,\ell^*)}\P$ is convex and dominated by a finite two point measure on $\{-z_0,z_0\}$. Moreover, \eqref{eq:condOmegaH:App} of Theorem~\ref{thm:upperB:App} holds, in view of Lemma~\ref{lemma:Hom2}. Since $\theta$ is linear, also $Q_1^{(\alpha,\ell^*)}\P_{\le s}$ and $Q_1^{(\alpha,\ell^*)}\P_{\ge t}$ are convex, and \eqref{eq:Sufficient} holds.
Therefore, Theorem~\ref{thm:upperB:App} establishes existence of an estimator $\hat{\theta}_n:\mathcal Z^n\to\R$, such that for $Q=\bigotimes_{i=1}^n Q_1^{(\alpha,\ell^*)}$,
$$
\sup_{\Pr\in\P} \E_{Q \Pr^{\otimes n}}\left[ l\left(|\hat{\theta}_n - \theta(\Pr)| \right)\right]
\;\le\;C_1\cdot l\left( \omega_H^{\left(Q_1^{(\alpha,\ell^*)}\right)}\left( n^{-1/2}\right)\right),
$$
where $C_1 = \left[1+\frac{8a^2}{2a-1} \right] a^{\lceil 2\log(C)/\log(3/2)\rceil}$, $C=\sqrt{2\log 2a}+1$ and $a>1$ is the constant from Condition~\ref{cond:loss}. But the upper bound is further bounded by 
$$
C_1\cdot l\left(\omega_{TV}\left( \frac{\eps}{c} \frac{e^\alpha+1}{e^\alpha-1} + \xi_0\right)+\xi_1 \right)\le
aC_1\cdot l\left(\omega_{TV}\left( \frac{4}{\sqrt{n}} \frac{e^\alpha+1}{e^\alpha-1} \right) \right),
$$
upon choosing $\eps=n^{-1/2}$, $\xi_0=\frac{\eps}{c}\frac{e^\alpha+1}{e^\alpha-1}$, $c=1/2$ and $\xi_1= \omega_{TV}\left( \frac{\eps}{c} \frac{e^\alpha+1}{e^\alpha-1} + \xi_0\right)/2$. Actually, we should make sure that $\xi_1>0$. However, if this is not the case, then Lemma~\ref{lemma:bestL1Rate} shows that $\theta$ must be constant, in which case the conclusion of Corollary~\ref{COR:ATTAINABILITY} is trivially true. Therefore, its proof is finished.

To establish Corollary~\ref{COR:LINEST}, we first note that we are operating under the same assumptions as above. Fix $\delta$, $\eps$, $c$, $\xi_0$, $\xi_1$ and $\ell^*$ as above, let $C>0$ be as in the statement of the corollary and set 
$$
\Delta := C^2 \omega_H^{(Q_1^{(\alpha,\ell^*)})}(n^{-1/2}).
$$ 
Using the fact that the Hellinger distance is upper bounded by the square root of the Kullback-Leibler divergence \citep[Lemma~2.4]{Tsybakov09} together with Theorem~1 of \citet{Duchi17}, we obtain $\dH(Q_1^{(\alpha,\ell^*)}\Pr_0,Q_1^{(\alpha,\ell^*)}\Pr_1) \le 2(e^\alpha-1)\dtv(\Pr_0,\Pr_1)$. Therefore, using \eqref{eq:condOmegaH:App}, which we have already verified, we see that 
\begin{align*}
\Delta &\ge \frac{3}{2} \lfloor \bar{C}\rfloor^2\omega_H^{(Q_1^{(\alpha,\ell^*)})}(n^{-1/2}) 
\ge \frac{3}{2}\omega_H^{(Q_1^{(\alpha,\ell^*)})}\left(\lfloor \bar{C}\rfloor n^{-1/2} \right) \\
&\ge \frac{3}{2}\omega_{TV}\left(\frac{\lfloor \bar{C}\rfloor}{2\sqrt{n}(e^\alpha-1)} \right)
\ge \frac{3}{2}\omega_{TV}\left( \frac{4}{\sqrt{n}}\frac{e^\alpha+1}{e^\alpha-1} \right) = \delta \ge \xi_1,
\end{align*} 
where $\bar{C} = \max\{8(e^\alpha+1), \sqrt{2\log 2a}+1\}+1$. Again, for constant $\theta$ the corollary is trivial, such that we can assume $\xi_1>0$. Thus,
$$
\inf_{\sigma\in\mathbb S_\Delta} \int_\X \ell^*\,d\sigma
\ge \inf_{\sigma\in\mathbb S_\delta} \int_\X \ell^*\,d\sigma > 0.
$$
Therefore, Proposition~\ref{prop:BinarySearch}\ref{lemma:BinarySearch:B} yields an affine function $g^{\text{(aff)}}:\R\to\R$ (depending on $\ell^*$), such that
$$\left|\Pi[g^{\text{(aff)}}(\bar{z}_n)]-\hat{\theta}_n^{(\Delta)}(z)\right| \le 2\Delta,
\quad\quad\forall z\in\mathcal Z^n,
$$ 
for $\hat{\theta}_n^{(\Delta)}$ the binary search estimator from part~\ref{lemma:BinarySearch:A} of that lemma, and where $\Pi:\R\to[M_-,M_+]$ is the projection described in the statement of Corollary~\ref{COR:LINEST}.
Now, the second conclusion of Theorem~\ref{thm:upperB:App}, the assumptions of which have already been verified, and the fact that scaling $\ell^*$ is equivalent to scaling $\bar{Z}_n$, finish the proof. \hfill\qed


\section{Attainability for non-convex $\P$}
\label{sec:Attainability:App}

\begin{corollary}\label{COR:ATTAINABILITY:App}
Fix $\alpha\in(0,\infty)$, $n\in\N$, suppose that Assumptions~\ref{cond:thetaBound} and \ref{cond:loss} hold, that $(\P,\dtv)$ is connected and that there exists $c'\in(0,1]$ such that for every $\delta\in[0,\infty)$, we have
\begin{equation}\label{eq:corr:saddlepoint2:App}
\sup_{\ell:\|\ell\|_\infty\le 1} \inf_{\sigma\in\mathbb S_\delta} \int_\X \ell\,d\sigma \;\ge\; c'\inf_{\sigma\in\mathbb S_\delta} \sup_{\ell:\|\ell\|_\infty\le 1}\int_\X \ell\,d\sigma,
\end{equation} 
where $\mathbb S_\delta = \{\Pr_0-\Pr_1 : \theta(\Pr_0)-\theta(\Pr_1)\ge \delta, \Pr_0,\Pr_1\in\P\}$. Moreover, assume that $\omega_{TV}(\nu)>0$ for all $\nu> 0$.
Then, 
$$
\mathcal M_{n,\alpha}(\P, \theta) := \inf_{Q\in\mathcal Q_\alpha} \mathcal M_n(Q,\P, \theta)\;\le\; C_1\cdot l\left(\omega_{TV}\left( \frac{4}{c'\sqrt{n}} \frac{e^\alpha+1}{e^\alpha-1} \right) \right),
$$
where $\mathcal Q_\alpha$ is the collection of $\alpha$-sequentially interactive channels as in \eqref{eq:SIset}. The constant $C_1$ is given by 
$$
C_1 = \left[1+\frac{8a^2}{2a-1} \right] a^{\lceil \log(C)/\log(3/2)\rceil+1},
$$ 
where $C = \sqrt{\frac{3}{2}}[\max\{8(e^\alpha+1), \sqrt{2\log 2a}+1\}+1]$ and $a>1$ is the constant from Condition~\ref{cond:loss}.
\end{corollary}

\begin{remark*}
Condition~\ref{eq:corr:saddlepoint2:App} can be verified, for instance, in the problem of estimating the endpoint of a uniform distribution (cf. Section~\ref{sec:uniform}). In this case the model $\P$ is non-convex. 
\end{remark*}

\begin{proof}
For arbitrary $\eps,\xi_0,\xi_1>0$ and $c\in(0,c')$ to be chosen later, define $\delta:= \omega_{TV}\left( \frac{\eps}{c}\frac{e^\alpha+1}{e^\alpha-1} + \xi_0\right) + \xi_1$ as in Theorem~\ref{thm:Attainability}. By \eqref{eq:corr:saddlepoint2:App} and parts~\ref{thm:Attainability:A0} and \ref{thm:Attainability:A} of that theorem there exists $\ell^*\in L_\infty(\X)$, such that $\|\ell^*||_\infty\le 1$,
\begin{equation}\label{eq:corr:saddlepoint3:App}
\inf_{\sigma\in\mathbb S_\delta} \int_\X \ell^*\,d\sigma \;\ge\; c \inf_{\sigma\in\mathbb S_\delta} \sup_{\ell:\|\ell\|_\infty\le 1}\int_\X \ell\,d\sigma > 0,
\end{equation}
and
$$
\omega_H^{(Q_1^{(\alpha, \ell^*)})}(\eps) \;\le\; \delta = \omega_{TV}\left( \frac{\eps}{c} \frac{e^\alpha+1}{e^\alpha-1} + \xi_0\right) + \xi_1,
$$
where $Q_1^{(\alpha,\ell^*)}$ is the binary channel of \eqref{eq:binaryChannel}. Then, clearly, $Q_1^{(\alpha,\ell^*)}\P$ is dominated by a finite two point measure on $\{-z_0,z_0\}$. Moreover, \eqref{eq:condOmegaH:App} holds, in view of connectedness of $\P$ and the discussion of Remark~\ref{rem:L1HQmoduli}. In case $\omega_{H}^{(Q_1^{(\alpha,\ell^*)})}(n^{-1/2})=0$, then, by Lemma~\ref{lemma:Hdistance}, monotonicity and \eqref{eq:condOmegaH:App}, we have for $s>0$ and $r\ge s\sqrt{n}$, that $\omega_H(s) \le \omega_{H}^{(Q_1^{(\alpha,\ell^*)})}(s)\le \omega_{H}^{(Q_1^{(\alpha,\ell^*)})}(rn^{-1/2}) \le r^2\omega_{H}^{(Q_1^{(\alpha,\ell^*)})}(n^{-1/2}) =0$. Thus $\theta:\P\to\R$ is constant and the desired result is trivial. We continue in the case $\omega_{H}^{(Q_1^{(\alpha,\ell^*)})}(n^{-1/2})>0$ and $\theta$ not constant. Instead of \eqref{eq:Sufficient} we verify the equivalent condition
\begin{equation}\label{eq:equivSufficent}
\inf_{\substack{\Pr_1\in\conv(Q_1^{(\alpha,\ell^*)}\P_{\le s})\\ \Pr_0\in\conv(Q_1^{(\alpha,\ell^*)}\P_{\ge t})}} \dH(\Pr_0,\Pr_1) 
\quad=\quad
\inf_{\substack{\Pr_1\in Q_1^{(\alpha,\ell^*)}\P_{\le s}\\ \Pr_0\in Q_1^{(\alpha,\ell^*)}\P_{\ge t}}} \dH(\Pr_0,\Pr_1), 
\quad\forall s,t: t-s\ge \Delta,
\end{equation}
where $\Delta=C^2\omega_H^{(Q_1^{(\alpha,\ell^*)})}(n^{-1/2})>0$ and $C> 0$ is as in the statement of the corollary. We identify a set of binary distributions supported on $\mathcal Z = \{-z_0,z_0\}$ with the set of corresponding success probabilities (probability of outcome $z_0$), i.e., by slight abuse of notation we write $Q_1^{(\alpha,\ell^*)}\P_{\le s}\subseteq [0,1]$. Therefore, to verify \eqref{eq:equivSufficent}, it remains to show that $\sup Q_1^{(\alpha,\ell^*)}\P_{\le s}< \inf Q_1^{(\alpha,\ell^*)}\P_{\ge t}$, for $t-s\ge \Delta$. General elements of $Q_1^{(\alpha,\ell^*)}\P_{\ge t}$ and $Q_1^{(\alpha,\ell^*)}\P_{\le s}$ are given by $\frac{1}{2}\left( 1 + \frac{\E_{\Pr_j}[\ell^*]}{z_0}\right)$, $j=0,1$, for $\theta(\Pr_1)\le s$ and $\theta(\Pr_0)\ge t$. Thus, for \eqref{eq:equivSufficent} it remains to show that the right-hand-side of
$$
\E_{\Pr_0}[\ell^*] - \E_{\Pr_1}[\ell^*] \;\ge\; \inf_{\sigma\in\mathbb S_\Delta} \int_\X\ell^*\,d\sigma,
$$
is strictly positive. But this follows from \eqref{eq:corr:saddlepoint3:App} if we can show that $\Delta\ge \delta$. To this end, we choose $\eps=n^{-1/2}$, $c=c'/2$, $\xi_0=\frac{\eps}{c}\frac{e^\alpha+1}{e^\alpha-1}$ and $\xi_1 = \omega_{TV}\left( \frac{\eps}{c}\frac{e^\alpha+1}{e^\alpha-1}+\xi_0\right)/2$, and note that these are as required above. In particular, $\xi_1>0$ by assumption. Using the fact that the Hellinger distance is upper bounded by the square root of the Kullback-Leibler divergence \citep[Lemma~2.4]{Tsybakov09} together with Theorem~1 of \citet{Duchi17}, we obtain $\dH(Q_1^{(\alpha,\ell^*)}\Pr_0,Q_1^{(\alpha,\ell^*)}\Pr_1) \le 2(e^\alpha-1)\dtv(\Pr_0,\Pr_1)$. Therefore, using \eqref{eq:condOmegaH:App}, which we have already verified, we see that 
\begin{align*}
\Delta &\ge \frac{3}{2} \lfloor \bar{C}\rfloor^2\omega_H^{(Q_1^{(\alpha,\ell^*)})}(n^{-1/2}) 
\ge \frac{3}{2}\omega_H^{(Q_1^{(\alpha,\ell^*)})}\left(\lfloor \bar{C}\rfloor n^{-1/2} \right) \\
&\ge \frac{3}{2}\omega_{TV}\left(\frac{\lfloor \bar{C}\rfloor}{2\sqrt{n}(e^\alpha-1)} \right)
\ge \frac{3}{2}\omega_{TV}\left( \frac{4}{\sqrt{n}}\frac{e^\alpha+1}{e^\alpha-1} \right) = \delta > 0,
\end{align*} 
where $\bar{C} = \max\{8(e^\alpha+1), \sqrt{2\log 2a}+1\}+1$. 
We have thus verified the assumptions of Theorem~\ref{thm:upperB:App}. The rest of the proof now follows in the same way as in the proof of Corollary~\ref{COR:ATTAINABILITY} in Section~\ref{sec:App:COR:ATTAINABILITY}.
\end{proof}

\section{Proof of Theorem~\ref{THM:CONSTRUCTIVE}}
\label{sec:App:THM:CONSTRUCTIVE}

Throughout this proof, we abbreviate $\Pr_n = Q^{(\alpha,\ell_{h_n})}\Pr^{\otimes n}$, $\E_n = \E_{\Pr_n}$, $Z_i(z) := z_i$ and 
$$
\Delta_n = \left( \frac{e^\alpha+1}{\sqrt{n}(e^\alpha-1)}\right)^{\frac{1}{1+\bar{r}}}.
$$ 
As a preliminary consideration, note that by definition of $Z_i$, $|Z_i| = z_0 = \|\ell_{h_n}\|_\infty\frac{e^\alpha+1}{e^\alpha-1}$, and for $p\ge 1$, $V_i := Z_i-\E_n[Z_i]$ satisfies
\begin{align*}
|V_i|^p 
&\le 
2^{p-1}(|Z_i|^p+\E_n[|Z_i|]^p)
\le
\left(2\|\ell_{h_n}\|_\infty\frac{e^\alpha+1}{e^\alpha-1}\right)^p\\
&\le
n^{p/2}\left(2D_0 \frac{e^\alpha+1}{\sqrt{n}(e^\alpha-1)} \left(  \frac{e^\alpha+1}{\sqrt{n}(e^\alpha-1)}\right)^{-\frac{\bar{r}}{1+\bar{r}}}\right)^p \\
&=
n^{p/2}(2D_0)^p \left(  \frac{e^\alpha+1}{\sqrt{n}(e^\alpha-1)}\right)^{\frac{p}{1+\bar{r}}},
\end{align*}
in view of Condition~\ref{cond:ellh}. Therefore, if $p\ge 2$, by the Marcinkiewicz-Zygmund inequality \citep[cf. Theorem~2 in Section~10.3 of][]{Chow97} and using Jensen's inequality (for the sample mean), we have
\begin{align*}
\E_n\left[ |\bar{Z}_n - \E_n[Z_1]|^p\right] 
&= 
n^{-p} \E_n\left[\left|\sum_{i=1}^n V_i\right|^p\right] 
\le
B_p n^{-p} \E_n\left[\left|\sum_{i=1}^n V_i^2\right|^{p/2} \right] \\
&=
B_p n^{-p/2} \E_n\left[\left|\frac{1}{n}\sum_{i=1}^n V_i^2\right|^{p/2} \right] \\
&\le
B_p n^{-p/2} \E_n\left[|V_1|^p \right] \\
&\le
B_p (2D_0)^p \left(  \frac{e^\alpha+1}{\sqrt{n}(e^\alpha-1)}\right)^{\frac{p}{1+\bar{r}}}\\
&=
B_p 2^pD_0^p \Delta_n^p,
\end{align*}
where $B_p$ is a constant that depends only on $p$.
For $a>1$ from Condition~\ref{cond:loss}, let $q = q(a)\ge 2$ be so that $\left(\frac{2}{3}\right)^q a<1$. 
For $\Pr\in\P$ and $n\in\N$, write $\nabla_n = |\bar{Z}_n-\theta(\Pr)|$ and set $\eta_0=0$ and $\eta_k = \left( \frac{3}{2}\right)^{k-1}\Delta_n$, for $k\in\N$. Recall that $\E_n[Z_1] = \E_\Pr[\ell_{h_n}]$ and 
$$
B_{\P,\theta}(\ell_{h_n}) = \sup_{\Pr\in\P}|\E_{\Pr}[\ell_{h_n}] - \theta(\Pr)|] \le D_0 \left( \frac{e^\alpha+1}{\sqrt{n}(e^\alpha-1)}\right)^{\frac{1}{1+\bar{r}}}
=
D_0\Delta_n,
$$
by \ref{cond:ellh}. Then, by the monotone convergence theorem,
\begin{align*}
\E_n[l(\nabla_n)] 
&\le 
\E_n\left[ \sum_{k=0}^\infty l(\eta_{k+1}) \mathds 1_{[\eta_k,\eta_{k+1})}(\nabla_n)\right]
\le
\sum_{k=0}^\infty l(\eta_{k+1}) \Pr_n(\nabla_n\ge \eta_k)\\
&\le
l(\eta_1) + \sum_{k=1}^\infty l(\eta_{k+1}) \frac{\E_n[\nabla_n^q]}{\eta_k^q}\\
&\le
l(\Delta_n) + l(\Delta_n)  \frac{\E_n[\nabla_n^q]}{\Delta_n^q}\frac{3}{2} \sum_{k=0}^\infty a^k\left( \frac{2}{3}\right)^{qk}  \\
&\le
l(\Delta_n) \left( 1+ \frac{2^{q-1}}{1-a(2/3)^q} \frac{\E_n[|\bar{Z}_n-\E_n[Z_1]|^q] + |\E_n[Z_1] - \theta(\Pr)|^q}{\Delta_n^q}   \right)\\
&\le
l(\Delta_n) \left( 1+ \frac{2^{q-1}}{1-a(2/3)^q} (B_q2^qD_0^q + D_0^q)   \right).
\end{align*}
 \hfill\qed

\section{Proofs of example section~\ref{SEC:EX}}
\label{sec:AppExamples}

To exclude trivialities, throughout this section we assume that $\theta$ is not constant on $\P$. We also point out that Condition~\ref{cond:ellh} together with a lower bound of the form $ \eps^{\frac{1}{1+\bar{r}}} \lesssim \omega_{TV}(\eps)$ implies the corresponding upper bound on $\omega_{TV}$, in view of Lemma~\ref{lemma:CondC}. 

\subsection{Estimating moment functionals}
\label{SEC:EXMomentFuncts}

Let $\X\subseteq\R$, $\mathcal F=\B(\X)$ and consider estimation of an integral functional $\theta(\Pr) = \E_\Pr[f]$ for some measurable $f:\X\to\R$, such that either $(a)$ $\text{Im}(|f|)\supseteq(0,\infty)$, or $(b)$ $\|f\|_\infty<\infty$. For instance, $f(x) = x^m$, for moment estimation, or $f(x) = e^{sx}$, for estimation of the moment generating function at the point $s\in\R$. 
For $\kappa\in(1,\infty)$ and $L>0$, consider the class $\P=\P_\kappa(L)$ of all probability measures $\Pr$ on $\B(\X)$ such that $\E_\Pr[|f|^\kappa] \le L$. Clearly, the model $\P_\kappa(L)$ is convex and $\theta$ is bounded on $\P_\kappa(L)$, because $\sup_{\Pr\in\P_\kappa(L)}|\theta(\Pr)| \le \sup_{\Pr\in\P_\kappa(L)} \E_\Pr[|f|^\kappa]^{\frac{1}{\kappa}}\le L^{\frac{1}{\kappa}}<\infty$. 

We first show that in case $(a)$ the total variation modulus $\omega_{TV}$ of $\theta$ over $\P_\kappa(L)$ satisfies, for all $\eps\in(0,1)$,
\begin{align*}
\omega_{TV}(\eps)\;\ge\; (L/2)^{\frac{1}{\kappa}} \eps^{\frac{\kappa-1}{\kappa}}.
\end{align*}
\begin{proof}
Fix $\eps\in(0,1)$ and $\delta\in(0,(L/2)^{1/\kappa})$. By our assumption on $|f|$, there exist $x_\eps, x_\delta\in\X$, such that $|f(x_\delta)| = \delta$ and $|f(x_\eps)| = (L/(2\eps))^{1/\kappa}$. Now let $\Pr_0$ be dirac measure at the point $\{x_\delta\}$ and let $\Pr_1(\{x_\delta\}) = 1-\eps$ and $\Pr_1(\{x_\eps\}) = \eps$. Then $\E_{\Pr_0}[|f|^\kappa] = \delta^\kappa \le L/2 \le L$ and $\E_{\Pr_1}[|f|^\kappa] = \delta^{\kappa}(1-\eps) + (L/(2\eps))\eps \le  L/2 + L/2 = L$. So both $\Pr_0$ and $\Pr_1$ belong to $\P_\kappa(L)$. Furthermore, $\dtv(\Pr_0,\Pr_1) = \eps$ and
\begin{align*}
\left| \E_{\Pr_0}[f] - \E_{\Pr_1}[f] \right|
=
\left| f(x_\delta)(1-\eps) + f(x_\eps)\eps - f(x_\delta) \right| 
\xrightarrow[\delta\to 0]{} &|f(x_\eps)|\eps \\
&= (L/2)^{\frac{1}{\kappa}} \eps^{\frac{\kappa-1}{\kappa}}.
\end{align*}
Thus, we have exhibited a sequence in the set $\{ |\theta(\Pr_0)- \theta(\Pr_1)| : \dtv(\Pr_0,\Pr_1)\le \eps, \Pr_j\in\P_\kappa(L)\}$ that converges to $(L/2)^{\frac{1}{\kappa}} \eps^{\frac{\kappa-1}{\kappa}}$, and the supremum can not be less than that quantity.
\end{proof}

Furthermore, $\ell_{h}(x) := f(x)\mathds 1_{|f(x)|\le \frac{1}{h}}$ satisfies Condition~\ref{cond:ellh} with $k=1$, $s_1=1$, $t_1=\kappa-1>0$, $D_0=C\lor 1$ and ${h}_0=1$. 
\begin{proof}
Clearly, $\|\ell_h\|_\infty \le h^{-1}$. To verify \eqref{eq:BiasCondition}, note that
\begin{align*}
B(\ell_h) &= \sup_{\Pr\in\P_\kappa(L)}|\theta(\Pr) - \E_\Pr[\ell_h]|
=
\sup_{\Pr\in\P_\kappa(L)}\left|\E_\Pr[f] - \E_\Pr\left[f\mathds 1_{|f|\le \frac{1}{h}}\right]\right|\\
&\le
\sup_{\Pr\in\P_\kappa(L)}
\E_\Pr\left[ |f| \mathds 1_{|f|>\frac{1}{h}}\right]
\le
\sup_{\Pr\in\P_\kappa(L)}
\left(\E_\Pr\left[ |f|^\kappa\right]\right)^{\frac{1}{\kappa}} \Pr(|f|>1/h)^{1-\frac{1}{\kappa}}\\
&\le
L^\frac{1}{\kappa}\sup_{\Pr\in\P_\kappa(L)} \left(h^\kappa \E_\Pr[|f|^\kappa]\right)^{\frac{\kappa-1}{\kappa}}
\le
L  h^{\kappa-1}.
\end{align*}

\end{proof}

Next, we show that in case $(b)$ there exist positive finite constants $A_0$ and $\eps_0$, so that the total variation modulus $\omega_{TV}$ of $\theta$ over $\P_\kappa(L)$ satisfies
$$
\omega_{TV}(\eps) \ge A_0\eps, \quad\quad\forall \eps\in[0,\eps_0].
$$
\begin{proof}
We would like to take $\Pr_0=\delta_{x_0}$ as point-mass at $x_0$, and $\Pr_1 = (1-\eps)\delta_{x_0} +\eps\delta_{x_1}$, for $x_0\ne x_1$ such that $\Pr_0,\Pr_1\in\P_\kappa(L)$. We first need to make sure that such a choice is possible. Note that if $|f(x)|^\kappa>L$, for all $x\in\mathcal X$, then $\P_\kappa(L)=\varnothing$. Suppose now that $|f(x_0)|^\kappa = L$ for some $x_0\in\mathcal X$ and that $|f(x)|^\kappa>L$ for all $x\neq x_0$. Then the only probability distribution $\Pr$ on $(\mathcal X,\mathcal B(\mathcal X))$ for which $\E_\Pr[|f|^\kappa] \le L$ holds, is the Dirac point mass at $x_0$. But this contradicts that $\theta$ is not constant on $\P_\kappa(L)$. Thus, either $(i)$ $|f(x_0)|^\kappa < L$ for some $x_0\in\X$, or $(ii)$ there exist $x_0, x_1\in\X$, $x_0\neq x_1$, such that $|f(x_j)|^\kappa = L$, for $j=0,1$. In case $(ii)$, suppose that $|f(x)|^\kappa\ge L$, for all $x\in\mathcal X$. Otherwise, we are in case $(i)$. Let $A = \{x\in\mathcal X : |f(x)|^\kappa=L\}$ and note that all elements of $\P_\kappa(L)$ must be supported on $A$. If $f$ is constant on $A$, then $\Pr\mapsto\theta(\Pr) = \E_\Pr[f]$ is constant on $\P_\kappa(L)$. But we have ruled out this trivial case. Thus, in case $(ii)$, there exist $x_0,x_1\in A$, such that $f(x_0)\neq f(x_1)$. 

Now, in case $(i)$, let $x_0$ be such that $|f(x_0)|^\kappa < L$ and take $x_1\in\mathcal X$, so that $f(x_1)\neq f(x_0)$. In case $(ii)$, let $x_0,x_1\in A$, such that $f(x_0)\neq f(x_1)$. In either case, for $\eps\in(0,1)$, set $\Pr_0 = \delta_{x_0}$, the Dirac point mass at $x_0$ and set $\Pr_1 = (1-\eps)\delta_{x_0} + \eps\delta_{x_1}$. Then, in case $(ii)$, we have $\E_{\Pr_0}|f|^\kappa = |f(x_0)|^\kappa\le L$ and $\E_{\Pr_1}|f|^\kappa = (1-\eps)|f(x_0)|^\kappa + \eps |f(x_1)|^\kappa\le L$. In case $(i)$, we have $\E_{\Pr_0}|f|^\kappa < L$, and there exists $\eps_0\in(0,1)$, so that for all $\eps\in[0,\eps_0]$, we have $\E_{\Pr_1}|f|^\kappa = (1-\eps)|f(x_0)|^\kappa + \eps |f(x_1)|^\kappa< L$. Hence, in both cases we have $\Pr_j\in\P_\kappa(L)$, for $j=0,1$, at least for $\eps\in[0,\eps_0]$. Now, it is easy to compute $|\theta(\Pr_0)-\theta(\Pr_1)| = \eps|f(x_0)-f(x_1)|$ which is non-zero in both cases. Since $\dtv(\Pr_0,\Pr_1) = \eps$, we arrive at the lower bound $\omega_{TV}(\eps) \ge |f(x_0)-f(x_1)| \eps$, for all $\eps\in[0,\eps_0]$.
\end{proof}

Furthermore, in case $(b)$, Condition~\ref{cond:ellh} holds with $k=1$, $s_1=0$, $t_1=\kappa-1$, $D_0=L\lor \|f\|_\infty$ and $h_0=1$. This is verified as in case $(a)$.

\subsection{Estimating the derivative of the density at a point}
\label{SEC:EXPointFuncts}

Let $\X=\R$, $\mathcal F=\B(\R)$ and, for $\beta,L>0$, consider the H\"older class $\mathcal H_{\beta,L}^{\ll\lambda} = \mathcal H_{\beta,L}^{\ll\lambda}(\R)$ of all Lebesgue densities $p$ on $\R$ that are $b := \lfloor \beta \rfloor$ times differentiable and whose $b$-th derivative $p^{(b)}$ satisfies
$$
|p^{(b)}(x) - p^{(b)}(y)| \le L |x-y|^{\beta-b},\quad\quad \forall x,y\in\R.
$$
We consider estimation of the $m$-th derivative of the density at a point $x_0\in\R$, i.e., for $p\in\mathcal H_{\beta,L}^{\ll\lambda}$, we consider the linear functional $\theta(p) = p^{(m)}(x_0)$, where $0\le m <\beta$. It follows from Condition~\ref{cond:ellh}, which we verify below \citep[see also][Equation~(1.9)]{Tsybakov09}, that this functional is uniformly bounded on $\mathcal H_{\beta,L}^{\ll\lambda}$. Clearly, $\mathcal H_{\beta,L}^{\ll\lambda}$ is convex.

We first show that there exist positive finite constants $A_0$ and $\eps_0$, depending only on $L$, $\beta$ and $m$, so that the total variation modulus $\omega_{TV}$ of $\theta$ over $\mathcal H_{\beta,L}^{\ll\lambda}$ satisfies, 
\begin{align}\label{eq:ModulusPoint}
\omega_{TV}(\eps) \ge A_0\eps^{\frac{\beta-m}{\beta+1}}, \quad\quad\forall \eps\in[0,\eps_0].
\end{align}
\begin{proof}
Let 
$$
\kappa_0(u) = \exp\left(-\frac{1}{1-4u^2}\right)\mathds 1_{(-1/2,1/2)}(u)
$$ and let $\kappa(u) = a_0\kappa_0(u)$, for an $a_0 = a_0(\beta)>0$, so that the $b$-th derivative of $\kappa$ is H\"older continuous with exponent $\beta-b$ and constant $1/2$. Similarly, for appropriate $a_1,a_2>0$, we obtain $p_0(x) := a_1\kappa_0( \frac{x-x_0}{a_2})$ such that $p_0\in\mathcal H_{\beta,L/2}^{\ll\lambda}(\R)$, and such that for constants $\delta_0,\delta_1>0$, depending only on $\beta$ and $L$, we have $p_0(x)\ge \delta_0$ for all $x\in(x_0-\delta_1,x_0+\delta_1)$. Now, for $x,y\in\R$ and $h>0$, set $g(y) := \kappa(y+1)-\kappa(y)$ and
$$
p_1(x) := p_0(x) + \frac{L}{2} h^\beta g\left( \frac{x-x_0}{h}\right).
$$
It follows that
\begin{align*}
|p_1^{(b)}(x) - p_1^{(b)}(y)| 
&\le 
|p_0^{(b)}(x) - p_0^{(b)}(y)|
\\
&\quad\quad+
\frac{L}{2} h^{\beta-b} \left|g^{(b)}\left(\frac{x-x_0}{h}\right) - g^{(b)}\left(\frac{y-x_0}{h}\right)\right|\\
&\le\frac{L}{2}|x-y|^{\beta-b} + \frac{L}{2} h^{\beta-b} \left| \frac{x-x_0}{h} - \frac{y-x_0}{h} \right|^{\beta-b}\\
&=
L|x-y|^{\beta-b}.
\end{align*}
Furthermore, since $g((x-x_0)/h)<0$ if, and only if, $x\in(x_0-h/2, x_0+h/2)$, we see that $p_1(x)\ge 0$, for all $x\in\R$, if $h<2\delta_1$ and $h^\beta < \delta_0/(L\|\kappa\|_\infty)$. Now
\begin{align}\label{eq:DensFunctDiff}
|\theta(p_0)-\theta(p_1)| = \frac{L}{2} h^{\beta-m} |g^{(m)}(0)|,
\end{align}
and
\begin{align*}
\dtv(\Pr_0,\Pr_1) = \frac{1}{2}\|p_0-p_1\|_{L_1} = \frac{L}{4}h^\beta \int_\R \left|g\left( \frac{x-x_0}{h}\right)\right|\,dx 
=
\frac{L}{2}h^{\beta+1} \|\kappa\|_{L_1}.
\end{align*}
Thus, if $\dtv(\Pr_0,\Pr_1)\le \eps$, the maximum value of $h$ we can choose is $h=(2\eps/(L\|\kappa\|_{L_1}))^{\frac{1}{\beta+1}}$, which obeys our restrictions on $h$, if $\eps\le \eps_0$, for an appropriate choice of $\eps_0$. Plugging this back into \eqref{eq:DensFunctDiff} yields the claimed lower bound.
\end{proof}

Let $K:[-1,1]\to \R$ be a kernel of order $b-m$ that is $m$-times continuously differentiable and satisfies $K^{(j)}(1) = K^{(j)}(-1) = 0$, for $j=0,1\dots, m-1$, $C_1:=\int_{-1}^1 |u|^{\beta-m}|K(u)|\,du < \infty$ and $\int_{-1}^1K(x)\,dx =(-1)^m$ \citep[for a construction see][]{Mueller84, Hansen05}. Then $\ell_h := \kappa_{h}^{(m)} \cdot \mathds 1_{[x_0-h,x_0+h]}$, where 
$$
\kappa_{h}(x) := \frac{1}{h}K\left( \frac{x-x_0}{h}\right),
$$ 
satisfies Condition~\ref{cond:ellh} with $k=1$, $s_1=m+1$, $t_1=\beta-m$, $D_0=\|K^{(m)}\|_\infty\lor \frac{C_1L}{(b-m)!}$ and $h_0=1$.
\begin{proof}
Since $\kappa_h^{(m)}(x) = h^{-(m+1)}K^{(m)}\left(\frac{x-x_0}{h}\right)$, we have $\|\ell_h\|_\infty\le\|K^{(m)}\|_\infty h^{-(m+1)}$.
Furthermore, $p^{(m)}$ is $(b-m)$-times continuously differentiable, so that 
$$
p^{(m)}(x_0+hu) = p^{(m)}(x_0) + hu\cdot p^{(m+1)}(x_0) + \dots + \frac{(hu)^{b-m}}{(b-m)!}p^{(b)}(x_0+\tau hu),
$$
for some $\tau\in[0,1]$. Thus, from the properties of the kernel $K$ and integration by parts, we get
\begin{align*}
|\E_p[\ell_h] - \theta(p)|
&= 
\left|\int_{x_0-h}^{x_0+h} \kappa_h^{(m)}(x) p(x)\,dx - p^{(m)}(x_0)\right|\\
&=
\left|(-1)^m\int_{x_0-h}^{x_0+h} \kappa_h(x) p^{(m)}(x)\,dx - p^{(m)}(x_0)\right|\\
&=
\left|\int_{-1}^{1} K(u) \left[p^{(m)}(x_0+hu)- p^{(m)}(x_0)\right]\,du\right|\\
&=
\left|\int_{-1}^{1} K(u) \frac{(hu)^{b-m}}{(b-m)!}p^{(b)}(x_0+\tau hu) \,du\right|\\
&=
\left|\int_{-1}^{1} K(u)  \frac{(hu)^{b-m}}{(b-m)!} \left[p^{(b)}(x_0+\tau hu)- p^{(b)}(x_0)\right]\,du \right|\\
&\le
\int_{-1}^{1} |K(u)|  \frac{|hu|^{b-m}}{(b-m)!} L |\tau h u|^{\beta-b}\,du
\le \frac{C_1L}{(b-m)!} h^{\beta-m}.
\end{align*}
\end{proof}

\subsection{Multivariate density estimation at a point}
\label{SEC:EXMultDens}

In this example, let $\X=\R^d$ and $\mathcal F = \B(\R^d)$. For $\beta\in(0,1]^d$ and $L\in(0,\infty)^d$, we consider the anisotropic H\"older-class $\mathcal H_{\beta, L}^{\ll\lambda}(\R^d)$ of Lebesgue densities $p$ on $\R^d$, such that for every $j\in\{1,\dots, d\}$ and every $x,x'\in\R^d$,
$$
|p(x_1,\dots, x_{j-1},x_j', x_{j+1},\dots, x_d) - p(x)| \;\le \; L_j|x_j'-x_j|^{\beta_j}.
$$
For $x_0\in\R^d$, the functional of interest is $\theta(p) = p(x_0)$, which is linear. Clearly, the anisotropic H\"older-class is convex. 

We first show that there exist positive finite constants $A_0$ and $\eps_0$, so that for $\bar{r} = \sum_{j=1}^d\frac{1}{\beta_j}$,
$$
\omega_{TV}(\eps) \;\ge\; A_0\eps^{\frac{1}{1+\bar{r}}}, \quad\quad\forall\eps\in(0,\eps_0].
$$
\begin{proof}
For $j=1,\dots, d$, we use the same construction as in Section~\ref{SEC:EXPointFuncts} to obtain kernels $\kappa_j\in\mathcal H_{\beta_j,1/2}^{\ll\lambda}(\R)$ and set $g_j(y) = \kappa_j(y+1) -\kappa_j(y)$. Moreover, we take  
$$
p_0(x) = \prod_{j=1}^d (2\pi\sigma^2)^{-1/2}\exp\left(-\frac{(x_j-x_{0,j})^2}{2\sigma^2}\right),
$$ 
which satisfies $p_0\in \mathcal H_{\beta,L/2}^{\ll\lambda}(\R^d)$ for some sufficiently large $\sigma>0$ depending only on $L=(L_1,\dots, L_d)'$. Furthermore, for $h, h_1, \dots, h_d>0$, define
$$
p_1(x) = p_0(x) + h \prod_{j=1}^d \frac{L_j}{2} g_j\left(\frac{x_j-x_{0,j}}{h_j}\right).
$$ 
Since the mappings $x_j \mapsto g_j((x_j-x_{0,j})/h_j)$ take only non-zero (and possibly negative) values if $x_{0,j}-3h_j/2<x_j<x_{0,j}+h_j/2$, we see that $p_1(x)$ is certainly non-negative for all $x\in\R^d$, provided that 
$$
h \prod_{j=1}^d  \frac{L_j}{2} \|g_j\|_\infty \;\le \; q_d(3(h_1,\dots, h_d)/2),
$$
where $q_d$ is the density of the $\mathcal N_d(0,\sigma^2 I_d)$ distribution. Thus, $p_1$ is a probability density, if $\max_j h_j\le 2/3$ and 
$$
h \;\le \; \frac{q_d(1,\dots, 1)}{\prod_{j=1}^d  \frac{L_j}{2} \|g_j\|_\infty}.
$$
Next, observe that for $x\in\R^d$ and $x_j'\in\R$,
\begin{align*}
&|p_1(x_1,\dots, x_{j-1},x_j',x_{j+1},\dots, x_d) - p_1(x)| \\
&\quad\le
|p_0(x_1,\dots, x_{j-1},x_j',x_{j+1},\dots, x_d) - p_0(x)| \\
&\quad\quad+
h  \left[\prod_{\substack{k=1\\k\ne j}}^d  \frac{L_k}{2} \left|g_k\left(\frac{x_k-x_k^{(0)}}{h_k}\right)\right|\right]
 \frac{L_j}{2}\left| g_j\left(\frac{x_j'-x_{0,j}}{h_j}\right) - g_j\left(\frac{x_j-x_{0,j}}{h_j}\right)\right| \\
&\quad\le
\frac{L_j}{2} |x_j' - x_j|^{\beta_j} 
+
h  \left[\prod_{\substack{k=1\\k\ne j}}^d  \frac{L_k}{2}\|g_k\|_\infty\right]
\frac{L_j}{2}\left|\frac{x_j'-x_{0,j}}{h_j} - \frac{x_j-x_{0,j}}{h_j} \right|^{\beta_j}\\
&\quad\le
L_j |x_j' - x_j|^{\beta_j}, 
\end{align*}
provided that 
$$
h_j^{-\beta_j} h\le \left[ \prod_{\substack{k=1\\k\ne j}}^d  \frac{L_k}{2}\|g_k\|_\infty \right]^{-1} =: \bar{c}_j.
$$
Consequently, we see that $p_1\in\mathcal H_{\beta,L}^{\ll\lambda}(\R^d)$, if $h_j^{-\beta_j} h \le c_0:= \min_j \bar{c}_j$, for all $j=1,\dots, d$. Now, $\theta(p_0) - \theta(p_1) = h\prod_{j=1}^d \frac{L_j g_j(0)}{2}=:h c_1$ and $\dtv(\Pr_0,\Pr_1) = h \prod_{j=1}^d\frac{h_jL_j\|g_j\|_1}{2} =: h[\prod_{j=1}^d h_j] c_2$. Thus, we solve the system
$$
h_j^{-\beta_j} h = c_0, \quad\quad c_2 h \prod_{j=1}^dh_j = \eps,
$$
which yields $h_j = (h/c_0)^{1/\beta_j}$ and 
$$
h=\left( \frac{\eps c_0^{\bar{r}}}{c_2} \right)^{\frac{1}{1+\bar{r}}},
$$
where, $\bar{r}=\sum_{j=1}^d\frac{1}{\beta_j}$, which establishes the claimed lower bound on $\omega_{TV}(\eps)$ for all small $\eps>0$.
\end{proof}

Furthermore, let $K:\R\to\R$ be a bounded kernel that satisfies
$$
\int_\R K(u)\,du = 1, \quad\quad \bar{c}_j:= \int_\R |K(u)| |u|^{\beta_j}\,du < \infty, \quad\forall j=1,\dots, d.
$$
Next, we show that with $h\in(0,\infty)^d$ and $x\in\R^d$, the function 
$$
\ell_h(x) = \prod_{j=1}^d \frac{1}{h_j} K\left( \frac{x_j-x_{0,j}}{h_j}\right)
$$ 
satisfies Condition~\ref{cond:ellh} with $k=d$, $s_j=1$, $t_j=\beta_j$, $D_0 = \|K\|_\infty\lor \max_j (dL_j\bar{c}_j)$ and $h_0=1$. 
\begin{proof}
The first part of \eqref{eq:BiasCondition} is trivial. To verify the bias condition, we note that for $p\in\mathcal H_{\beta,L}^{\ll\lambda}(\R^d)$, $H=\diag(h_1,\dots, h_d)$ and with the substitution $T(x) = H^{-1}(x-x^{(0)})$, $|DT^{-1}(u)| = |H| = \prod_{j=1}^d h_j$, we have
\begin{align*}
\left|\int_{\R^d} \ell_h(x) p(x)\,dx - p(x^{(0)})\right|
&\le
\int_{\R^d} \left| p(Hu+x^{(0)}) - p(x^{(0)})\right|\prod_{j=1}^d |K(u_j)|\,du\\
&\le
\sum_{j=1}^d \int_{\R^d} |K(u_j)| L_j\left| h_ju_j\right|^{\beta_j}\,du
=
\sum_{j=1}^d L_j \bar{c}_j h_j^{\beta_j}.
\end{align*}
Thus, \eqref{eq:BiasCondition} holds with $k=d$, $s_j=1$, $t_j=\beta_j$, $\bar{C}_0 = \|K\|_\infty\lor \max_j (dL_j\bar{c}_j)$ and $\bar{h}_0=1$.
\end{proof}

\subsection{Estimating the endpoint of a uniform distribution}
\label{sec:uniform}
Fix $M\ge1$ and consider the class of distributions $\P = \P_M = \{ \Pr_\vartheta : \vartheta\in(0,M]\}$, where $\Pr_\vartheta$ denotes the uniform distribution on $[0,\vartheta]$. We may take $\X = [0,M]$. Clearly, the functional $\theta(\Pr_\vartheta) = 2\int_\R x\,d\Pr_\vartheta(x) = \vartheta$ is linear (when extended to the convex hull of $\P$), but $\P$ is not convex. Nevertheless, the lower bound of Section~\ref{sec:lowerB} applies and, since for $\eps\in(0,1)$, $\dtv(\Pr_1,\Pr_{1-\eps/2}) = \eps$, we have $\omega_{TV}(\eps) \ge \eps/2$. Furthermore, we can verify the conditions of Corollary~\ref{COR:ATTAINABILITY:App}.
\begin{proof}
Boundedness of $\theta$ is clear and connectedness of $(\P,\dtv)$ is also immediate, as it is even a path-connected space. To verify \eqref{eq:corr:saddlepoint2:App}, we first compute
\begin{align*}
&\inf_{\sigma\in\mathbb S_\delta} \sup_{\ell:\|\ell\|_\infty\le 1}\int_\X \ell\,d\sigma
=
\inf_{\substack{\vartheta_0-\vartheta_1\ge \delta\\\vartheta_0,\vartheta_1\in(0,M]}} \int \left|\frac{1}{\vartheta_0}\mathds 1_{[0,\vartheta_0]} - \frac{1}{\vartheta_1}\mathds 1_{[0,\vartheta_1]} \right|\\
&\quad=
\inf_{\substack{\vartheta_0-\vartheta_1\ge \delta\\\vartheta_0,\vartheta_1\in(0,M]}}
\left|\frac{1}{\vartheta_0}-\frac{1}{\vartheta_1}\right| \vartheta_1 + \frac{\vartheta_0-\vartheta_1}{\vartheta_0}
=
2\frac{\delta\land M}{M}.
\end{align*}
Now take $\ell^*(x) := \frac{2x}{M} - 1$ and note that $\|\ell^*\|_\infty \le 1$ and
\begin{align*}
&\inf_{\sigma\in\mathbb S_\delta} \int_\X \ell^*\,d\sigma 
=
\inf_{\substack{\vartheta_0-\vartheta_1\ge \delta\\\vartheta_0,\vartheta_1\in(0,M]}}
\frac{2}{M}\left[\left(\frac{1}{\vartheta_0}-\frac{1}{\vartheta_1}\right) \int_0^{\vartheta_1} x \,dx + \frac{1}{\vartheta_0}\int_{\vartheta_1}^{\vartheta_0} x\,dx\right]\\
&\quad=
\inf_{\substack{\vartheta_0-\vartheta_1\ge \delta\\\vartheta_0,\vartheta_1\in(0,M]}}
\frac{2}{M}\left[\left(\frac{1}{\vartheta_0}-\frac{1}{\vartheta_1}\right) \frac{\vartheta_1^2}{2} + \frac{1}{\vartheta_0}\frac{\vartheta_0^2 - \vartheta_1^2}{2}\right]
=
\frac{\delta}{M}.
\end{align*}
Therefore,
$$
\sup_{\ell:\|\ell\|_\infty\le 1}\inf_{\sigma\in\mathbb S_\delta} \int_\X \ell\,d\sigma
\ge \frac{\delta}{M} \ge \frac{\delta\land M}{M} = \frac{1}{2} \inf_{\sigma\in\mathbb S_\delta} \sup_{\ell:\|\ell\|_\infty\le 1}\int_\X \ell\,d\sigma,
$$
and \eqref{eq:corr:saddlepoint2:App} holds with $c'=1/2$. Positivity of $\omega_{TV}(\nu)$, for $\nu>0$, has been established above.
\end{proof}
 
Alternatively, Condition~\ref{cond:ellh} clearly holds for $\ell_h(x) = 2x$ and with $k=1$, $s_1=0$, $D_0=2M$ and any $t_1>0$. Both of these arguments lead to a convergence rate of $l(n^{-1/2})$.
This rate should be compared with the well known rate of $l(n^{-1})$ from the case of direct observations. Even though $\P_M$ is not convex, the rate of direct estimation, in this case, is still characterized by the Hellinger modulus, as one easily shows that for $\eps\in(0,1)$, $\eps^2(1-\eps^2/4)\le \omega_H(\eps) \le M\eps^2$. 

\begin{proof}
First, to establish the identity 
\begin{equation}\label{eq:H^2}
\dH^2(\Pr_{\vartheta_0},\Pr_{\vartheta_1}) = 2\left(1-\frac{\vartheta_0\land\vartheta_1}{\sqrt{\vartheta_0\vartheta_1}}\right),
\end{equation}
simply note that 
$$
\rho(\Pr_{\vartheta_0},\Pr_{\vartheta_1}) = \int_\R\sqrt{\frac{1}{\vartheta_0\vartheta_1} \mathds 1_{[0,\vartheta_0\land \vartheta_1]}(x)}\,dx = \frac{\vartheta_0\land\vartheta_1}{\sqrt{\vartheta_0\vartheta_1}}.
$$
For the upper bound on $\omega_H$, note that if $0\le\vartheta_0< \vartheta_1\le M$,
\begin{align*}
|\theta(\Pr_{\vartheta_0})-\theta(\Pr_{\vartheta_1})| &= |\vartheta_0-\vartheta_1| 
= \vartheta_1\left| 1-\frac{\vartheta_0}{\vartheta_1}\right|
= \vartheta_1\left( 1-\sqrt{\frac{\vartheta_0}{\vartheta_1}}\right)\left( 1+\sqrt{\frac{\vartheta_0}{\vartheta_1}}\right)\\
&\le 2M\left( 1-\sqrt{\frac{\vartheta_0}{\vartheta_1}}\right).
\end{align*}
Moreover, $\dH(\Pr_{\vartheta_0},\Pr_{\vartheta_1})\le \eps$ if, and only if, $2\left( 1-\sqrt{\frac{\vartheta_0}{\vartheta_1}}\right)\le \eps^2$, by \eqref{eq:H^2}. Thus,
\begin{align*}
\omega_H(\eps) 
&= 
\sup\left\{|\theta(\Pr_{\vartheta_0})-\theta(\Pr_{\vartheta_1})|  : \dH(\Pr_{\vartheta_0},\Pr_{\vartheta_1})\le \eps, \vartheta_0,\vartheta_1\in[0,M]\right\}\\
&\le 
\sup\left\{2M\left( 1-\sqrt{\frac{\vartheta_0}{\vartheta_1}}\right)  : 2\left( 1-\sqrt{\frac{\vartheta_0}{\vartheta_1}}\right)\le \eps^2, 0\le\vartheta_0<\vartheta_1\le M\right\}\\
&\le M\eps^2.
\end{align*}
For the lower bound on $\omega_H$, take for $\eps\in(0,1)$, $\vartheta_0:=(1-\frac{\eps^2}{2})^2\le 1\le M$ and $\vartheta_1:=1\le M$, so that  $\omega_H(\eps)\ge |\vartheta_0-\vartheta_1| = 1-(1-\eps^2+\frac{\eps^4}{4}) = \eps^2(1-\frac{\eps^2}{4})$.
\end{proof}


\section{Auxiliary results and proofs}
\label{sec:AppAux}

\subsection{Proof of Proposition~\ref{prop:Sion}}
\label{sec:AppSion}

Set $M= \max(|a|,|b|)$. We check the conditions of Corollary~3.3 in \citet{Sion58}. Clearly, $\mathbb S$ and $\mathbb T$ are convex sets and the function $F(\phi,\sigma) := \int_{\Omega} \phi \;d\sigma$ on $\mathbb T \times \mathbb S$ is quasi-concave-convex, because it is linear in both arguments. We equip $\mathbb S$ with the topology induced by $\|\cdot\|_{TV}$ and note that this makes $\sigma\mapsto F(\phi,\sigma)$ continuous, for every $\phi\in\mathbb T$, because 
\begin{align*}
|F(\phi,\sigma_1) - F(\phi,\sigma_2)| &\le \sup_{\|\phi\|_\infty\le M} \left|\int_{\Omega} \phi\,d(\sigma_1-\sigma_2) \right| \le 2M\|\sigma_1-\sigma_2\|_{TV}.
\end{align*}
The proof is finished if we can find a topology $\tau$ for $L_\infty = L_\infty(\Omega, \mathcal A,\mu)$ in which $\mathbb T$ is compact and $\phi\mapsto F(\phi,\sigma)$ is continuous, for every $\sigma\in\mathbb S$. Abbreviate the space of equivalence classes of $\mu$-integrable functions by $L_1 = L_1(\Omega, \mathcal A, \mu)$ and its topological dual by $L_1^*=L_1^*(\Omega, \mathcal A, \mu)$. For $f\in L_1$, let $E_f: L_1^* \to \R$ denote the evaluation functional on the dual space $L_1^*$, i.e., for $\psi\in L_1^*$, set $E_f (\psi) = \psi(f)$. 
Set $V_M=\{f\in L_1 : \|f\|_{L_1}\le 1/M\}$ and $K = \{\psi\in L_1^* : |\psi(f)|\le M, \forall f\in V_1\}= \{\psi\in L_1^* : |\psi(f)|\le 1, \forall f\in V_M\}$. By the Banach-Alaoglu Theorem \citep[][Section~3.15]{RudinFunA}, the set $K$ is compact in the weak$^*$-topology on $L_1^*$, i.e., the weakest topology $\tau^*$ on $L_1^*$ for which all the evaluation functionals $E_f$, $f\in L_1$, are continuous. Next, we use the fact that $L_\infty$ is the dual of $L_1$ \citep[][Theorem~IV.8.5]{DunfSchw57}. Let $\Psi:(L_\infty, \|\cdot\|_\infty) \to (L_1^*, \|\cdot\|_{L_1^*})$ denote the isometric isomorphism that associates each $\phi\in L_\infty$ with the linear functional $f\mapsto \int_{\Omega} \phi f\,d\mu$. Therefore, we can map the weak$^*$-topology $\tau^*$ to a topology $\tau := \{\Psi^{-1}(O) : O\in\tau^*\}$ on $L_\infty$, so that all the functions $E_f\circ\Psi: (L_\infty, \tau) \to \R$, $f\in L_1$, are continuous and 
\begin{align*}
\Psi^{-1}(K) = \left\{\phi\in L_\infty : \left|\int_{\Omega} \phi f\,d\mu\right| \le M, \forall f\in V_1\right\} 
= \{\phi\in L_\infty : \|\phi\|_\infty \le M\}
\end{align*}
is $\tau$-compact, because $\Psi$ is an isomorphism and therefore $\Psi^{-1}$ is continuous. If $f$ is a $\mu$-density of the (finite) measure $\sigma\in\mathbb S$, then $f\in L_1$ and $E_f\circ\Psi(\phi) = \int_{\Omega} \phi f\,d\mu = F(\phi,\sigma)$. Thus, we see that $\phi\mapsto F(\phi, \sigma)$ is $\tau$-continuous for every $\sigma\in\mathbb S$. It remains to show that $\mathbb T\subseteq \Psi^{-1}(K)$ is $\tau$-closed. But clearly
\begin{align*}
\mathbb T^c &= \left\{ \phi\in L_\infty : a\le \int_{\Omega} \phi f\,d\mu \le b, \forall f\in V_1\right\}^c \\
&=\bigcup_{f\in V_1} \left([E_f\circ\Psi]^{-1}((-\infty,a)) \cup [E_f\circ\Psi]^{-1}((b,\infty))\right) \quad\in\quad \tau.
\end{align*}
\hfill\qed

\subsection{Auxiliary results}

\begin{lemma}\label{lemma:Hdistance}
Consider two measurable spaces $(X,\mathcal X)$ and $(Z,\mathcal Z)$, a Markov kernel $Q:\mathcal Z\times X\to [0,1]$ and two finite measures $\Pr_0$, $\Pr_1$ on $(X,\mathcal X)$. Then $Q \Pr_0$ and $Q \Pr_1$ are finite measures and 
$$
\rho(\Pr_0,\Pr_1) \;\le \; \rho(Q \Pr_0,Q \Pr_1) \quad\text{and}\quad \dH(Q \Pr_0,Q \Pr_1) \;\le\; \dH(\Pr_0,\Pr_1).
$$
\end{lemma}

\begin{proof}
The result is an immediate consequence of Proposition~1.1 in \citet{Del03}, because $\Phi(x,y):=(\sqrt{x}-\sqrt{y})^2$ is convex on $\R_+^2$. For the convenience of the reader, we include a direct proof below.

Finiteness of $Q  \Pr_0$ and $Q  \Pr_1$ is obvious. 
Clearly, the two remaining conclusions are equivalent, because $\dH^2 = 2(1-\rho)$, and we focus on the Hellinger affinities.
Set $\mathbb Q_0 = Q \Pr_0$, $\mathbb Q_1 = Q \Pr_1$, $\mu=\Pr_0+\Pr_1$ and $\nu = \mathbb Q_0 + \mathbb Q_1$, and let $p_0$, $p_1$ and $q_0$, $q_1$ denote the corresponding densities. Consider the Lebesgue decomposition \citep[cf.][Theorem~7.33]{Klenke08} of $\Pr_0$ with respect to $\Pr_1$, i.e.,
$$
\Pr_0 \;=\; \Pr_0^A + \Pr_0^\perp,
$$
where $\Pr_0^A \ll \Pr_1$ and $\Pr_0^\perp \perp \Pr_1$. Clearly, $\Pr_0^A$ and $\Pr_0^\perp$ are absolutely continuous with respect to $\mu$ and we write $p_0^A$ and $p_0^\perp$ for corresponding $\mu$-densities, which satisfy $p_0 = p_0^A + p_0^\perp$. For $D\in\mathcal Z$, define the (finite) measures $\mathbb Q_0^A(D) := \int_X Q(D|x)\,d\Pr_0^A(x)$ and $\mathbb Q_0^\perp(D):= \int_X Q(D|x)\,d\Pr_0^\perp(x)$ and note that $\mathbb Q_0 = \mathbb Q_0^A + \mathbb Q_0^\perp$, so that $\mathbb Q_0^A$ and $\mathbb Q_0^\perp$ are absolutely continuous with respect to $\nu$ and we write $q_0^A$ and $q_0^\perp$ for corresponding $\nu$-densities, which satisfy $q_0 = q_0^A + q_0^\perp$. Now, by singularity of $\Pr_0^\perp$ and $\Pr_1$, there exists a set $S\in\mathcal X$, such that $p_0^\perp$ is $\mu$-a.e. equal to zero on $S$ and $p_1$ is $\mu$-a.e. equal to zero on $S^c$. Therefore,
\begin{align*}
\rho(\Pr_0,\Pr_1) &= \int_X\sqrt{p_0p_1}\,d\mu 
= \int_S \sqrt{p_0^Ap_1 + p_0^\perp p_1}\,d\mu 
+
\int_{S^c} \sqrt{p_0^Ap_1 + p_0^\perp p_1}\,d\mu \\
&= \int_S \sqrt{p_0^Ap_1}\,d\mu \le \int_X \sqrt{p_0^Ap_1}\,d\mu =\rho(\Pr_0^A,\Pr_1).
\end{align*}
On the other hand, we have
$$
\rho(\mathbb Q_0,\mathbb Q_1) = \int_Z\sqrt{q_0q_1}\,d\nu \ge \int_Z\sqrt{q_0^Aq_1}\,d\nu = \rho(\mathbb Q_0^A,\mathbb Q_1).
$$
Thus, it remains to show that $\rho(\Pr_0^A,\Pr_1)\le \rho(\mathbb Q_0^A,\mathbb Q_1)$. To this end, consider a $\Pr_1$-density $\tilde{p}_0$ of $\Pr_0^A$. Clearly, the function $\tilde{p}_1 \equiv 1$ is a $\Pr_1$-density of $\Pr_1$. Thus, we have $\rho(\Pr_0^A,\Pr_1) = \int_X\sqrt{\tilde{p}_0}\,d\Pr_1$ and $\mathbb Q_0^A(D) = \int_X Q(D|x)\tilde{p}_0(x)\,d\Pr_1(x)$, $\mathbb Q_1(D) = \int_X Q(D|x)\,d\Pr_1(x)$, so that $\mathbb Q_0^A \ll \mathbb Q_1$, and we let $\tilde{q}_0$ denote a corresponding $\mathbb Q_1$ density.
Therefore, it remains to show that
\begin{align*}
\int_X\sqrt{\tilde{p}_0}\,d\Pr_1 \;\le\; \int_Z\sqrt{\tilde{q}_0}\,d\mathbb Q_1.
\end{align*}

In fact, we will show slightly more than that. For a Markov kernel $Q:\mathcal Z\times X\to[0,1]$, a finite measure $\Pr$ on $(X,\X)$ and a non-negative function $p\in\mathcal L_1(X,\mathcal X, \Pr)$, we show that
\begin{equation}\label{eq:ShowMore}
\int_X\sqrt{p}\,d\Pr \;\le\; \int_Z\sqrt{q}\,d\mathbb Q,
\end{equation}
where $\mathbb Q := Q\Pr$ dominates the finite measure $\mathbb Q^A(dz):=\int_X Q(dz|x)p(x)d\Pr(x)$ on $(Z,\mathcal Z)$ and $q:Z\to[0,\infty)$ is a corresponding $\mathbb Q$-density.
We establish this fact first for simple functions $p = \sum_{i=1}^n \alpha_i\mathds 1_{A_i}$, where $\alpha_i\in(0,\infty)$ and the $A_1,\dots, A_n \in\X$ are pairwise disjoint.
By disjointness, we easily see that
\begin{align}\label{eq:Simple}
\int_X\sqrt{p}\,d\Pr
=
\int_X\sqrt{\sum_{i=1}^n \alpha_i\mathds 1_{A_i} }\,d\Pr
=
\sum_{i=1}^n \sqrt{\alpha_i} \Pr(A_{i}).
\end{align}
Next, define the measures $\mathbb Q_i^A(dz):= \int_{A_i} Q(dz|x)\,d\Pr(x)$ and note that for any $D\in\mathcal Z$,
\begin{align*}
\int_D q\,d\mathbb Q = \mathbb Q^A(D) = \int_Z Q(D|x)p(x)\,d\Pr(x) = \sum_{i=1}^n \alpha_i \mathbb Q_i^A(D).
\end{align*}
Since $\alpha_i>0$, we have $\mathbb Q_i^A\ll \mathbb Q^A\ll \mathbb Q$, and we write $q_i$ for corresponding finite $\mathbb Q$-densities, so that 
$\int_D q\,d\mathbb Q = \int_D \sum_{i=1}^n \alpha_iq_i\,d\mathbb Q$, for every $D\in\mathcal Z$, which implies that $q = \sum_{i=1}^n \alpha_i q_i$, $\mathbb Q$-almost everywhere.
Now, set $r(z) := \sum_{i=1}^n q_i(z)$ and $R := \{z\in Z: r(z)\in(0,\infty)\}$, and note that for every $D\in\mathcal Z$,
\begin{align*}
\int_D 1 \,d\mathbb Q &= \mathbb Q(D) = \int_X Q(D|x)\,d\Pr \ge \sum_{i=1}^n \int_{A_i} Q(D|x)\,d\Pr 
= \sum_{i=1}^n \int_D q_i\,d\mathbb Q\\
&= \int_D r\,d\mathbb Q, 
\end{align*}
by disjointness of the $A_i$, so that $r\le 1$, $\mathbb Q$-almost everywhere.
Thus, using Jensen's inequality, we obtain the lower bound
\begin{align*}
\int_Z \sqrt{q}\,d\mathbb Q
&\ge
\int_R \sqrt{\sum_{i=1}^{n} \alpha_{i} q_i }\,d\mathbb Q
=
\int_R \sqrt{r} \sqrt{\sum_{i=1}^{n} \alpha_{i} \frac{q_i}{r} }\,d\mathbb Q\\
&\ge
\int_R \sqrt{r} \sum_{i=1}^{n}
\sqrt{\alpha_{i}} \frac{q_i}{r} \,d\mathbb Q
\ge
\int_R \sum_{i=1}^{n}
\sqrt{\alpha_{i}} q_i \,d\mathbb Q
=
\sum_{i=1}^n\sqrt{\alpha_i}\mathbb Q_i^A(R).
\end{align*}
But clearly, $\mathbb Q_i^A(R) = \mathbb Q_i^A(Z) - \mathbb Q_i^A(R^c) = \mathbb Q_i^A(Z) = \Pr(A_i)$. 
So, in view of \eqref{eq:Simple}, we have established \eqref{eq:ShowMore} for simple $p$.
Now, for general $p$, let $(p^{(n)})_{n\in\N}$ be a sequence of simple functions such that $p^{(n)}(z) \uparrow p(z)$. If $q^{(n)}$ denotes a $\mathbb Q$-density of the finite measure $\int_X Q(dz|x)p^{(n)}(x)\,d\Pr(x)$, then it is easy to see that $q^{(n)}\le q$, $\mathbb Q$-almost everywhere. Thus, from \eqref{eq:ShowMore} for simple functions, we conclude that
$$
\int_X \sqrt{p^{(n)}}\,d\Pr \le \int_Z \sqrt{q^{(n)}}\,d\mathbb Q \le \int_Z \sqrt{q}\,d\mathbb Q.
$$
Relation~\eqref{eq:ShowMore} now follows from the monotone convergence theorem.
\end{proof}

\begin{lemma}
\label{lemma:bestL1Rate}
Let $\P$ be a non-empty set of probability measures on some measurable space $(\Omega,\mathcal A)$ and let $\theta:\P\to\R$ be a functional.
If $\P$ is convex and $\theta$ is linear and non-constant, then there exist constants $c_0, c_1>0$, such that $\omega_{TV}(\eps)\ge c_0\eps$ and $\omega_H(\eps)\ge \frac{c_0}{2}\eps^2$, for all $\eps\in[0,c_1]$.
\end{lemma}
\begin{proof}
Choose $\Pr_0,\Pr_1\in\P$, such that $\tilde{c}_0:= |\theta(\Pr_0) - \theta(\Pr_1)|>0$ and $c_1:=\dtv(\Pr_0,\Pr_1)\in(0,1]$. Now, for $\eps\in[0,c_1]$ set $\lambda = \eps/c_1\in[0,1]$ and $\Pr_\lambda = \lambda\Pr_0 + (1-\lambda)\Pr_1\in\P$. Thus, $\dtv(\Pr_\lambda, \Pr_1) = \lambda \dtv(\Pr_0,\Pr_1) = \eps$ and
$$
\frac{\omega_{TV}(\eps)}{\eps} = \frac{\omega_{TV}(\dtv(\Pr_\lambda, \Pr_1))}{\dtv(\Pr_\lambda, \Pr_1)} 
\ge
\frac{|\theta(\Pr_\lambda) - \theta(\Pr_1)|}{\dtv(\Pr_\lambda, \Pr_1)} =\frac{\lambda \tilde{c}_0}{\lambda c_1} = \frac{\tilde{c}_0}{ c_1} =:c_0.
$$
For the second claim, note that $\eps^2/2<\eps$ and
$$
\frac{\omega_{H}(\eps)}{\eps^2/2} \ge \frac{\omega_{TV}(\eps^2/2)}{\eps^2/2} \ge c_0.
$$
\end{proof}

\begin{lemma}\label{lemma:32x}
$$
\forall x\in\R: \quad\left(\frac{3}{2} \right)^x \;\ge\; x.
$$
\end{lemma}
\begin{proof}
For $x\in\R$, we set $f(x) := (3/2)^x-x$ and show that $f(x)\ge0$. Since $f(x) =\exp(x\log(3/2)) - x$, we have $f'(x) = (3/2)^x\log(3/2) - 1$ and $f''(x) = (3/2)^x(\log(3/2))^2>0$. Thus, $f$ is strictly convex and has its unique minimum at 
$$
x_0 = \frac{-\log\log(3/2)}{\log(3/2)},
$$
because $f'(x_0) = \exp(x_0\log(3/2)) \log(3/2) - 1 = 0$. But $1< 3/2\le e$, such that $0< \log(3/2) \le 1$. Since, furthermore, $\log\log(3/2) >-1$, 
$$
f(x_0) = \frac{1}{\log(3/2)} + \frac{\log\log(3/2)}{\log(3/2)} > 0.
$$
\end{proof}

\begin{lemma}[\citet{Krafft67}]\label{lemma:minimaxTest}
Let $\P_0$ and $\P_1$ be two sets of probability measures on a measurable space $(\Omega, \mathcal A)$ and $\mu$ a $\sigma$-finite measure on $(\Omega, \mathcal A)$ that dominates $\P_0\cup\P_1$. Let $\mathbb T$ be the collection of all randomized test functions, i.e., all measurable functions $\phi:\Omega\to[0,1]$. Then there exists a minimax test, i.e., an element $\phi^*\in\mathbb T$, so that
$$
\sup_{\substack{\Pr_0\in\P_0\\\Pr_1\in\P_1}} \E_{\Pr_0}[\phi^*] + \E_{\Pr_1}[1-\phi^*] 
=
\inf_{\phi\in\mathbb T}\sup_{\substack{\Pr_0\in\P_0\\\Pr_1\in\P_1}} \E_{\Pr_0}[\phi] + \E_{\Pr_1}[1-\phi]. 
$$
\end{lemma}

\begin{proof}
Without loss of generality, we may assume that $\P_0$ and $\P_1$ are non-empty, because otherwise any test $\phi\in\mathbb T$ is minimax. For $\phi\in\mathbb T$, $\Pr_0\in\P_0$ and $\Pr_1\in\P_1$, set 
\begin{align*}
\pi(\Pr_0,\Pr_1,\phi) &:=  \E_{\Pr_0}[\phi] + \E_{\Pr_1}[1-\phi], \quad\text{and}\\
R(\phi) &:= \sup_{\substack{\Pr_0\in\P_0\\\Pr_1\in\P_1}} \pi(\Pr_0,\Pr_1,\phi),
\end{align*}
and let $\phi_n\in\mathbb T$ be a sequence of tests, such that $R(\phi_n)\to\inf_{\phi\in\mathbb T} R(\phi) =:\rho$, as $n\to\infty$. From \citet{Nole67}, it follows that there exists a subsequence $(\phi_{n_m})_{m\in\N}$ of $(\phi_n)_{n\in\N}$ and a test $\phi^*\in\mathbb T$, such that
$$
\int_\Omega \phi_{n_m} f \,d\mu \xrightarrow[m\to\infty]{} \int_\Omega \phi^* f\,d\mu,
$$
for every $f\in \mathcal L_1(\Omega, \mathcal A, \mu)$. This entails, in particular, that $\E_{\Pr}[\phi_{n_m}] \to \E_\Pr[\phi^*]$, for every $\Pr\in\P_0\cup\P_1$. Consequently, for every $\Pr_0\in\P_0$ and $\Pr_1\in\P_1$,
$$
\pi(\Pr_0, \Pr_1, \phi^*) = \lim_{m\to\infty} \pi(\Pr_0,\Pr_1,\phi_{n_m}) \le \lim_{m\to\infty}R(\phi_{n_m}) = \lim_{n\to\infty}R(\phi_{n}) =\rho.
$$
But this entails that $R(\phi^*) \le \rho$, whereas $R(\phi^*) \ge \rho$ holds trivially.
\end{proof}

\begin{lemma}\label{lemma:CriticalVals}
For $s,t\in(0,1)$ and $s\ne t$, define
$$
G(s,t) := \frac{\log\left( \frac{1-s}{1-t} \right)}{\log\left( \frac{t}{s} \frac{1-s}{1-t}\right)},
$$
and set $G(s,s)=s$. Then $G:(0,1)^2\to(0,1)$ is strictly increasing in both arguments and $s<G(s,t)<t$, if $s<t$.
\end{lemma}
\begin{proof}
For any $s,t\in(0,1)$, we have $G(s,t)=G(t,s)$. In particular, we see that if $G$ is strictly increasing in its first argument, then it must also be strictly increasing in its second argument. 
Now, for $0<s<t<1$, clearly $(1-s)/(1-t)>1$ and
$$
G(s,t) = 
\frac{\log\left( \frac{1-s}{1-t} \right)}{\log\left( \frac{t}{s}\right) + \log\left(\frac{1-s}{1-t}\right)}
\in(0,1).
$$
For $s\ne t$, write
$$
G(s,t) = \left(1 + \frac{\log\left( \frac{t}{s} \right)}{\log\left( \frac{1-s}{1-t}\right)}\right)^{-1},
$$
and note that $G$ is strictly increasing in its first argument if, and only if,
$$
\frac{\partial}{\partial s} \frac{\log\left( \frac{t}{s} \right)}{\log\left( \frac{1-s}{1-t}\right)}
=
\frac{-1}{s\log\left( \frac{1-s}{1-t}\right)} + \frac{\log \left(\frac{t}{s}\right)}{\left( \log \frac{1-s}{1-t}\right)^2}\frac{1}{1-s} <0.
$$
But this is equivalent to 
$$
s\log \left(\frac{t}{s}\right) < (1-s)\log\left( \frac{1-s}{1-t}\right),
$$
and, in turn, to
\begin{equation}\label{eq:s<G}
s < \frac{\log\left( \frac{1-s}{1-t}\right)}{\log\left( \frac{t}{s}\frac{1-s}{1-t}\right)} = G(s,t),
\end{equation}
provided that $s<t$, and the required monotonicity is equivalent to $G(t,s)=G(s,t)<s$, if $s>t$. Therefore, it remains to show that $s<G(s,t)<t$, for $s<t$.
To that end, write
$$
G(s,t) = 
1 + \frac{\log\left( \frac{s}{t} \right)}{\log\left( \frac{1-s}{1-t}\right) - \log\left(\frac{s}{t}\right)},
$$
and note that $G(s,t)<t$ if, and only if, 
$$
H(t):=\log\left( \frac{s}{t} \right) + (1-t)\left[\log\left( \frac{1-s}{1-t}\right) - \log\left(\frac{s}{t}\right) \right]<0.
$$
Now re-write
$$
H(t) = (1-t)\log\left( \frac{1-s}{1-t}\right) + t\log\left( \frac{s}{t} \right),
$$
and note that $H(t)\to 0$, if $t\downarrow s$. Therefore, $H(t)$ is seen to be strictly negative if $H'(t)<0$. But the derivative evaluates to
\begin{align*}
H'(t) = -\log\left( \frac{1-s}{1-t}\right) + (1-t) \frac{1}{1-t} + \log\left( \frac{s}{t} \right) - t \frac{1}{t}
= \log\left( \frac{s}{t}\frac{1-t}{1-s}\right) <0.
\end{align*}
The argument for $s<G(s,t)$ is analogous and omitted.
\end{proof}

Recall that $\omega_{H}^{(Q_1)}(\eps) = \sup \left\{ |\theta(\Pr_0)-\theta(\Pr_1)| : \dH(Q_1 \Pr_0,Q_1 \Pr_1) \le \eps, \Pr_0,\Pr_1\in\P\right\}$ and note that the following lemma does not require boundedness and linearity of $\theta:\P\to\R$.

\begin{lemma}\label{lemma:Hom2}
Let $k\in\N$, $\eps\in[0,\infty)$ and $Q_1:\mathcal Z\times\X\to[0,1]$ a channel with $\omega_{H}^{(Q_1)}(k\eps) <\infty$. If $Q_1\P$ is convex, then
$$
\omega_{H}^{(Q_1)}(k\eps) \le k^2\omega_{H}^{(Q_1)}(\eps).
$$
\end{lemma}
\begin{proof}
The statement is trivial in case $k=1$. Thus, let $k\ge2$ and fix $\delta\in(0,\infty)$. Since $\omega_{H}^{(Q_1)}(k\eps) = \sup\{|\theta(\Pr_0)-\theta(\Pr_1)|:\dH(Q_1\Pr_0,Q_1\Pr_1)\le k\eps, \Pr_0,\Pr_1\in\P\}$ is finite, there exist $\Pr_0,\Pr_1\in\P$, such that $D:=\dH(Q_1\Pr_0,Q_1\Pr_1)\le k\eps$ and $|\theta(\Pr_0)-\theta(\Pr_1)| + \delta \ge \omega_{H}^{(Q_1)}(k\eps)$. For $\lambda\in[0,1]$ let $\Pr_\lambda = (1-\lambda)\Pr_0 + \lambda \Pr_1$, such that by convexity of $Q_1\P$, we have $Q_1\Pr_\lambda = (1-\lambda)Q_1\Pr_0 + \lambda Q_1\Pr_1 \in Q_1\P$. Note that for $\mu:=Q_1\Pr_0+Q_1\Pr_1$, $q_0 = \frac{dQ_1\Pr_0}{d\mu}$, $q_1=\frac{dQ_1\Pr_1}{d\mu}$, $q_\lambda = \frac{dQ_1\Pr_\lambda}{d\mu} = (1-\lambda)q_0+\lambda q_1$ and because of concavity of the square root, we have
\begin{align}
\dH^2(Q_1\Pr_0,Q_1\Pr_\lambda) &= 2\left(1-\int_{\mathcal Z} \sqrt{q_0[(1-\lambda)q_0+\lambda q_1]}\,d\mu\right)\notag\\
&\le
2\left(1-\int_{\mathcal Z} \sqrt{q_0}[(1-\lambda)\sqrt{q_0}+\lambda \sqrt{q_1}]\,d\mu\right) \notag\\
&=
2\left(1-(1-\lambda) - \lambda\int_{\mathcal Z} \sqrt{q_0q_1}\,d\mu\right) \notag\\
&=
\lambda 2\left(1-\rho_H(Q_1\Pr_0,Q_1\Pr_1)\right) 
=
\lambda \dH^2(Q_1\Pr_0,Q_1\Pr_1).\label{eq:dHhom}
\end{align}
For $j\in\{0,\dots, k^2\}$, choose $\lambda_j := \frac{j}{k^2}$ and note that therefore $\Pr_{\lambda_0}=\Pr_0$ und $\Pr_{\lambda_{k^2}}=\Pr_1$. In view of display~\eqref{eq:dHhom}, we see that $\dH(Q_1\Pr_0,Q_1\Pr_{\lambda_1}) \le \sqrt{\lambda_1} \dH(Q_1\Pr_0,Q_1\Pr_1) \le \eps$. Interchanging the roles of $\Pr_0$ and $\Pr_1$, the same inequality \eqref{eq:dHhom} yields $\dH(Q_1\Pr_{\lambda_j},Q_1\Pr_1) \le \sqrt{1-\lambda_j} D$, for every $j\in\{0,\dots, k^2\}$. For $j\in\{1,\dots, k^2-1\}$, set $\eta_j := \frac{\lambda_{j+1}-\lambda_j}{1-\lambda_j}$, and note that $1-\eta_j = \frac{1-\lambda_{j+1}}{1-\lambda_j}$ and
\begin{align*}
(1-\eta_j) \Pr_{\lambda_j} + \eta_j \Pr_1 &= (1-\eta_j)(1-\lambda_j)\Pr_0 + (1-\eta_j)\lambda_j \Pr_1 + \eta_j \Pr_1 \\
&=
(1-\lambda_{j+1})\Pr_0 + (1-\lambda_j)\eta_j \Pr_1 + \lambda_j \Pr_1\\
&=
(1-\lambda_{j+1})\Pr_0 + \lambda_{j+1} \Pr_1 = \Pr_{\lambda_{j+1}}.
\end{align*}
Using \eqref{eq:dHhom} again, now with $\Pr_{\lambda_j}$ replacing $\Pr_0$ and $\Pr_{\lambda_{j+1}}$ replacing $\Pr_\lambda$, yields 
$$
\dH(Q_1\Pr_{\lambda_j}, Q_1\Pr_{\lambda_{j+1}}) \le \sqrt{\eta_j} \dH(Q_1\Pr_{\lambda_j}, Q_1\Pr_1)
\le
\sqrt{\eta_j}\sqrt{1-\lambda_j} D =\sqrt{\frac{1}{k^2}}D \le \eps.
$$
Hence, for each $j\in\{0,\dots, k^2-1\}$ the probability measures $Q_1\Pr_{\lambda_j}$ und $Q_1\Pr_{\lambda_{j+1}}$ are at a Hellinger distance of at most $\eps$. Therefore, 
\begin{align*}
|\theta(\Pr_0) - \theta(\Pr_1)| = \left|\sum_{j=0}^{k^2-1}\theta(\Pr_{\lambda_j})-\theta(\Pr_{\lambda_{j+1}})\right| \le \sum_{j=0}^{k^2-1}|\theta(\Pr_{\lambda_j})-\theta(\Pr_{\lambda_{j+1}})|
\le k^2\omega_{H}^{(Q)}(\eps).
\end{align*}
Consequently, we obtain $\omega_{H}^{(Q)}(k\eps)\le  k^2\omega_{H}^{(Q)}(\eps) + \delta$. The proof is finished because $\delta>0$ was arbitrary.
\end{proof}

\begin{lemma}\label{lemma:CondC}
Suppose that Condition~\ref{cond:ellh} is satisfied (cf. Section~\ref{sec:Constructive}). Then, for $\bar{r} = \sum_{j=1}^k \frac{s_j}{t_j}$ and $0\le\eps\le h_0^{\left[1+\bar{r}\right](\max_j t_j)}$, we have
$$
\omega_{TV}(\eps) \quad\le\quad 4D_0\eps^{\frac{1}{1+\bar{r}}}.
$$
If, in addition, $B_{\P,\theta}(\ell_h) = 0$ for all $h\in(0,{h}_0]^k$, then, for all $\eps>0$, 
\begin{align*}
\omega_{TV}(\eps) \quad&\le\quad 2D_0 \left(h_0^{-\sum_{j=1}^k s_j}\right) \eps.
\end{align*}
\end{lemma}

\begin{proof}
Set $h^* := h^*(\eps) := (h_1^*,\dots, h_k^*)^T\in\R^k$, where $h_j^* :=  \eps^{\frac{1}{t_j(1+\bar{r})}}$ and note that by assumption, $h_j^*\le h_0\le 1$. Now, for $\dtv(\Pr_0,\Pr_1)\le \eps$,
\begin{align*}
|\theta(\Pr_0)-\theta(\Pr_1)| &\le 
|\E_{\Pr_0}[\ell_{h^*}] - \E_{\Pr_1}[\ell_{h^*}]| + |\theta(\Pr_0) - \E_{\Pr_0}[\ell_{h^*}]| + |\theta(\Pr_1) - \E_{\Pr_1}[\ell_{h^*}]| \\
& 
\le \left|\int_{\mathcal X} \ell_{h^*} \,d(\Pr_0-\Pr_1)\right| +2B_{\P,\theta}(\ell_{h^*}) \\
&\le 
 2\dtv(\Pr_0,\Pr_1) D_0 \prod_{j=1}^k (h_j^*)^{-s_j}  + 2D_0 \frac{1}{k} \sum_{j=1}^k (h_j^*)^{t_j} \\
&\le
 4D_0\eps^{\frac{1}{1+\bar{r}}}.
\end{align*}
The statement about the zero bias case is analogous.
\end{proof}


\section{Improved lower bounds for large $\alpha$}
\label{sec:ApplargeAlpha}

In this section we discuss a way to somewhat improve the lower bound of Section~\ref{sec:ApplowerB} in the case where the privacy level $\alpha>0$ is large and not considered as a fixed constant as sample size $n$ increases. For the sake of brevity, we only consider non-interactive channel distributions $Q$ with identical marginals $Q_1$. Recall the definition of the privatized Hellinger-modulus $\omega_H^{(Q_1)}$ at the beginning of Section~\ref{sec:Attainability}, i.e., 
$$
\omega_H^{(Q_1)}(\eps) = \sup \left\{ |\theta(\Pr_0)-\theta(\Pr_1)| : \dH(Q_1 \Pr_0,Q_1 \Pr_1) \le \eps, \Pr_0,\Pr_1\in\P\right\}.
$$
Using the notation of Section~\ref{sec:ApplowerB}, it is easy to see (cf. the proof of Lemma~\ref{lemma:g1H}) that
$$
\omega_H^{(Q_1)}(g_H(\eta)^-) \;\le\; \Delta_2^{(n)}(Q,\eta),
$$
where 
$$
g_H(\eta) := \inf\{ d_H(\Pr_0, \Pr_1) : \pi(\Pr_0^{\otimes n}, \Pr_1^{\otimes n})\le \eta, \Pr_j\in\P, j=0,1\}.
$$
Now, one can proceed as in the proof of Lemma~3.2 of \citet{Donoho91} and apply the result of Lemma~3.3 in the same reference to show that for all $\eta_0\in(0,1)$ and every $\eps_0>0$, there exists a finite positive constant $c = c(\eta_0,\eps_0)>0$, such that 
$$
g_H(\eta) \;\ge \; c\sqrt{\frac{|\log\eta|}{n}},
$$
for all $\eta\in(0,\eta_0)$ and for all $n>|\log\eta|/\eps_0$. Thus, if $Q_1$ is $\alpha$-private, we obtain from Theorem~\ref{thm:lowerDeltaA} that, for all such $\eta$ and $n$,
$$
\mathcal M_n(Q,\P,\theta) \;\ge \; \frac{\eta}{2} l\left( \frac{1}{4} \omega_H^{(Q_1)}\left( \frac{c}{2}\sqrt{\frac{|\log \eta|}{n}}\right)\right).
$$
It therefore remains to lower bound the privatized Hellinger-modulus.
\begin{lemma}\label{lemma:largeAlpha}
For an arbitrary one-dimensional $\alpha$-differentially private channel $Q_1:\mathcal B(\mathcal Z)\times\mathcal X\to[0,1]$ and every $\varepsilon\ge0$, we have
$$
\omega_{TV}\left( \frac{\varepsilon}{e^{\alpha/2}} \right)\; \le\; \omega_H^{(Q_1)}(\varepsilon).
$$
\end{lemma}
\begin{proof}
Use the elementary relation $\left[2(\sqrt{a}-\sqrt{b}) \sqrt{a\land b} \right]^2 \le [a-b]^2$, for $a, b\ge 0$, and the privacy constraint 
$$
\frac{q(z|x')}{q(z|x^*)} \le e^\alpha, \quad\forall z\in\mathcal Z, x',x^*\in\X,
$$
for a suitable version $q$, 
to obtain
\begin{align*}
&d_H(Q_1P_0, Q_1P_1) = \left[ \int \left( \sqrt{\int q(z|x)p_0(x)\,dx} - \sqrt{\int q(z|x)p_1(x)\,dx}\right)^2 \,dz\right]^{1/2}\\
&\le
\frac{1}{2}\left[ \int \frac{\left(\int q(z|x)[p_0(x)-p_1(x)]\,dx\right)^2}{\int q(z|x)p_0(x)\,dx\land \int q(z|x)p_1(x)\,dx}\,dz \right]^{1/2} \\
&\le
\frac{1}{2}\left[ \int \frac{\int q(z|x)|p_0(x)-p_1(x)|\,dx\cdot \sup_{x'}q(z|x')\int |p_0(x)-p_1(x)|\,dx}{\inf_{x^*} q(z|x^*)}\,dz \right]^{1/2} \\
&\le
\frac{1}{2}
\left[ 
\int \sup_{x',x^*} \frac{q(z|x')}{q(z|x^*)} 
 \int q(z|x)|p_0(x)-p_1(x)|\,dx\;\int |p_0(x)-p_1(x)|\,dx\,dz \right]^{1/2} \\
 &\le
e^{\alpha/2} \frac{1}{2}
 \int |p_0(x)-p_1(x)|\,dx = e^{\alpha/2} d_{TV} (P_0,P_1).
\end{align*}
\end{proof}
Recall that the $\alpha$-dependence of the lower bound in Section~\ref{sec:ApplowerB} was determined by 
$$
\omega_{TV}\left( \frac{1-\eta}{2\sqrt{n}(e^\alpha-1)}\right),
$$
which, for sufficiently large $\alpha$, is smaller than the new bound
$$
\omega_{TV}\left( \frac{c}{2}\frac{\sqrt{|\log \eta|}}{\sqrt{n}e^{\alpha/2}}\right),
$$
that follows from Lemma~\ref{lemma:largeAlpha}. 


\bibliography{arXiv}{}

\end{document}